\documentclass[a4paper]{article}
\usepackage{lineno,hyperref}

\usepackage{graphicx}
\usepackage{epstopdf}
\usepackage{amssymb}
\usepackage{amsmath}
\usepackage{amsfonts}
\usepackage{amsthm} 

\usepackage{lipsum}

\usepackage{mathtools}
\usepackage{a4wide}
\usepackage{caption}
\usepackage{subcaption}
\usepackage{setspace}
\usepackage{xr}
\usepackage{color}
\usepackage{soul}
\usepackage{xcolor}
\usepackage{mathtools}
\usepackage{setspace}
\usepackage{float}
\usepackage{enumerate}

\usepackage{url}
\usepackage[utf8]{inputenc}
\usepackage[T1]{fontenc}
\usepackage{lmodern}
\usepackage[capitalise,noabbrev,nameinlink]{cleveref}

\usepackage{tabularx}
\usepackage{array}
\usepackage{booktabs}
\usepackage{multirow}

\usepackage{eurosym}
\usepackage{textcomp}

\usepackage[affil-it]{authblk}
\usepackage{longtable}

\newtheorem{theorem}{Theorem}[section]

\newtheorem{lemma}[theorem]{Lemma}

\newtheorem{proposition}[theorem]{Proposition}

\theoremstyle{definition}

\begin{document}



\title{Superwind and navigation of least time \\ on Riemannian manifolds}

\author{Nicoleta Aldea$^1$}
\author{Piotr Kopacz$^{2}$}



\affil{\small{$^1$Transilvania University of Bra\c{s}ov, Faculty of Mathematics and Computer Science\\ Iuliu Maniu 50, Bra\c{s}ov, Romania}}
\affil{$^2$Gdynia Maritime University, Faculty of Navigation \\  Al. Jana Paw{\l}a II 3, 81-345 Gdynia, Poland}

\date{}
\date{{\normalsize {\small {e-mail:} \texttt{codruta.aldea@unitbv.ro, p.kopacz@wn.umg.edu.pl}}}}
\maketitle

\begin{abstract} 

In this work, 
 minimum-time navigation on Riemannian manifolds  is analyzed by means of Finsler geometry. 
The introduced notion of a superwind being anisotropic, reducible, and unknown a priori encompasses a wind in the spirit of Zermelo's navigation,  and a gravitational wind in the sense of a slippery slope model, including Matsumoto's navigation on a~mountain slope as the particular cases. 
 The considered generalization also deploys a non-uniform 
 slope, 
where the counterbalance to the transverse 
 impact of superwind causing a lateral 
  drift 
 varies over space. 
This requires 
in particular an extension of the existing theory of   general $(\alpha, \beta)$-metrics, which is presented in our study. The solution of the 
 navigation problem is given by a new Finsler metric and the time-minimizing 
  geodesics, which are determined.  
 Moreover, the thorough analysis establishes the necessary and sufficient conditions for strong convexity under which the resultant velocity defines a Finsler metric.  
   For completeness, a variety of comprehensive and comparative examples are explored in dimension two.  
\end{abstract}


\bigskip \noindent \textbf{MSC 2020}: 53B40, 53C60.

\smallskip \noindent \textbf{Keywords:} Time-minimizing geodesic, Riemann-Finsler manifold, General $(\alpha, \beta)$-metric, \linebreak Matsumoto slope-of-a-mountain problem, Zermelo navigation problem, Slippery slope model,  
 Anisotropic deformation, Time-optimal navigation.   



\tableofcontents 

\section{Introduction}
\label{Intro}

\subsection{Main objectives and background of the study} 

The ultimate 
 goal of this study is to determine the geodesics of a new Finsler metric obtained in this work on a Riemannian manifold $(M, h)$ of arbitrary dimension, where the perturbing vector field is considered more generally than both a wind $W$ in the sense of the Zermelo navigation data $(h, W)$, and the gravity impact 
 as in the slippery slope model, including the Matsumoto navigation on a mountain slope \cite{BRS, matsumoto, slippery}. The desired 
 geodesics correspond to the time-minimizing trajectories that solve the tackled problem of optimal navigation in a purely geometric way. The analyzed generalization involves in particular two new issues, i.e., the concept of a superwind and a non-uniform slope, 
 which enable us to construct the improved model for least time navigation by means of 
 Riemann-Finsler  geometry. 
 Moreover, the study is carried out to establish the necessary and sufficient conditions for strong convexity under which the resultant velocity defines a 
Finsler metric. Since the existing theory for the general $(\alpha, \beta)$-metrics in the literature is not sufficient to solve the addressed task, 
the 
 investigation presented  contains 
 its extension.  
 Furthermore, it also describes an effective anisotropy-based method that allows one to generate 
Finsler metrics, 
 incorporating the standard 
 Zermelo navigation technique, which relies on a rigid translation of a Riemannian indicatrix as the special case.

The initial idea introducing the concept of a {slippery mountain slope} which made a direct link between two classical problems of time-minimizing navigation in Finsler geometry, i.e., Zermelo's navigation problem (ZNP for short; \cite{Zer0,Zer,BRS}) and Matsumoto's slope-of-a-mountain problem (MAT\footnote{In the preceding studies the abbreviation ``MAT'' referred to the original Matsumoto problem which is analyzed under gravity \cite{matsumoto}. In this work, we admit arbitrary wind, 
 as in the Zermelo navigation, however, all the  Matsumoto-like problems, where the perturbation 
is not gravitational,  we also denote by MAT. All of them yield a Finsler metric of Matsumoto type.} for short;  \cite{matsumoto,S-Sabau,Sabau_Tai}) was presented in \cite{slippery}. Next, the 1-parameter exposition was  extended in \cite{cross, slipperyx} and then generalized in \cite{general}, effectively unifying the preceding studies. By then, both above scenarios were considered in the literature 
 as two separate problems, essentially as based on dissimilar constructions of the respective Finslerian indicatrices; see, for example, \cite{baoroblesricci} in this regard. 

Moreover, the slippery slope model made it possible to express the original Matsumoto reasoning 
 in the style of Zermelo's navigation, with the use of a gravitational wind\footnote{A gravitational field reads $\mathbf{G}=\mathbf{G}^{T}+\mathbf{G}^{\perp}$, where the {gravitational wind} $\mathbf{G}^{T}$  is tangent to a mountain slope $M$ and acts in the steepest downhill direction and $\mathbf{G}^{\perp }$ is normal to $M$, considering a 2-dimensional model of the slope.} described  therein. If such a vector field is considered in full, i.e., without resisting its action on the slope\footnote{In the slippery slope model, a transverse gravity-additive pushing a craft or walker to the side (downhill) is restrained due to nonzero traction on the slope which is in general described by a real parameter $\eta\in[0, 1]$ called a {traction coefficient} \cite{slippery}.}, then it 
 corresponds to a particular wind in the Zermelo sense. This particularity, however, was a limitation regarding wind direction and force (i.e., the norm with respect to the background Riemannian metric $h$) that were admitted in the previous investigations. In the current study, we shall drop such a restriction so that the model is complemented in this regard, 
incorporating the Zermelo navigation in the full 
 range. 
 Thus, this will extend the scenarios considered under a gravitational  wind that blows 
 along one fixed direction (steepest descent) 
as in \cite{matsumoto, slippery, slipperyx, general,Nicprw}. 
 Consequently, in contrast to the framework of the Matsumoto slope, the perturbation will now act  in an arbitrary direction and its norm will not 
 depend solely on the slope gradient\footnote{Moreover, physically speaking, the action of the vector field is not only now limited to gravity, but it can have various (arbitrary) physical sources in applied modelling, for example, electric, magnetic, hydrometeorological.
}. So, the landscape resembles the classical 
 Zermelo navigation in the presence of space-dependent wind, where the speed induced by the wind on the craft travelling 
 on a Riemannian sea or slope is always lower than the craft's own speed \cite{SH, BRS, chern_shen}. 
 On the other hand, however, the perturbation in our approach can also be counterbalanced like the gravitational wind, 
 and this essential property is not shared in Zermelo's navigation. More precisely, this means the ability to reduce the transverse component of wind with respect to the direction of motion indicated by the craft's own velocity.  

In order to accomplish the above mentioned task, we begin by 
 introducing the notion of a generalized wind which will encompass as the particluar 
 cases a wind as is known in 
 the Zermelo navigation as well as a reducible gravitational wind covering the original Matsumoto setting applied recently in the slippery slope model \cite{slippery}. 



\subsection{Superwind and non-uniform 
 background 
space}

Let $u$ be own velocity of an imaginary craft or ship proceeding at its maximum speed $||u||_h=1$ on a Riemannian manifold $(M, h)$, where a 
 tangent vector field $W(x)$ has been installed, 
 $x\in M$. Due to the action of wind, in general, the resultant velocity and craft's own velocity are  not collinear and the corresponding speeds differ from each other. Clearly, the resultant speed would be the same in all directions in the absence of perturbation. 
We admit that the impact of the perturbing vector field causing a side drift can be resisted (partially or entirely) as is considered for the gravitational wind in the slippery slope model \cite{slippery}. Now, however, the perturbation will blow in an arbitrary direction and its force needn't be determined by gravity, i.e., the gradient vector field, as with the Zermelo navigation in this regard. It means that a transverse wind component which is orthogonal to $u$ can be compensated, 
 scaling by a real parameter $\eta\in[0, 1]$ and it determines the drifiting (sliding) effect on a Riemannian slope or sea. The parameter is named a \textit{traction coefficient} as in the preceding study \cite{slippery}. In short, the greater the traction (more generally, the ability of counterbalance or resistance to any type of wind), the lesser the drifting. The part which is left (unreduced) has the actual influence upon the trajectory, 
 namely, the further equations of motion, the resultant Finsler metric and in consequence, the behaviour of time-minimizing geodesics. 

 More precisely, the anisotropic 
 perturbation $\mathcal{W}_{\eta}$ given by 
\begin{equation}
\mathcal{W}_{\eta}=(1-\eta)\text{Proj}_{u^{\perp}}W+\text{Proj}_{u}W,    
\label{eq_superwind}
\end{equation}
or equivalently, 
$\mathcal{W}_{{\eta}}=\eta \mathcal{W}_{MAT}+(1-\eta )W,$  $\eta \in[0, 1],$
 will be called a \textit{superwind}, where  $u^\perp$ denotes the orthogonal direction to $u$ and $\mathcal{W}_{MAT}$ stands for $\text{Proj}_{u}W$ as in the original Matsumoto setting , i.e., $\mathcal{W}_1=\mathcal{W}_{MAT}$ \cite{matsumoto}.  Therefore, for all $\eta \in[0, 1]$ we get $||\mathcal{W}_{{\eta}}||_h\in[||\mathcal{W}_{MAT}||_h, ||W||_h]$, where $||\mathcal{W}_{MAT}||_h\in[0, ||W||_h]$. 
 The action of superwind involves in general two geometric transformations, namely, the direction-dependent deformation and rigid translation. They relate to the Riemannian indicatrix ($h$-circle)  we begin with on the way to the final solution of optimal navigation which is to determine geodesics of a (new) Finsler metric on $M$. We recall that the Finslerian geodesics in turn correspond  to the time-minimizing paths. This makes a substantial difference from the Zermelo navigation method which is widely used in Finsler  geometry and physics, where the elliptical $h$-indicatrix is only rigidly translated by the entire wind $W$, while constructing the resultant Randers metric \cite{chern_shen, cr, BRS}. The former  transformation will depend on the angle between $u$ and $W$, no matter their specific directions\footnote{Since $W(x)$ is known (fixed) a priori for all $x\in M$,  the anisotropy actually relates to the direction of $u$ which is unknown initially, keeping in mind that now $W$, acting on the slope, will not have one fixed direction like the gravitational wind $\mathbf{G}^{T}$, i.e.,  steepest downhill along a negative gradient, but will be arbitrary. This implies that the anisotropic active  wind $\mathbf{G}_\eta$ from \cite{slippery} becomes the particular case of the  superwind  $\mathcal{W}_{\eta}$.}. Furthermore, note that the term $\eta \mathcal{W}_{MAT}$ 
 will induce a Finsler metric of Matsumoto type, whereas the term $(1-\eta )W$ will induce a Finsler metric of Randers type, if they are considered separately. 
  We remark that both types belong to a class of $(\alpha, \beta)$-metrics that have been widely investigated are present in the theory and applications of Finsler geometry \cite{baoroblesricci,Kristaly2,B-Miron}.

 Having introduced the notion of superwind, we can recall in different words that the key task 
 in the problem tackled thus far is to find the direction of $u$  (optimal steering) under the action of superwind so that the resultant velocity determines the path which has the minimal Finslerian length 
 between 
a starting point and an endpoint on a Riemannian manifold $(M, h)$. 

 In short, 
 unlike the Zermelo wind, 
 a superwind 
 is in general reducible, 
 anisotropic and unknown a priori. In order to clearly emphasize the difference in the nature of the perturbations,  
 the vector field $W$, arising from the standard navigation data, i.e., in the Zermelo sense, will be called a \textit{common wind}, or alternatively an \textit{ordinary wind}. 
Thus, we have $\mathcal{W}_0=W$, becoming the particular case of the superwind $\mathcal{W}_{\eta}$.  
Bearing in mind the Zermelo navigation framework, one can say that regardless of 
 the direction of motion, the traction coefficient $\eta$ does not influence $W$, since such a wind is not subject to any reducibility. In other words, the common 
 wind 
 does not depend on the $u$-direction at all. 
So, its actual impact upon the ship's velocity is always taken in full. This is in contrast to, for instance, the Matsumoto-type navigation on the slope under gravity. 
 Although non-trivial, yet without anisotropy involved, 
 the Zermelo navigation looks like 
 the simplest case among all time-minimizing $\eta$-navigations, $\eta\in[0, 1]$,     
 where the unit sphere of the resulting Finsler metric is simply the $\mathcal{W}_{0}$-translate of the unit sphere of the background Riemannian metric $h$.

\begin{figure}[h!]
\centering
~\includegraphics[width=0.67\textwidth]{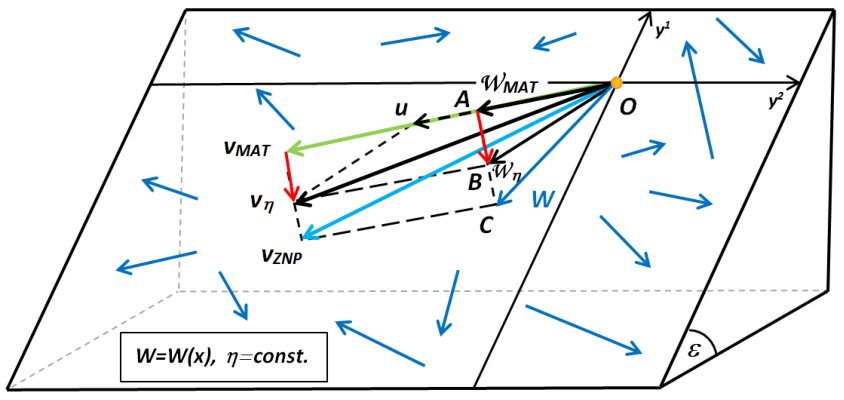}
~\includegraphics[width=0.30\textwidth]{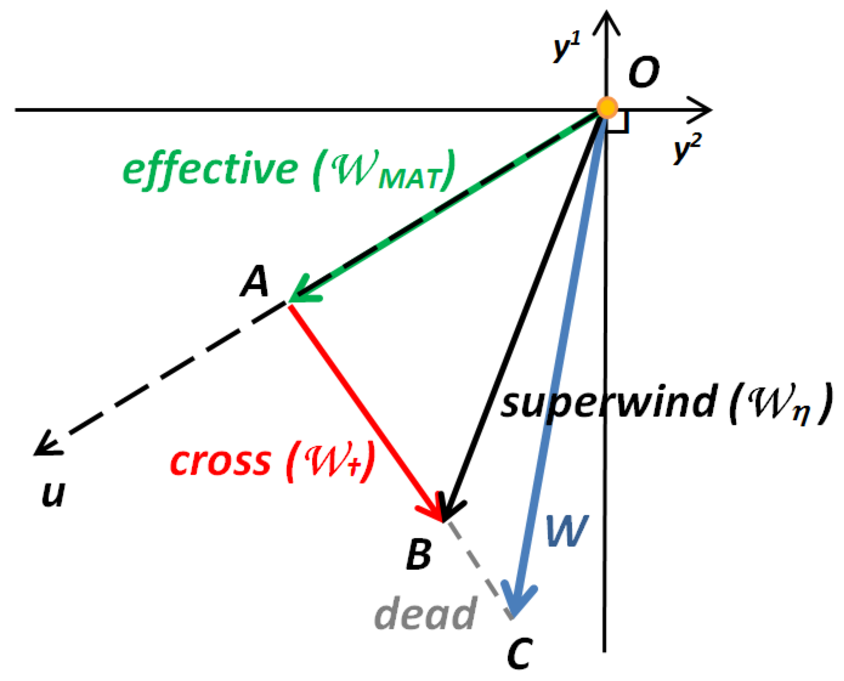} 
\caption{Left: A two-dimensional model being an inclined plane  $M$ of the slope
angle $\varepsilon$ in $\mathbb{R}^{3}$, under the action of the anisotropic superwind $\mathcal{W}_{\eta}$, where the space-dependent ordinary wind $W$ (i.e., in the Zermelo sense; blue) has an arbitraty direction and norm (weak); $O\in M$. The resultant velocity is represented by $v_{\eta}$, including the boundary cases for $\eta\in\{0, 1\}$, i.e., the Zermelo case ($v_{ZNP}$, light blue) and the Matsumoto case ($v_{MAT}$, green), respectively. The transverse (cross) component $\overrightarrow{AB}$ (red) of the superwind wind  $\mathcal{W}_{\eta}=\overrightarrow{OB}$ varies  in particular w.r.t. the traction coefficient $\eta$. 
Right: 
A common wind $W=\mathcal{W}_0$ (blue) is a vector sum of a superwind $\mathcal{W}_{\eta}$  (solid black) and a  dead superwind ($\overrightarrow{BC}$, dashed grey); $OA\perp AC$. The superwind is decomposed into two orthogonal components: an effective superwind (which coincides with  $\mathcal{W}_{MAT}=\mathcal{W}_1$ in the model;  green) and a cross superwind $\mathcal{W}_{\dag}$ (red). The cross impact in general varies, i.e., $||\mathcal{W}_{\dag}||_h\in[0, ||\textnormal{Proj}_{u^{\perp}}W||_h]$, while the along-wind impact is always full, i.e., $||\mathcal{W}_{MAT}||_h$  in the model.    
The Riemannian own velocity of the craft $u$ is shown as a dashed black vector, $||u||_h=1$
}
\label{fig_superwind}
\end{figure}

For comparison, it is worth pointing out that the $h$-indicatrix in MAT ($\eta=1$) is deformed anisotropically and the resultant velocity $v$ and own velocity $u$ are always collinear, however, there is not any rigid translation applied. 
For clarity's sake and simple interpretation, imagine two slippery slopes equipped with the same traction coefficient $\eta$: one under action of gravity expressed by a gravitational wind $\mathbf{G}^{T}$, 
 and the second in the presence of a common wind $W$. 
 In the former case,  the impact of the perturbation 
 is subject to change (reduced) into $\mathbf{G}_{{\eta}}$ (see, e.g., \cite{slippery} for more details), while in the latter, the same wind impact is preserved entirely and continously during navigation on the slope, since traction has no influence on non-reducible $W$ at all. There is a close analogy with the real-world application in marine navigation, where two types of natural perturbations affect a ship's intended route during its voyage at sea, i.e., a usual 
 wind (corresponding to the above $\mathbf{G}_{{\eta}}$-slope) and a water current, e.g., a tidal stream (corresponding to the above $W$-slope); see, for example, \cite{transnav}.  
The stream data, i.e., the so-called set and rate\footnote{In marine navigation, set is the direction of the current, and its rate (drift) is its magnitude (speed); these describe the induced effect of the perturbation on a ship's motion. 
}  are taken entirely as they are given or measured (regardless of the ship's size, heading and speed, direction of current) and such are applied to the triangle of velocities 
 in 
 the navigational calculations done onboard a vessel at sea. Roughly speaking, the offset vector induced on 
 $u$ by the current is considered equal to the ``raw'' current velocity. In contrast, 
 the wind force induced on a ship can be partially resisted (i.e., by compensating the lateral wind component), resulting in, for example, changes in the heel angle\footnote{A rotation around a longitudinal ship axis; if repeated on both sides of the ship, such angular movements are called rolling.} and not only in the course offset (changing the direction of the ship's track). This scenario resembles the slippery slope model, namely, the angle of the ship's heel at sea (resisting the lateral wind drift) corresponds to the traction effect on the hillside (partially counterbalancing the gravity-additive which pushes a walker or craft sideways) \cite{slippery}. For the wind drift (called leeway), the ship's parameters (e.g., size, above-water side surface area, draught) matter substantially in reality, causing different drifts. In the edge case, there may not be any changes in the  heading of the heeled ship (full compensation of wind effect upon $u$), although the wind blows continously, so analogous to the MAT case in our model.

The superwind, encompassing as particular cases the winds in the spirit of both the Zermelo navigation and slippery slope under gravity \cite{SH, BRS, slippery}, gives rise to a more general 
 framework for time-minimizing navigation on Riemannian manifolds in the presence of space-dependent and weak perturbation than the  ones described in the literature 
 thus far.  
In particular, the notion of a traction that initially referred to the scenarios on a mountain  slope under gravity in dimension 2 (\cite{slippery}) has now a broader interpretative (physical) meaning.  
 Namely, it stands for the ability of resistance that opposes the (lateral) motion 
 on any type of background. 
For instance, the traction coefficient $\eta$ 
 resembles and to some extent incorporates  
  the drag coefficient  that is used in fluid dynamics to describe the drag or resistance of an object in a fluid environment, for example  water or air. 
Furthermore, one can also model 
 various external forces that are described and induced by a superwind, i.e., in the sense of, counterbalance,  
 for instance, magnetic, gravitatonal, electric, hydrometeorological; see also \cite{markvorsen,Iran_RWA}. 

The scenarios on  
 the common 
 Euclidean plane, if used in the slope models under gravity, are simplified to the Riemannian case, where the gradient-dependent gravitational wind 
 vanishes.      
This picture is in striking contrast with the non-Riemannian 
 Zermelo navigation, 
 including a variety of nonzero winds that blow tangentially to the plane, however, they are not subject to any counterbalance; 
 cf. \cite{transnav}. 
 The superwind will allow us to cover and unify 
 all the above properties that come from different preceding models (types of navigation).

 Clearly, a superwind can be reduced 
  to some specific winds known from other studies. As an example, one can extract the active gravitational wind $\mathbf{G}_{MAT}$ as in the classical  Matsumoto slope-of-a-mountain problem. This  can be done now twofold. The method used so far, as in \cite{slippery}, is to specialize the common wind $W$ to the full gravitational wind $\mathbf{G}^T$ which corresponds to a specific  $\mathbf{G}_{\eta}$, 
 and then to apply the maximum counterbalance by setting $\eta=1$, which yields $\mathbf{G}_{MAT}$. In short,  $W\leadsto\mathbf{G}^T=\mathbf{G}_{\eta=0}\leadsto\mathbf{G}_{MAT}$. Alternatively, another sequence is first to compensate the transverse component of the superwind $\mathcal{W_\eta}$ entirely by traction ($\eta=1$) which gives  $\mathcal{W}_{MAT}$, and then to specialize its direction and norm, i.e., to be gradient-driven, where its physical source comes from gravity. 
 Briefly, $W=\mathcal{W}_{\eta=0}\leadsto\mathcal{W}_{MAT}\leadsto\mathbf{G}_{MAT}$.


The part of a superwind that is neutralized 
 due to traction, 
 i.e., $W-\mathcal{W}_{\eta}=\eta\text{Proj}_{u^{\perp}}W=\eta (W-\mathcal{W}_{MAT})$ has no impact on the further equations of motion and the resultant Finsler metric. By close analogy with the scenarios considered under a gravitational wind $\mathbf{G}_\eta$ (\cite{slippery, cross, slipperyx, general}), 
 we name it a \textit{dead} superwind.  
 Its norm is a measure of counterbalance (resistance) 
 to the full superwind 
 acting on a Riemannian manifold $(M, h)$. As a consequence, 
 the transverse offset 
 is weaker in comparison with $W$, i.e., it yields $\mathcal{W}_{\eta}$, $\eta\in (0, 1]$. The  
  impact is decreased by the $\eta$-traction if the craft doesn't go parallel to the superwind direction. For example, it vanishes in ZNP, since the corresponding superwind has the strongest force  ($\mathcal{W}_{\eta}=W$), causing a full drift. 
 On the other edge, if $\eta=1$, then the superwind is reduced ($\mathcal{W}_{\eta}=\mathcal{W}_{MAT}$) so that there is no side drift at all. The impact of perturbation is decreased as much as possible because its transverse component with respect to $u$-direction is entirely resisted. Hence, the related dead superwind is then maximal\footnote{For the given direction of motion indicated by $u$ and wind $W(x)$, $x\in M$.}, i.e, $\text{Proj}_{u^{\perp}}W$.

Taking the above into account, the resultant velocity $v_\eta$ reads 
\begin{equation}
v_\eta=u+\mathcal{W}_{\eta},   
\end{equation} 
where $\eta\in [0, 1]$.  In general, the velocities $v_\eta$ and $u$ are not collinear because of the lateral slide or drift.  The collinearity occurs in  the Matsumoto model ($\eta=1$) 
 or when $\eta\in[0, 1]$, while the craft moves exactly with or against the superwind\footnote{Actually, in such particular cases, $\forall \eta\in[0, 1]\ \mathcal{W}_{\eta}=W$, since $\text{Proj}_{u}W=W$, which yields  $h(v_\eta, u)=\pm||v_\eta||_h||u||_h$, so $v_\eta\parallel u$, $||v_\eta||_h=||u||_h\pm||W||_h$.}, i.e., $\mathcal{W}_{\eta}\parallel u$, in which case 
 $\text{Proj}_{u}W=W$. 
In order to maintain consistency with the preceding studies we follow some terminology 
 employed 
 in the descriptions of the slippery slope models \cite{slippery, slipperyx, cross, general}. 
We next decompose the superwind into two components, i.e., the orthogonal projection onto $u$, which is called an \textit{effective} superwind as well as onto the orthogonal direction to $u$, which is called a \textit{cross} superwind. 
The norm of the latter, being subject to compensation, is a linear measure of drift 
 that in general depends on the parameter 
 $\eta$, direction of $u$, and wind force $||W||_h$. 
Thus, for each point $x\in M$ and  direction of motion  
 the vector sum of the maximum 
 lateral (cross)\footnote{In the 1-parameter model analyzed in this research, the effective wind is always maximal for given direction of $u$ and norm $||W||_h$, i.e.,  equal to $\text{Proj}_{u}W=\mathcal{W}_{MAT}$ analogously to the scenario presented in \cite{slippery}. Namely, the longitudinal component of superwind w.r.t. $u$ is not reduced at all, so it acts always in full.} and longitudinal (effective)  components yields the common  wind $W$ as in the Zermelo case. For clarity, see also \cref{fig_superwind} that illustrates a two-dimensional model on the inclined plane. 

The reducibility property\footnote{Alternatively, from another point of view, the craft's capability to counteract the superwind impact.
} of superwind is manifested in the fact that any compensation with $\eta\in(0, 1]$ 
 automatically incorporates anisotropy 
  and influences the resultant velocity $v_\eta$ 
   considerably. It turns out again that the only case (wind) with the $u$-directional independence refers to ZNP. 

%



\begin{figure}[h!]
\centering
~\includegraphics[width=0.7\textwidth]{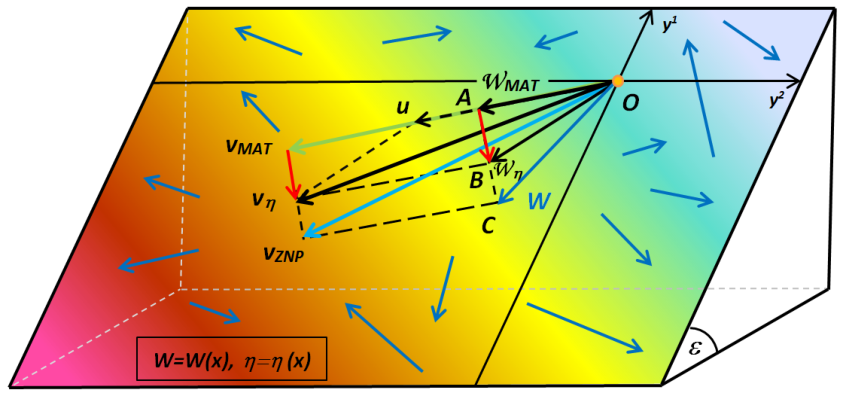} 
\caption{A two-dimensional model of the non-uniform slope being an inclined plane as in \cref{fig_superwind}, with a space-dependent traction coefficient $\eta(x)\in{[0, 1]}$ in addition (marked illustratively by the color-coded slope). The varying traction will influence the behaviour of time-minimizing geodesics and related time fronts; see the specific examples on further reading (\cref{sec_ex2,sec_ex3}).   
}
\label{fig_superwind2}
\end{figure}


The previous model for time-minimizing navigation (\cite{slippery}) is first generalized by the above introduced concept of superwind. The second new issue is to make a scaling factor 
 that varies with respect to a position on $(M, h)$, i.e., a space-dependent traction coefficient $\eta=\eta(x), x\in M$.   
Thus, we shall drop the constraint that traction must be fixed which leads to a more general framework. However, in order to do this, we need to introduce and develop the new Finsler metrics, since the existing results on the general $(\alpha, \beta)$-metrics (e.g., \cite{Yu} applied effectively in \cite{slippery, slipperyx, cross, general}) are insufficient for the current goal. 

The varying traction means that the slope is non-uniform, having different magnitude of resistance 
 to the superwind impact at different points of the Riemannian manifold. This yields changes in the equations of motion and, in consequence, in the behaviour of time geodesics we seek. 
Recalling \eqref{eq_superwind}, we now consider 
\begin{equation}
\mathcal{W}_{{\eta}}=\eta(x) \mathcal{W}_{1}+(1-\eta(x))W,   
\label{eq_superwind2x}
\end{equation}
where  $\eta$ is a $C^{\infty }$-function on $M$, $\eta(x) \in[0, 1]$. To support intuiton and for some interpretation, we can make use of a 2-dimensional model of the generalized slippery slope; see \cref{fig_superwind2} in this regard as well as the thorough examples in \cref{sec_ex2,sec_ex3} on further reading.  
Such a new set-up resembles, for example, a mountain or hill slope covered in different regions by different grounds, e.g., dry, wet, icy or snowy that yield changes in tractions and so, slides (drifts).  The varying $\eta$ will influence the resultant path between two given points on $(M, h)$. To sum up briefly, admitting the traction coefficient to vary over space actually extends the notion of superwind.

Lastly, we remark that in general the perturbing vector field in the analyzed model 
 could also be strong and vary over time, however, this is beyond the scope of this paper. 
 Such a setting would lead in particular to time-dependent Finsler metrics, whose the least time trajectories can be modeled in the ZNP case as lightlike geodesics of an associated Lorentz-Finsler metric \cite{CJS}. 
It is also worth noting that the notion of superwind can be effectively extended with application of another parameter\footnote{Such a parameter was applied effectively in the so-called \textit{slippery-cross-slope}  model and named an \textit{along-traction coefficient} $\tilde{\eta}\in[0, 1]$ \cite{slipperyx}.} instead, scaling its longitudinal component with respect to $u$-direction and not the transverse one,  as shown in \cite{slipperyx}. Moreover, it can be ultimately combined with a 2-parameter model for time-optimal navigation with $(\eta, \tilde{\eta})\in [0, 1] \times [0, 1]$ as presented recently in \cite{general}.  
Looking further ahead, this would be a natural progression of the current work as well as an interesting and fruitful area for future research.


\subsection{
Problem statement 
and main contributions
}

The research problem of interest that we deal with by means of Riemann-Finsler geometry  is formulated as follows

%

%

\begin{quote}
\begin{itshape}
\noindent {Suppose a craft moves at 
 its maximum own 
 speed $||u||_h=1$ 
 on 
 a Riemannian manifold $(M, h)$ under the influence of a superwind $\mathcal{W}_{{\eta}}$ 
 with 
 a space-dependent  traction  coefficient  $\eta(x)\in[0, 1],$ $x\in M$, 
 admitting a 
 drift 
  in an arbitrary 
direction. 
Which path should the craft follow to get from one point to another in the least time?

}
\end{itshape}
\end{quote}


A response to the question will be achieved by a Finsler metric, 
 called the \textit{superwind metric} and denoted by $\tilde{F}_{\eta }$, on account of two key details. The first one refers to the fact that in Finsler 
geometry, the notion of arc length can be interpreted as time and thus, one can 
 regard the time-minimal paths parametrized by arc length as being, locally, the Finsler
geodesics. They are called {time geodesics} in the literature; see, e.g., \cite{matsumoto} in this regard. The second remark is that any time geodesic which corresponds to a superwind metric $\tilde{F}_{\eta }$ has unit $\tilde{F}_{\eta }$-length because it arises as a trajectory in the Zermelo navigation method developed throughout the paper. Consequently, the superwind
metric will play an essential role in the navigation problem under consideration. Our first main result formulated in the spirit of \cite{slippery,cross,slipperyx,general} reads 

\begin{theorem}
\label{Thm1}\textnormal{(Superwind metric)} Let $(M,h)$ be an $n$%
-dimensional Riemannian manifold $(M,h)$, $n>1$, and $\eta :M\rightarrow
\lbrack 0,1]$ be a $C^{\infty }$-function on $M$. Given a vector field $W$
on $M$, the time-minimal paths on $(M,h)$ under the action of a superwind $%
\mathcal{W}_{{\eta }}$ as in \eqref{eq_superwind2x} are the geodesics of a
Finsler metric $\tilde{F}_{\eta }$ which satisfies 
\begin{equation}
\tilde{F}_{\eta }\sqrt{\alpha ^{2}+2(1-\eta )\beta \tilde{F}_{\eta }+(1-\eta
)^{2}||W||_{h}^{2}\tilde{F}_{\eta }^{2}}=\alpha ^{2}+(2-\eta )\beta \tilde{F}%
_{\eta }+(1-\eta )||W||_{h}^{2}\tilde{F}_{\eta }^{2},  \label{TH_mama}
\end{equation}%
with $\alpha =\alpha (x,y),$ $\beta =\beta (x,y)$ given by \eqref{NOT},
where either $||W||_{h}<1$ and $\eta (x)\in \lbrack 0,\frac{1}{2}]$, or $%
||W||_{h}<\frac{1}{2\eta }$ and $\eta (x)\in (\frac{1}{2},1]$, for all $%
x\in M$. In particular, if $\eta (x)=1$ for all $x\in M$, then the~superwind
metric is reduced to a Matsumoto metric, and if $\eta (x)=0$ for all $%
x\in M$, it is reduced to a Randers metric.
\end{theorem}

It is worth mentioning that the above theorem now encompasses as the 
particular cases the solutions to Zermelo's navigation problem on Riemannian manifolds in the presence of an arbitrary mild wind $W$ and Matsumoto's slope-of-a-mountain problem under gravity.  Furthermore, the superwind metric $\tilde{F}_{\eta }$ provides new Finsler metrics which extend the general $(\alpha ,\beta )$-metrics. 

To completely answer 
 to the question posed above, we
are interested in finding the time geodesics of the superwind metric which,
restricted to the indicatrix of $\tilde{F}_{\eta },$ provide the time-minimal
paths on $(M,h)$ under the superwind $\mathcal{W}_{{\eta }}$. 
In this respect, we establish the following result 

\begin{theorem}
\label{Thm2} \textnormal{(Time geodesics) }Let $(M,h)$ be an $n$-dimensional
Riemannian manifold $(M,h)$, $n>1$, and $\eta :M\rightarrow \lbrack 0,1]$ be
a $C^{\infty }$-function on $M$. Given a vector field $W$ on $M$, the
time-minimal paths on $(M,h)$ under the action of a superwind $\mathcal{W}_{{%
\eta }}$ as in \eqref{eq_superwind2x} are the time-parametrized solutions $%
\gamma (t)=(\gamma ^{i}(t)),$ $i=1,...,n$ of the ODE system 
\begin{equation}
\ddot{\gamma}^{i}(t)+2\tilde{\mathcal{G}}_{\eta }^{i}(\gamma (t),\dot{\gamma}%
(t))=0,  \label{GGG}
\end{equation}%
where%
\begin{eqnarray*}
\tilde{\mathcal{G}}_{\eta }^{i}(\gamma (t),\dot{\gamma}(t)) &=&\mathcal{G}%
_{\alpha }^{i}(\gamma (t),\dot{\gamma}(t))+\alpha \tilde{Q}s_{0}^{i}-\alpha
^{2}[\tilde{R}(r^{i}+s^{i})+\frac{1}{2}\tilde{P}\eta ^{i}] \\
&&+\{\mathit{\tilde{\Theta}}[-2\alpha \tilde{Q}s_{0}+r_{00}+\alpha ^{2}(2%
\tilde{R}r+\tilde{P}q)]+\alpha \lbrack \mathit{\tilde{\Omega}}(r_{0}+s_{0})+%
\frac{1}{2}\mathit{\tilde{\Lambda}}\eta _{0}]\}\frac{\dot{\gamma}^{i}(t)}{%
\alpha } \\
&&-\{\mathit{\tilde{\Psi}}[-2\alpha \tilde{Q}s_{0}+r_{00}+\alpha ^{2}(2%
\tilde{R}r+\tilde{P}q)]+\alpha \lbrack \mathit{\tilde{\Pi}}%
r_{0}(r_{0}+s_{0})+\frac{1}{2}\mathit{\tilde{\Upsilon}}\eta _{0}]\}W^{i},
\end{eqnarray*}%
with%
\begin{equation*}
\begin{array}{l}
\mathcal{G}_{\alpha }^{i}(\gamma (t),\dot{\gamma}(t))=\left. \frac{1}{4}%
 \left[ h^{im} \left( 2\frac{\partial h_{jm}}{\partial x^{k}}-\frac{\partial
h_{jk}}{\partial x^{m}}\right) \right] \right \vert _{\gamma (t)}\dot{\gamma}^{j}(t)%
\dot{\gamma}^{k}(t), \\ 
\\ 
\eta _{k}=\left. \frac{\partial \eta }{\partial x^{k}}\right\vert _{\gamma
(t)},\qquad \eta _{0}=\eta _{k} \dot{\gamma}^{k}(t),\qquad \eta ^{i}=\left. h^{ij} \right\vert _{\gamma
(t)}\eta
_{j},\qquad q=-W_{k}\eta ^{k}, \\ 
~ \\ 
r_{00}=-W_{j|k}\dot{\gamma}^{j}(t)\dot{\gamma}^{k}(t),\qquad r_{0}=\frac{1}{2%
}(W_{j|k}+W_{k|j})\dot{\gamma}^{j}(t)W^{k},\qquad
r^{i}+s^{i}=\left. h^{ij} \right\vert _{\gamma
(t)}W_{k|j}W^{k}, \\ 
\\ 
r=-W_{j|k}W^{j}W^{k},\qquad s_{0}=-\frac{1}{2}(W_{j|k}-W_{k|j})\dot{\gamma}%
^{j}(t)W^{k},\qquad s_{0}^{i}=-\frac{1}{2}\left. h^{ij} \right\vert _{\gamma
(t)}(W_{j|k}-W_{k|j})\dot{\gamma}%
^{k}(t), \\ 
~ \\ 
\tilde{R}=\frac{1-\eta }{2\alpha ^{4}\tilde{B}}(\alpha ^{2}\tilde{B}+\eta
),\qquad \tilde{Q}=\frac{\tilde{A}}{\alpha \tilde{B}},\qquad \tilde{P}=\frac{%
1}{2\alpha ^{4}\tilde{B}}[\beta (1-\alpha ^{2}\tilde{B})+(1-2\eta -\alpha
^{2}\tilde{B})||W||_{h}^{2}], \\ 
\\ 
\mathit{\tilde{\Theta}}=\frac{\alpha }{2\tilde{E}}(\alpha ^{6}\tilde{A}%
\tilde{B}^{2}-\eta ^{2}\beta ),\qquad \mathit{\tilde{\Psi}}=\frac{\alpha ^{2}%
}{2\tilde{E}}(\alpha ^{4}\tilde{A}^{2}\tilde{B}+\eta ^{2}), \\ 
~ \\ 
\mathit{\tilde{\Omega}}=\frac{1-\eta }{\alpha ^{2}\tilde{B}\tilde{E}}%
[(\alpha ^{2}\tilde{B}+\eta )(\alpha ^{6}\tilde{B}^{3}+\eta
^{2}||W||_{h}^{2})-\eta ^{2}\alpha ^{2}(\beta \tilde{B}+||W||_{h}^{2}\tilde{A%
})], \\ 
~ \\ 
\mathit{\tilde{\Pi}}=\frac{1-\eta }{2\alpha ^{3}\tilde{B}\tilde{E}}\{(\alpha
^{2}\tilde{B}+\eta )(2\alpha ^{6}\tilde{A}\tilde{B}^{2}-\eta ^{2}\beta
)+\eta ^{2}\alpha ^{2}\tilde{B}[2\alpha ^{2}+(1-\eta )\beta ]\}, \\ 
\\
\mathit{\tilde{\Lambda}}=\frac{1}{\alpha ^{2}\tilde{B}\tilde{E}}[\beta
(1-\alpha ^{2}\tilde{B})+(1-2\eta -\alpha ^{2}\tilde{B})||W||_{h}^{2}](%
\alpha ^{6}\tilde{B}^{3}+\eta ^{2}||W||_{h}^{2}) \\~\\
\qquad -\frac{1}{\tilde{B}\tilde{E}}\{[\frac{1}{2}\alpha ^{4}\tilde{B}%
^{2}(1-\alpha ^{2}\tilde{B})+\eta \lbrack \beta +(1-2\eta
)||W||_{h}^{2}]\}(\beta \tilde{B}+||W||_{h}^{2}\tilde{A}), \\ 
\\ 
\mathit{\tilde{\Upsilon}}=\frac{1}{\alpha ^{3}\tilde{B}\tilde{E}}[\beta
(1-\alpha ^{2}\tilde{B})+(1-2\eta -\alpha ^{2}\tilde{B})||W||_{h}^{2}][%
\alpha ^{6}\tilde{A}\tilde{B}^{2}+\eta (\alpha ^{3}\tilde{C}-\eta \beta )]
\\~\\
\qquad +\frac{\alpha ^{2}}{\tilde{E}}\tilde{C}[\frac{1}{2}\alpha ^{2}\tilde{B%
}(1-\alpha ^{2}\tilde{B})+\eta (\beta +||W||_{h}^{2})], \\ 
~ \\ 
\end{array}
\end{equation*}
\begin{equation}
 \begin{array}{l}
\tilde{A}=-\frac{1}{\alpha ^{2}}\{(1-\eta )\left[ 1-(2-\eta )||W||_{h}^{2}%
\right] -(2-\eta )^{2}\beta -(2-\eta )\alpha ^{2}\}, \\ 
~ \\ 
\tilde{B}=-\frac{1}{\alpha ^{2}}\{[1-2(1-\eta )||W||_{h}^{2}]-2(2-\eta )\beta -2\alpha ^{2}\}, \\ 
~ \\ 
\tilde{C}=\frac{1}{\alpha }\left( \alpha ^{2}\tilde{B}+\beta \tilde{A}%
\right) ,\qquad \tilde{E}=\alpha ^{6}\tilde{B}\tilde{C}^{2}+(||W||_{h}^{2}%
\alpha ^{2}-\beta ^{2})(\alpha ^{4}\tilde{A}^{2}\tilde{B}+\eta ^{2}),%
\end{array}
\label{geo_tilde}
\end{equation}%
$\alpha =\alpha (\gamma (t),\dot{\gamma}(t)),$ $\beta =\beta (\gamma (t),%
\dot{\gamma}(t))$ and the components $W^{i}$ of $W$ and $W_{k}=h_{ik}W^{i}$
are evaluated at $\gamma (t)$.
\end{theorem}

\ 


The paper is structured as follows. First, in \cref{Subsec_2.1} we present some
general notions and important aspects regarding Riemann-Finsler geometry
that we need to prove the main results. Then, in \cref{Subsec_2.2} we introduce
an extension of the general $(\alpha ,\beta )$-metrics that plays
a key role in solving the raised problem on a non-uniform slope under the influence of superwind. This is required by the fact that the counterbalance to the transverse component of the superwind, with respect to the direction of motion $u$, produces a varying lateral drift which is controlled by a $C^{\infty }$-function $\eta :M\rightarrow \lbrack 0,1]$. \cref{Subsec_3.1} is devoted to the proof of \cref{Thm1}, which, in turn, is divided into two steps. The first one performs an anisotropic deformation of the background Riemannian metric $h$ by
the vector field $\eta(x)\mathcal{W}_{MAT}$, the outcome being a Matsumoto metric $F(x,y)=%
\frac{\alpha ^{2}}{\alpha -\eta(x)\beta }$ together with other properties of its indicatrix  (\cref{Lema1,Lema2}). The second step deals with the Zermelo navigation (\cite{CJS,SH}), where  the indicatrix of the Finsler metric $F$ provided in the first step or, in particular, the Riemannian metric $h$ if $%
\eta(x) =0$ for some $x \in M$ is rigidly translated by the scaled vector field $(1-\eta(x) )W,$ assuming that $F(x,-(1-\eta(x) )W)<1.$ In this way, we achieve the superwind metric $\tilde{F}_{\eta }$ as a Finsler metric  being an extension of general $(\alpha, \beta)$-metrics (\cref{Lema3}). In \cref{Subsec_3.2}, as a direct application of the results established in \cref{Subsec_2.2}, we prove \cref{Thm2}. In fact, the proof relies on the implicit expression of the superwind metric together with quite technical results, included in  \cref{PropXX,Lema4,Lema44,Lema5,Prop5}, and the key argument that any geodesic of the superwind metric has unit 
$\tilde{F}_{\eta }$-length because it is a trajectory in the Zermelo navigation technique  developed in the second step. In \cref{Sec_4}, 
we carry out a $2$-dimensional study with a few
examples, comparing the impact of various superwinds' effects  combined with a uniform or a non-uniform slippery slopes, analyzing the behaviour of the corresponding time geodesics and evolution of time fronts.

\section{Preliminaries}

\label{Sec_2}We first briefly collect those notions and general facts from
Riemann-Finsler geometry that are needed for presenting and proving our
aforementioned results; for more details, see, e.g. \cite%
{chern_shen,B-Miron,BRS,SH,Yu,Kristaly,CJS,Musznay,mat2}. Second,
we extend the general $(\alpha ,\beta )$-metrics to a convenient Finsler
metric which serves as a main tool for our study.

\subsection{Finsler manifolds}

\label{Subsec_2.1}Let $M$ be an $n$-dimensional $C^{\infty }$-manifold, with $%
n>1,$ and let $(x^{i}),$ $i=1,...,n$ be the local coordinate system on a
local chart in $x\in M$. Given $T_{x}M$ the tangent space at $x\in M,$ let $%
TM=\underset{x\in M}{\cup }T_{x}M$ be the tangent bundle which is itself a $%
C^{\infty }$-manifold. For every $y\in T_{x}M$, one has $y=y^{i}\frac{%
\partial }{\partial x^{i}},$ where $\left\{ \frac{\partial }{\partial x^{i}}%
\right\} ,$ $i=1,...,n$ $\ $denotes the natural basis and the coordinates on
a local chart in $(x,y)\in TM$ are denoted by $(x^{i},y^{i}),$ $i=1,...,n$.

A natural generalization of a Riemannian metric $h$ on $M$ is a \textit{%
Finsler metric }denoted by $F.$ Actually, the pair $(M,F)$ is a Finsler
manifold if $F:TM\rightarrow \lbrack 0,\infty )$ is a continuous function
with the following properties:

i) $F$ is a $C^{\infty }$-function on the slit tangent bundle $%
TM_{0}=TM\backslash \{0\}$;

ii) $F$ is positively homogeneous of degree one with respect to $y$, i.e., $%
F(x,{c}y)={c}F(x,y)$, for all ${c}>0$;

iii) the Hessian $g_{ij}(x,y)=\frac{1}{2}[F^{2}]_{y^{i}y^{j}},$ where $%
[F^{2}]_{y^{i}y^{j}}=\frac{\partial ^{2}F^{2}}{\partial y^{i}\partial y^{j}}%
, $ is positive definite for all $(x,y)\in TM_{0}.$

\noindent Throughout the paper, we use $I_{F}$ to denote the
indicatrix of $F$, i.e.,%
\begin{equation*}
I_{F}=\left\{ (x,y)\in TM\text{ }|\text{ }F(x,y)=1\right\} .
\end{equation*}%
We notice that the property iii) refers to the fact that $I_{F}$ is strongly
convex. In particular, the Finsler metric $F$ is a Riemannian metric if and
only if $g_{ij}(x,y)$ does not depend on $y,$ i.e., $g_{ij}(x,y)=g_{ij}(x).$
When the domain of $F$ is not the whole $TM$ and the conditions i), ii) and iii) are satisfied only on a conic open subset $\mathcal{A}$ of $TM$  (i.e., if $%
y\in \mathcal{A}_{x},$ then $c y$ $\in \mathcal{A}_{x}$ for every $%
c >0$, where  for each $x\in M,$ 
$\mathcal{A}_{x}=\mathcal{A\ \cap \ }T_{x}M$; see \cite{CJS,JS}), $F$ is called a \textit{conic
Finsler metric}.

A smooth vector field on $TM_{0},$ locally expressed as $S=y^{i}\frac{%
\partial }{\partial x^{i}}-2\mathcal{G}^{i}\frac{\partial }{\partial y^{i}},$
where $\mathcal{G}^{i}=\mathcal{G}^{i}(x,y),$ $i=1,...,n$ are positively\
homogeneous of degree two with respect to $y,$ i.e., $\mathcal{G}^{i}(x,{c}y)=%
{c}^{2}\mathcal{G}^{i}(x,y),$ for all ${c}>0,$ is called a \textit{spray} on 
$M.$  We notice that the functions $\mathcal{G}^{i}$ are called the \textit{%
spray coefficients} \cite{chern_shen}. In the case where the spray is
induced by a Finsler metric ${F}^{2}{=}g_{ij}(x,y)y^{i}y^{j}$, the spray
coefficients read as 
\begin{equation}
\mathcal{G}^{i}(x,y)=\frac{1}{4}g^{il}\{[{F}^{2}]_{x^{k}y^{l}}y^{k}-[{F}%
^{2}]_{x^{l}}\}=\frac{1}{4}g^{il}\left( 2\frac{\partial g_{jl}}{\partial
x^{k}}-\frac{\partial g_{jk}}{\partial x^{l}}\right) y^{j}y^{k},  \label{S1}
\end{equation}%
$(g^{il})$ denoting the inverse matrix of $(g_{il})$.

Given a regular piecewise $C^{\infty }$-curve on $M,$ $\gamma
:[0,1]\rightarrow M$, $\gamma (t)=(\gamma ^{i}(t)),$ $i=1,...,n,$ let $\dot{%
\gamma}(t)=\frac{d\gamma }{dt}$ denote the velocity vector of $\gamma .$ If $%
\dot{\gamma}(t)$ is parallel along $\gamma (t)$, i.e., in the local
coordinates, $\gamma ^{i}(t),$ $i=1,...,n$ are the solutions of the ODE
system%
\begin{equation}
\ddot{\gamma}^{i}(t)+2\mathcal{G}^{i}(\gamma (t),\dot{\gamma}(t))=0,
\label{geo1}
\end{equation}%
then $\gamma $ is called an $F$\textit{-geodesic.}

Zermelo's navigation, apart from the fact that it is a classical optimal
control problem, serves as a method to construct a new Finsler metric, by perturbing a
given Finsler metric (being called the background metric), making use of a vector field $W,$
i.e., a time-independent wind on a manifold $M$, under some constraints. As
already emphasized, by considering that the background metric is a
Riemannian one, the Randers metric solves Zermelo's problem
of navigation in the case of a weak wind $W,$ i.e., $||W||_{h}<1$ \cite%
{BRS,chern_shen}. When $W$ is a critical wind, i.e., $||W||_{h}=1,$
the problem is solved by the Kropina metric \cite{Y-Sabau}. Throughout the
paper, we need the following result (see \cite[Lemma 3.1]{SH}, \cite[%
Proposition 2.14]{CJS}, \cite[Lemma 1.4.1]{chern_shen}).

\begin{proposition}
\label{Prop3} Let $(M,F)$ be a Finsler manifold and $W$ a vector field on $M$
such that $F(x,-W)<1$. Then the solution of the Zermelo's navigation problem
with the navigation data $(F,W)$ is a Finsler metric $\tilde{F}$ obtained by
solving the equation%
\begin{equation}
F(x,y-\tilde{F}(x,y)W)=\tilde{F}(x,y),\text{ }  \label{MAIN}
\end{equation}%
for any nonzero $y\in T_{x}M$, $x\in M.$
\end{proposition}

\noindent It is worth pointing out that \eqref{MAIN} admits a unique
positive solution $\tilde{F}$ for any nonzero $y\in T_{x}M,$  due to the
fact that the indicatrix $I_{F}$ is strongly convex and the inequality $F(x,-W)<1$ is assumed   \cite{CJS,SH}. Moreover, $F(x,-W)<1$ also assures the
fact that $\tilde{F}$ is a Finsler metric, having the indicatrix $I_{\tilde{F%
}}=\left\{ (x,y)\in TM \text{ \ }|\text{ }\tilde{F}(x,y)=1\right\} $
strongly convex, as well as, for any $x\in M$, $y=0$ belongs to the region
bounded by $I_{\tilde{F}}$; for more details, see \cite{CJS}. Over and above
that, any regular piecewise $C^{\infty }$-curve $\gamma :[0,1]\rightarrow M$%
, parametrized by time, that represents a trajectory in Zermelo's navigation
problem, has unit $\tilde{F}$-length, i.e., $\tilde{F}(\gamma (t),\dot{\gamma}%
(t))=1,$ \cite[Lemma 1.4.1]{chern_shen}.

\newpage

\subsection{An extension of the general $(\protect\alpha ,\protect\beta )$%
-metrics}

\label{Subsec_2.2}Let $\alpha ^{2}=a_{ij}(x)y^{i}y^{j}$ be a quadratic form,
where $a_{ij}(x)$ is a Riemannian metric on $M$. For further use, let $%
\alpha $ be called a Riemannian metric. Given a differential $1$-form $\beta
=b_{i}(x)dx^{i}$ on $M,$ simply used for $\beta =b_{i}y^{i}$, a Finsler
metric $F$ is called \textit{general} $(\alpha ,\beta )$-\textit{metric }if
it can be written as $F=\alpha \phi (b^{2},s)$, where $\phi (b^{2},s)$ is a
positive $C^{\infty }$-function in the variables $b^{2}=||\beta ||_{\alpha
}^{2}=a^{ij}b_{i}b_{j}$ and $s=\frac{\beta }{\alpha }$, with $|s|\leq
b<b_{0} $ and $0<b_{0}\leq \infty $ (see \cite{Yu}). Various examples of
general $(\alpha ,\beta )$-metrics are provided by the \textit{slippery
slope metrics} with one or two parameters, which have been described
recently in \cite{slippery,cross,slipperyx,general}. In particular, if $\phi $
depends only on the variable $s$, the function $F=\alpha \phi (s)$ is a $%
(\alpha ,\beta )$-\textit{metric}, often exemplified by the Randers metric $%
F=\alpha +\beta ,$ with $\phi (s)=1+s,$ which carries out the solution to
Zermelo's navigation problem under the influence of a weak wind, i.e., $%
|s|\leq b<1$ \cite{chern_shen}, or by the Matsumoto metric $F=\frac{\alpha
^{2}}{\alpha -\beta },$ with $\phi (s)={\frac{1}{1-s}}$ and $|s|\leq b<{%
\frac{1}{2}},$ which solves Matsumoto's slope-of-a-mountain problem \cite%
{matsumoto}.

For the purpose of the present paper, a general $(\alpha ,\beta )$-metric $%
F=\alpha \phi (b^{2},s)$ cannot serve as a toolkit because of the additional
variable $\eta =\eta (x),$ $x\in M$, provided by a non-constant traction,
which is assumed in the proposed navigation problem, denoted by $\mathcal{P}_{\eta }$. This
requires an extension of $F=\alpha \phi (b^{2},s)$ to a Finsler metric $%
F=\alpha \phi (b^{2},s,\eta ),$ with $\phi (b^{2},s,\eta )$ a positive $%
C^{\infty }$-function in the variables $b^{2},$ $s=\frac{\beta }{\alpha }$
and $\eta $, assuming that $\eta :M\rightarrow \lbrack 0,1]$ is a $%
C^{\infty }$-function on $M.$ Although the variables $b^{2}$ and $\eta $
both depend on $x\in M,$ we take them separately, being more convenient for
computations. Several geometric properties can be studied for this extension. However, we focus only
on some basic results regarding $F=\alpha \phi (b^{2},s,\eta )$  which include the necessary and sufficient conditions for being a Finsler metric as well as an
explicit formula for the spray coefficients which corresponds to  $F=\alpha \phi (b^{2},s,\eta )$.

In order to do this, let $\phi _{1},$ $\phi _{2}$ and $\phi _{3}$ be the
differentiations of the function $\phi $ with respect to the first variable $%
b^{2},$ the second variable $s$ and the third variable $\eta $,
respectively. Similarly, throughout the paper $\phi _{12},$ $\phi _{22}$ and $\phi _{32}$ correspondingly denote the differentiations of $\phi _{1},$ $\phi _{2}$ and $\phi _{3}$ with
respect to $s$. To present the desired results, we
also need to recall and fix the following notations and properties: 
\begin{equation}
\begin{array}{l}
r_{ij}=\frac{1}{2}(b_{i|j}+b_{j|i}),\quad r_{i}=b^{j}r_{ji},\quad
r^{i}=a^{ij}r_{j},\quad r_{00}=r_{ij}y^{i}y^{j},\quad r_{0}=r_{i}y^{i},\quad
r=b^{i}r_{i},\quad \\ 
~ \\ 
s_{ij}=\frac{1}{2}(b_{i|j}-b_{j|i}),\quad s_{i}=b^{j}s_{ji},\quad
s^{i}=a^{ij}s_{j},\quad s_{0}^{i}=a^{ij}s_{jk}y^{k},\quad s_{0}=s_{i}y^{i},
\\ 
\\ 
\lbrack \alpha ^{2}]_{y^{i}}=y_{i}=a_{ij}y^{j},\text{ \ \ }%
[b^{2}]_{x^{i}}=2(r_{i}+s_{i}),\text{ \ \ \ }s_{y^{i}}=\alpha ^{-2}(\alpha
b_{i}-sy_{i}), \\ 
\\ 
\eta _{x^{i}}=\eta _{i},\text{ \ \ \ }\eta _{0}=\eta _{i}y^{i},\text{ \ \ }%
\eta ^{i}=a^{ij}\eta _{j},\text{ \ \ \ }q=b^{i}\eta _{i},%
\end{array}
\label{rs}
\end{equation}%
with $\eta _{x^{i}}=\frac{\partial \eta }{\partial x^{i}},$ $b^{j}=a^{ji}b_{i}$, $b_{i|j}=\frac{\partial b_{i}}{\partial x^{j}}%
-\Gamma _{ij}^{k}b_{k}$ and $\Gamma _{ij}^{k}=\frac{1}{2}a^{km}\left( \frac{%
\partial a_{jm}}{\partial x^{i}}+\frac{\partial a_{im}}{\partial x^{j}}-%
\frac{\partial a_{ij}}{\partial x^{m}}\right)$ being the Christoffel
symbols of the Riemannian metric $a_{ij}$, where $(a^{ij})=(a_{ij})^{-1}.$ We
point out that the differential $1$-form $\beta $ is closed if and only if $%
s_{ij}=0$ (see \cite{chern_shen}).

Moreover, since $b^{2}$ and $\eta $ do not depend on $y,$ the expression for
the Hessian $g_{ij}(x,y)=\frac{1}{2}[F^{2}]_{y^{i}y^{j}}$ and some direct
computations and properties concerning $F=\alpha \phi (b^{2},s,\eta )$ can
be extracted from \cite{Yu,chern_shen}. More precisely, owing to \cite[%
Propositions 3.2 and 3.3]{Yu}, one can similarly prove the following two results:

\begin{proposition}
\label{Prop0}Let $M$ be an $n$-dimensional $C^{\infty }$-manifold, $n>1$%
. The aforementioned function $F=\alpha \phi (b^{2},s,\eta )$ holds the
following statements:
\begin{itemize}
\item[i)]  $g_{ij}=\rho a_{ij}+\rho _{0}b_{i}b_{j}+\rho _{1}(b_{i}\alpha
_{y^{j}}+b_{j}\alpha _{y^{i}})-s\rho _{1}\alpha _{y^{i}}\alpha _{y^{j}}$,
where $\rho =\phi (\phi -s\phi _{2}),$  $\rho _{0}=\phi \phi _{22}+(\phi
_{2})^{2}$\\~\\  and $\rho _{1}=(\phi -s\phi _{2})\phi _{2}-s\phi \phi _{22};$

\item[ii)] $\det (g_{ij})=\phi ^{n+1}(\phi -s\phi _{2})^{n-2}[\phi -s\phi
_{2}+(b^{2}-s^{2})\phi _{22}]\det (a_{ij});$

\item[iii)] $g^{ij}=\rho ^{-1}[a^{ij}+\tau b^{i}b^{j}+\alpha
^{-1}\tau _{0}(b^{i}y^{j}+b^{j}y^{i})+\alpha ^{-2}\tau _{1}y^{i}y^{j}],$ if $\det
(g_{ij})\neq 0,$ where 
\begin{equation*}
\tau =\frac{-\phi _{22}}{\phi -s\phi _{2}+(b^{2}-s^{2})\phi _{22}},\text{ \ }%
\tau _{0}=\frac{-[(\phi -s\phi _{2})\phi _{2}-s\phi \phi _{22}]}{\phi \lbrack
\phi -s\phi _{2}+(b^{2}-s^{2})\phi _{22}]},\text{ \ }\tau _{1}=\frac{-[s\phi
+(b^{2}-s^{2})\phi _{2}]}{\phi }\tau _{0}.
\end{equation*}
\end{itemize}
\end{proposition}

\begin{proposition}
\label{Prop1}Let $M$ be an $n$-dimensional $C^{\infty }$-manifold, $n>1$%
. The aforementioned function $F=\alpha \phi (b^{2},s,\eta )$ is a Finsler metric for any
Riemannian metric $\alpha ,$ differential $1$-form $\beta $ with $||\beta
||_{\alpha }<b_{0}$ and $C^{\infty }$-function $\eta :M\rightarrow \lbrack
0,1]$ if and only if $\phi =\phi (b^{2},s,\eta )$ is a positive $C^{\infty }$%
-function supporting the following inequalities:%
\begin{equation*}
\phi -s\phi _{2}>0,\text{ \ \ }\phi -s\phi _{2}+(b^{2}-s^{2})\phi _{22}>0,
\end{equation*}%
when $n\geq 3$ or%
\begin{equation*}
\phi -s\phi _{2}+(b^{2}-s^{2})\phi _{22}>0,
\end{equation*}%
when $n=2$, while $s=\frac{\beta }{\alpha }$ and $b=||\beta ||_{\alpha }$
hold $|s|\leq b<b_{0}$.
\end{proposition}

\begin{proposition}
\label{Prop2}Let $M$ be an $n$-dimensional $C^{\infty }$-manifold, $n>1.$ The relationship between the spray coefficients $\mathcal{G}%
^{i} $ of the Finsler metric $F=\alpha \phi (b^{2},s,\eta )$ and $\mathcal{G}%
_{\alpha }^{i}$ $=\frac{1}{2}\Gamma _{jk}^{i}y^{j}y^{k}$ of $\alpha $ reads
as follows 
\begin{eqnarray*}
\mathcal{G}^{i}\ \ &=&\ \mathcal{G}_{\alpha }^{i}+\alpha Qs_{0}^{i}+\{\Theta
\lbrack -2\alpha Qs_{0}+r_{00}+\alpha ^{2}(2Rr+Pq)]+\alpha \lbrack \Omega
(r_{0}+s_{0})+\frac{1}{2}\Lambda \eta _{0}]\}\frac{y^{i}}{\alpha } \\
&&+\{\Psi \lbrack -2\alpha Qs_{0}+r_{00}+\alpha ^{2}(2Rr+Pq)]+\alpha \lbrack
\Pi (r_{0}+s_{0})+\frac{1}{2}\Upsilon \eta _{0}\}b^{i} \\
&&-\alpha ^{2}[R(r^{i}+s^{i})+\frac{1}{2}P\eta ^{i}],
\end{eqnarray*}%
where 
\begin{eqnarray*}
R\ \ &=&\ \frac{\phi _{1}}{\phi -s\phi _{2}},\text{ \ \ \ \  \ \ }Q\ \ =\ \frac{\phi
_{2}}{\phi -s\phi _{2}},\text{ \ \ \ \ \ \ \ }P\ \ =\ \frac{\phi _{3}}{\phi -s\phi _{2}%
}, \\
\Theta \ \ &=&\ \ \frac{(\phi -s\phi _{2})\phi _{2}-s\phi \phi _{22}}{2\phi
\lbrack \phi -s\phi _{2}+(b^{2}-s^{2})\phi _{22}]}, \; \; \text{\ \ \ \ \ \ \ \ \ \ }\Psi \ \ =\ 
\frac{\phi _{22}}{2[\phi -s\phi _{2}+(b^{2}-s^{2})\phi _{22}]}, \\
\Pi \ \ &=&\ \ \frac{(\phi -s\phi _{2})\phi _{12}-s\phi _{1}\phi _{22}}{%
(\phi -s\phi _{2})[\phi -s\phi _{2}+(b^{2}-s^{2})\phi _{22}]},\text{ \ \ }%
\Omega \ \ =\ \frac{2\phi _{1}}{\phi }-\frac{s\phi +(b^{2}-s^{2})\phi _{2}}{%
\phi }\Pi , \\
\Upsilon \ \ &=&\ \ \frac{(\phi -s\phi _{2})\phi _{32}-s\phi _{3}\phi _{22}}{%
(\phi -s\phi _{2})[\phi -s\phi _{2}+(b^{2}-s^{2})\phi _{22}]},\text{ \ \ }%
\Lambda \ \ =\ \frac{2\phi _{3}}{\phi }-\frac{s\phi +(b^{2}-s^{2})\phi _{2}}{%
\phi }\Upsilon .
\end{eqnarray*}
\end{proposition}

\begin{proof} The proof partially relies on some technical computations which
are done in \cite[Propositions 3.4]{Yu}. For completeness, we provide a few
details. Here, we have in addition the dependence of $\phi $ on the third
variable $\eta (x).$ On account of \eqref{S1} and following the
differentiations%
\begin{equation*}
\begin{array}{l}
\lbrack {F}^{2}]_{x^{k}}\ \ =\ [\alpha ^{2}]_{x^{k}}\phi ^{2}+2\alpha
^{2}\phi \{\phi _{1}[b^{2}]_{x^{k}}+\phi
_{2}s_{x^{k}}+\phi _{3}\eta _{x^{k}}\},\quad \\ 
~ \\ 
\lbrack {F}^{2}]_{x^{k}y^{l}}\ =\ [\alpha ^{2}]_{x^{k}y^{l}}\phi
^{2}+2[\alpha ^{2}]_{x^{k}}\phi \phi _{2}s_{y^{l}}+2[\alpha
^{2}]_{y^{l}}\phi \phi _{1}[b^{2}]_{x^{k}}+2\alpha ^{2}(\phi _{1}\phi
_{2}+\phi \phi
_{12})[b^{2}]_{x^{k}}s_{y^{l}} \\~\\
\ \text{\ \ \ \ \ \ \ \ \ \ \ \ \ }+2[\alpha ^{2}]_{y^{l}}\phi \phi
_{2}s_{x^{k}}+2\alpha ^{2}[(\phi _{2})^{2}+\phi
\phi _{22}]s_{x^{k}}s_{y^{l}}+2\alpha ^{2}\phi \phi _{2}s_{x^{k}y^{l}} \\~\\
\text{ \ \ \ \ \ \ \ \ \ \ \ \ \ }+2[\alpha ^{2}]_{y^{l}}\phi \phi _{3}\eta
_{x^{k}}+2\alpha ^{2}(\phi _{3}\phi _{2}+\phi \phi _{32})\eta _{x^{k}}s_{y^{l}},%
\end{array}%
\end{equation*}%
let us observe that $\mathcal{G}^{i}$ can be expressed as%
\begin{equation*}
\mathcal{G}^{i}=\mathcal{G}_{1}^{i}+\mathcal{G}_{2}^{i}+\mathcal{G}_{3}^{i},
\end{equation*}%
where each term $\mathcal{G}_{1}^{i}$ and $\mathcal{G}_{3}^{i}$ includes at
least a differentiation of $\phi $ with respect to the first and the third
variable, respectively, the term $\mathcal{G}_{2}^{i}$ comprising only $\phi
_{2}$ and $\phi _{22}.$ Precisely, they are 
\begin{eqnarray*}
\mathcal{G}_{1}^{i}\ \ &=&\ \frac{1}{2}g^{il}\{[\alpha ^{2}]_{y^{l}}\phi
\phi _{1}[b^{2}]_{x^{k}}y^{k}+\alpha ^{2}\phi _{1}\phi
_{2}[b^{2}]_{x^{k}}y^{k}s_{y^{l}}+\alpha ^{2}\phi \phi
_{12}[b^{2}]_{x^{k}}y^{k}s_{y^{l}}-\alpha ^{2}\phi \phi
_{1}[b^{2}]_{x^{l}}\}, \\
\mathcal{G}_{2}^{i}\ \ &=&\ \frac{1}{4}g^{il}\{[\alpha
^{2}]_{x^{k}y^{l}}y^{k}\phi ^{2}+2[\alpha ^{2}]_{x^{k}}y^{k}\phi \phi
_{2}s_{y^{l}}+2[\alpha ^{2}]_{y^{l}}\phi \phi _{2}s_{x^{k}}y^{k}+2\alpha
^{2}(\phi _{2})^{2}s_{x^{k}}s_{y^{l}}y^{k} \\
&&+2\alpha ^{2}\phi \phi _{22}s_{x^{k}}s_{y^{l}}y^{k}+2\alpha ^{2}\phi \phi
_{2}s_{x^{k}y^{l}}y^{k}-[\alpha ^{2}]_{x^{l}}\phi ^{2}-2\alpha ^{2}\phi \phi
_{2}s_{x^{l}}\}, \\
\mathcal{G}_{3}^{i}\ \ &=&\ \frac{1}{2}g^{il}\{[\alpha ^{2}]_{y^{l}}\phi
\phi _{3}\eta _{x^{k}}y^{k}+\alpha ^{2}\phi _{3}\phi _{2}\eta
_{x^{k}}y^{k}s_{y^{l}}+\alpha ^{2}\phi \phi _{32}\eta
_{x^{k}}y^{k}s_{y^{l}}-\alpha ^{2}\phi \phi _{3}\eta _{x^{l}}\},
\end{eqnarray*}%
where $\mathcal{G}_{1}^{i}$ and $\mathcal{G}_{2}^{i}$ can be taken from \cite%
[Propositions 3.4]{Yu}, i.e.,%
\begin{eqnarray}
\mathcal{G}_{1}^{i} &=&[2\alpha ^{2}\Theta Rr+\alpha \Omega (r_{0}+s_{0})]%
\frac{y^{i}}{\alpha }+[2\alpha ^{2}\Psi Rr+\alpha \Pi
(r_{0}+s_{0})]b^{i}-\alpha ^{2}R(r^{i}+s^{i}),  \label{G12} \\
\mathcal{G}_{2}^{i} &=&\mathcal{G}_{\alpha }^{i}+\alpha Qs_{0}^{i}+\Theta
(-2\alpha Qs_{0}+r_{00})\frac{y^{i}}{\alpha }+\Psi (-2\alpha
Qs_{0}+r_{00})b^{i}.  \notag
\end{eqnarray}%
Thus, it remains only to compute explicitly the term $\mathcal{G}_{3}^{i}.$
Making use of \eqref{rs} and \cref{Prop0}, one has that%
\begin{equation}
\mathcal{G}_{3}^{i}\ =\ g^{il}(Uy_{l}+Vb_{l}+X\eta _{l})=\rho
^{-1}(Zy^{i}+Yb^{i}+X\eta ^{i}),  \label{3G}
\end{equation}%
where%
\begin{equation*}
U\ \ =\ \frac{1}{2}[2\phi \phi _{3}-s(\phi _{3}\phi _{2}+\phi \phi
_{32})]\eta _{0},\text{ \ \ }V\ \ =\ \frac{1}{2}\alpha (\phi _{3}\phi
_{2}+\phi \phi _{32})\eta _{0},\text{ \ \ }X\ \ =\ -\frac{1}{2}\alpha
^{2}\phi \phi _{3},
\end{equation*}%
and implicitly for \eqref{3G}  to be identically checked, it yields%
\begin{eqnarray*}
Z\ \ &=&\ U+(Us+V\alpha ^{-1}b^{2}+X\alpha ^{-1}q)\tau _{0}+(U+V\alpha
^{-1}s+X\alpha ^{-2}\eta _{0})\tau _{1}, \\
Y\ \ &=&\ V+\alpha \lbrack (Us+V\alpha ^{-1}b^{2}+X\alpha ^{-1}q)\tau
+(U+V\alpha ^{-1}s+X\alpha ^{-2}\eta _{0})\tau _{0}].
\end{eqnarray*}%
Next, an elementary computation shows that 
\begin{equation*}
\rho ^{-1}Z\ =\ \frac{1}{2}\Lambda \eta _{0}+\alpha \Theta Pq,\text{ \ \ \ }%
\rho ^{-1}Y\ =\ \alpha (\frac{1}{2}\Upsilon \eta _{0}+\alpha \Psi Pq),\text{
\ \ \ }\rho ^{-1}X\ =\ -\frac{1}{2}\alpha ^{2}P,
\end{equation*}%
and thus, by \eqref{3G} it turns out that%
\begin{equation*}
\mathcal{G}_{3}^{i}\ =(\frac{1}{2}\alpha \Lambda \eta _{0}+\alpha ^{2}\Theta
Pq)\ \frac{y^{i}}{\alpha }+\alpha (\frac{1}{2}\Upsilon \eta _{0}+\alpha \Psi
Pq)b^{i}-\frac{1}{2}\alpha ^{2}P\eta ^{i}.
\end{equation*}%
The claim follows by the latter relation and \eqref{G12}.
\end{proof}

\section{Proofs of main theorems}

\label{Sec_3}

\subsection{Superwind metric: proof of \cref{Thm1}}

\label{Subsec_3.1}In the sequel, we prove \cref{Thm1}, supplying the superwind metric which
collects all requirements for a Finsler one. To end this, some background
preparations will be fixed. Given an $n$-dimensional Riemannian manifold $%
(M,h)$, $n>1,$ let $u$ and $W$ be two vector fields on $M.$ The vector field 
$W$ is given and it is commonly referred to as a wind, whereas $u$ stands for the self-velocity of a moving imaginary
craft on $M,$ assuming throughout this section that $||u||_{h}=1.$ Posing the navigation problem $\mathcal{P}_{\eta }$ on $(M,h)$ under the
action of the superwind $\mathcal{W}_{{\eta }}=\eta \mathcal{W}%
_{MAT}+(1-\eta )W,$ the resultant velocity is $v_{\eta }=u+\mathcal{W}_{{%
\eta }},$ for any $C^{\infty }$-function $\eta :M\rightarrow \lbrack 0,1].$
Since $\mathcal{W}_{MAT}$ denotes the projection of the vector field $W$ on $u$, it
follows that $||\mathcal{W}_{{\eta }}||_{h}\leq ||W||_{h},$ for any $%
C^{\infty }$-function $\eta =\eta (x)\in \lbrack 0,1]$. For convenience, in what follows, we will use to write simple $\eta$ instead of  $\eta(x)$.

Ostensibly, $\mathcal{P}_{\eta }$ looks like a standard Zermelo navigation
problem, where the solution is given by a Finsler metric of Randers type if
the wind is weak, intensively studied in the literature \cite%
{BRS,chern_shen,CJS}. However, it is quite complicated
because of the feature of the superwind $\mathcal{W}_{{\eta }}$ which is not
a priori known, underlining that only the wind $W$ is given, the vector field $%
\mathcal{W}_{MAT}$ depends on the direction of $u.$ Consequently, $\mathcal{W%
}_{{\eta }}$ depends on $u$ and thus, a key ingredient that we use to carry
out this dependence is, as a first step, to take into consideration the
effect of the vector field $\eta \mathcal{W}_{MAT},$ which refers to the
deformation of the background Riemannian metric $h$ by $\eta \mathcal{W}%
_{MAT}$, i.e., an anisotropic deformation. Then, to perform the proof of %
\cref{Thm1}, the classical Zermelo navigation is applied and developed in the second step,
where the indicatrix of the Finsler metric $F$ of Matsumoto type, achieved by
the first step, is rigidly translated by the scaled vector field $(1-\eta )W,$  under the condition $F(x,-(1-\eta )W)<1,$ which practically secures that
the moving of the imaginary craft is possible in any direction and also the
uniqueness of the superwind metric (see \cite{SH,CJS}). Therefore, we devide the proof into two steps.

\paragraph{Step I: anisotropic deformation}

This step states that the direction-dependent deformation of the Riemannian
metric $h$ by $\eta \mathcal{W}_{MAT}$ provides a Finsler metric if and only
if $||W||_{h}<\frac{1}{2\eta },$ which requires the condition $\eta ||%
\mathcal{W}_{MAT}||_{h}<1,$ for any $C^{\infty }$-function $\eta
:M\rightarrow (0,1].$ In the complementary case when $\eta ||\mathcal{W}%
_{MAT}||_{h} \geq 1$ (possible only at some directions), the deformation cannot
yield a Finsler metric.

In order to achieve these, we describe the deformation of $h$ by the vector
field $\eta \mathcal{W}_{MAT}$ in terms of the resultant velocity 
\begin{equation}
v=u+\eta \mathcal{W}_{MAT},  \label{res}
\end{equation}%
for any $C^{\infty }$-function $\eta :M\rightarrow (0,1].$ We note that
whenever $\eta (x)=0$ for some $x\in M$, then $v=u$ and thus, there does not exist
any deformation of $h.$

Since $\mathcal{W}_{MAT}$ stands for the projection of the vector field $W$
on $u,$ the vectors $u$ and $v$ are collinear. Moreover, 
\begin{equation}
\mathcal{W}_{MAT}=\frac{h(v,W)}{||v||_{h}^{2}}v,  \label{proj}
\end{equation}%
when the resultant velocity $v$ is nonzero (i.e., $u\neq -\eta \mathcal{W}_{MAT}$) and consequently one has that $||\mathcal{W}_{MAT}||_{h}=\frac{|h(v,W)|}{||v||_{h}}%
\neq \frac{1}{\eta }.$ The last condition will be splited below into two
cases: $\eta ||\mathcal{W}_{MAT}||_{h}<1$ and $\eta ||\mathcal{W}%
_{MAT}||_{h}>1$ which will be  analyzed separately. It is worth mentioning
that $v$ can vanish only when $u=-\eta \mathcal{W}_{MAT},$ which implies $||%
\mathcal{W}_{MAT}||_{h}=\frac{1}{\eta }.$ However, the case $||\mathcal{%
W}_{MAT}||_{h}=\frac{1}{\eta }$ will also be taken into consideration at the
end of this step, since it does not lead only to $v=0.$ Indeed, if $||%
\mathcal{W}_{MAT}||_{h}=\frac{1}{\eta }$ and $v\neq 0$, then $u=\eta \mathcal{%
W}_{MAT}$ and $v=2\eta \mathcal{W}_{MAT}$.

In view of \eqref{res} and \eqref{proj}, a direct computation together with the equality $u=[1-\eta \frac{h(v,W)}{%
||v||_{h}^{2}}]v$ leads to the equation 
\begin{equation}
\left( ||v||_{h}-\eta \frac{h(v,W)}{||v||_{h}}\right) ^{2}=1,  \label{E1}
\end{equation}%
if $\frac{|h(v,W)|}{||v||_{h}}\neq \frac{1}{\eta }.$

In addition, one may introduce the notation $g_{1}(x,v)=\left( ||v||_{h}-\eta \frac{h(v,W)}{%
||v||_{h}}\right) ^{2}-1$.  Then, \eqref{E1} can be written in the equivalent form $g_{1}(x,v)=0$ and, supported by Okubo's method (\cite{matsumoto}), we will be able to get the function $F(x,v)$ as solutions of the equation $g_{1}(x,\frac{v}{F})=0$, i.e.,%
\begin{equation*}
\left( \frac{||v||_{h}^{2}}{\eta h(v,W)+||v||_{h}}-F(x,v)\right) \left( 
\frac{||v||_{h}^{2}}{\eta h(v,W)-||v||_{h}}-F(x,v)\right) =0,
\end{equation*}%
with $\frac{|h(v,W)|}{||v||_{h}}\neq \frac{1}{\eta }.$ Moreover, we can
extend $F(x,v)$ to an arbitrary nonzero vector $y\in T_{x}M,$ for any $x\in
M $ because any nonzero $y$ can be expressed as $y=cv,$ $c>0,$ and $F(x,v)=1$%
. Namely, it turns out the positive homogeneous $C^{\infty }$-function $%
F(x,y)$ which holds%
\begin{equation}
\left( \frac{||y||_{h}^{2}}{\eta h(y,W)+||y||_{h}}-F(x,y)\right) \left( 
\frac{||y||_{h}^{2}}{\eta h(y,W)-||y||_{h}}-F(x,y)\right) =0,  \label{sol}
\end{equation}%
with $\frac{|h(y,W)|}{||y||_{h}}\neq \frac{1}{\eta },$ for any $C^{\infty }$%
-function $\eta :M\rightarrow (0,1].$

There is still a certain amount of properties regarding the
function $F(x,y)$ which can be explicitly provided by the solutions of the
equation \eqref{sol}. To emphasize this, we need to study the following
cases: 1. $\frac{|h(y,W)|}{||y||_{h}}<\frac{1}{\eta }$ and 2. $\frac{|h(y,W)|%
}{||y||_{h}}>\frac{1}{\eta }.$

\noindent Case 1. We assume that $\frac{|h(y,W)|}{||y||_{h}}<\frac{1}{\eta },$
which actually means $\eta ||\mathcal{W}_{MAT}||_{h}<1$. We first claim that
this is a necessary and sufficient condition that the equation \eqref{sol}
admits the following unique positive root 
\begin{equation}
F(x,y)=\frac{||y||_{h}^{2}}{||y||_{h}+\eta h(y,W)}\text{ },  \label{eta_mat}
\end{equation}%
for any $C^{\infty }$-function $\eta :M\rightarrow (0,1],$ on all $TM_{0}.$

In view of the result, it seems to be reasonable to observe that the
positivity of \eqref{eta_mat} on $TM_{0}$ means that 
\begin{equation}
||y||_{h}+\eta h(y,W)>0,  \label{ineq}
\end{equation}%
for all nonzero $y$ and any $C^{\infty }$-function $\eta :M\rightarrow
(0,1]. $ If the positivity is achieved on $TM_{0},$ we can replace $y$ with $%
-W\neq 0$ in \eqref{ineq} and thus, it turns out that $\eta ||W||_{h}<1.$
Together with the inequality $||\mathcal{W}_{MAT}||_{h}\leq ||W||_{h}$, we
conclude that $\eta ||\mathcal{W}_{MAT}||_{h}<1$ on all $TM_{0}.$
Conversely, if $\eta ||\mathcal{W}_{MAT}||_{h}<1$ on all $TM_{0}$ (the case $%
\mathcal{W}_{MAT}=0$ is also included), it follows that $\frac{|h(y,\mathbf{G%
}^{T})|}{||y||_{h}}<\frac{1}{\eta }$ for any nonzero $y,$ which gives %
\eqref{ineq}. Thus, the assertion that the function $F(x,y)$ provided by %
\eqref{eta_mat} is positive on $TM_{0}$ is proved.

From now on, we fix the following notations: 
\begin{equation}
\alpha ^{2}=||y||_{h}^{2}=h_{ij}y^{i}y^{j}\text{ \ \ \ \ and \ \ }\beta
=-h(y,W)=b_{i}y^{i},  \label{NOT}
\end{equation}%
where $\alpha =\alpha (x,y),$ $\beta =\beta (x,y),$ $W=W^{i}\frac{\partial }{%
\partial x^{i}},$ $b_{i}=-h_{ij}W^{j}=-W_{i}$ and $b=||\beta
||_{h}=||W||_{h}.$ Based on these, we can express the function %
\eqref{eta_mat} as 
\begin{equation}
F(x,y)=\frac{\alpha ^{2}}{\alpha -\eta \beta },\text{ }
\label{Matsumoto_eta}
\end{equation}%
for any $C^{\infty }$-function $\eta :M\rightarrow (0,1].$ Obviously, it is
of Matsumoto type with the indicatrix%
\begin{equation}
I_{F}=\left\{ (x,y)\in TM_{0}\text{ }|\text{ }\alpha ^{2}(\alpha -\eta \beta
)^{-1}=1\right\} \subset TM.  \label{indicatrix}
\end{equation}%
Moreover, since $y=0$ does not lie in the closure of the
indicatrix $I_{F}$  in $TM_{0}$, we can extend $F(x,y)$ continuously to all $TM,$ i.e., $%
F(x,0)=0$ for any $x\in M$ (see \cite{CJS}). Therefore, the function %
\eqref{Matsumoto_eta} seems to be a promising Finsler metric. In order to
confirm this, we aim to establish the necessary and sufficient conditions
for the strong convexity of the indicatrix $I_{F},$ for any $C^{\infty }$%
-function $\eta :M\rightarrow (0,1]$.

Let us write $F(x,y)=\alpha \phi (s,\eta ),$ where $\phi (s,\eta)=\frac{1}{1-\eta
s}\,$\ with $s=\frac{\beta }{\alpha }$ and $\eta=\eta(x).$ In the following, we collect a few
desired properties for $\phi (s,\eta )$ and also we describe the force of
the wind $W$ via the variable $s.$ Since the assumed condition $\frac{%
|h(y,W)|}{||y||_{h}}<\frac{1}{\eta }$ is actually $|s|<\frac{1}{\eta }$, for
arbitrary nonzero $y\in T_{x}M$ and $x\in M$, \ it turns out that $\phi $ is
a positive $C^{\infty }$-function on the open interval $\mathcal{I}=\left( -%
\frac{1}{\eta },\frac{1}{\eta }\right) $, for any $C^{\infty }$-function $%
\eta :M\rightarrow (0,1].$

\begin{lemma}
\label{Lema1} Given $\phi (s,\eta )=\frac{1}{1-\eta s}$ with $s\in \mathcal{I%
}$, for any $C^{\infty }$-function $\eta :M\rightarrow (0,1]$, the following
statements are equivalent:

\begin{itemize}
\item[i)] $\phi (s,\eta )-s\phi _{2}(s,\eta )+(b^{2}-s^{2})\phi _{22}(s,\eta
)>0$, where $b=||W||_{h}$; $\phi _{2}$ and $\phi _{22}$ denote the
differentiations of $\phi $ and $\phi _{2}$ with respect to $s,$ respectively;

\item[ii)] $|s|\leq b<b_{0}$, where $b_{0}=\frac{1}{2\eta }$;

\item[iii)] $||W||_{h}<\frac{1}{2\eta }.$
\end{itemize}
\end{lemma}

\begin{proof}
In view of the Cauchy-Schwarz inequality $|h(y,W)|\leq ||y||_{h}||W||_{h},$
clearly follows that $|s|\leq ||W||_{h}=b,$ for any nonzero $y\in T_{x}M$ and $%
x\in M.$

We first establish a lower bound for $(b^{2}-s^{2})\phi _{22}(s,\eta )$ and
then, we evaluate the expression $\xi =\phi (s,\eta )-s\phi _{2}(s,\eta
)+(b^{2}-s^{2})\phi _{22}(s,\eta ).$ Indeed, since $|s|<\frac{1}{\eta }$,
one has that 
\begin{equation*}
(b^{2}-s^{2})\phi _{22}(s,\eta )=(b^{2}-s^{2})\frac{2\eta ^{2}%
}{(1-\eta s)^{3}}\geq 0
\end{equation*}
Moreover, this yields that the minimum value of $%
(b^{2}-s^{2})\phi _{22}(s,\eta )$ is $0$ and it is achieved when $|s|=b,$
for any $C^{\infty }$-function $\eta :M\rightarrow (0,1]$. By a simple
computation we get for $\xi$ the following expression
\begin{equation}
\xi =\frac{(1-\eta s)(1-2\eta s)+2(b^{2}-s^{2})\eta ^{2}}{(1-\eta s)^{3}}.
\label{IEQ}
\end{equation}

To prove the implication i) $\Rightarrow $ ii), we assume that $\xi >0.$
Taking $s=b$ in \eqref{IEQ}, we arrive at $1-2\eta b>0$, and then $b<\frac{1%
}{2\eta }.$ Hence, we conclude that $|s|\leq b<\frac{1}{2\eta }$ which is
the required ii). Conversely, assuming that $|s|\leq b<\frac{1}{2\eta }$ and
using \eqref{IEQ}, we obtain%
\begin{equation*}
\xi \geq \frac{(1-\eta s)(1-2\eta s)}{(1-\eta s)^{3}}=\frac{1-2\eta s}{%
(1-\eta s)^{3}}~>~0,
\end{equation*}%
for any $C^{\infty }$-function $\eta :M\rightarrow (0,1]$.

Now we prove iii) $\Rightarrow $ ii). Since\textbf{\ }$||W||_{h}<\frac{1}{%
2\eta }$ and $|s|\leq ||W||_{h}=b$, it follows the inequality $|s|\leq b<%
\frac{1}{2\eta }.$ The converse implication ii) $\Rightarrow $ iii) is
obvious. 
\end{proof}

Note that that the statement $|s|\leq b<\frac{1}{2\eta }$ also implies that
for any $C^{\infty }$-function $\eta :M\rightarrow (0,1]$, $\phi (s,\eta
)-s\phi _{2}(s,\eta )>0$. In view of the above results and applying %
\cref{Prop1}, we have stated the following result

\begin{lemma}
\label{Lema2}For any $C^{\infty }$-function $\eta :M\rightarrow (0,1],$ the Matsumoto type function  $F(x,y)=\frac{\alpha ^{2}}{\alpha -\eta \beta }$ is a Finsler metric if and
only if $\ ||W||_{h}<\frac{1}{2\eta }.$
\end{lemma}

\noindent Therefore, by \cref{Lema2}, we conclude that the
indicatrix $I_{F}$ is strongly convex if and only if $||W||_{h}<\frac{1}{%
2\eta }$, for any $C^{\infty }$-function $\eta :M\rightarrow (0,1].$

\medskip \noindent  Case 2. Now, we deal with the second possibility $\frac{|h(y,W)|}{%
||y||_{h}}>\frac{1}{\eta }.$ Since this is equivalent to the inequality $%
\eta ||\mathcal{W}_{MAT}||_{h}>1$, it turns out that $||W||_{h}>\frac{1}{%
\eta }.$ Moreover, in this case the equation \eqref{sol} admits two positive
solutions%
\begin{equation}
F_{1,2}(x,y)=\frac{||y||_{h}^{2}}{\pm ||y||_{h}+\eta h(y,W)},  \label{F12}
\end{equation}%
only when $y\in \mathcal{A}_{x}^{\ast }=\mathcal{A}^{\ast }\mathcal{\ \cap \ 
}T_{x}M,$ for any $x\in M,$ where%
\begin{equation*}
\mathcal{A}^{\ast }=\{(x,y)\in TM\text{ }|\text{ }\ ||y||_{h}-\eta h(y,W)<0\}
\end{equation*}%
is an open conic subset of $TM_{0},$ for any $C^{\infty }$-function $\eta
:M\rightarrow (0,1].$ It is worthwhile to mention that the inequality $\frac{%
h(y,W)}{||y||_{h}}>\frac{1}{\eta }$ is a necessary and sufficient condition
for $F_{1,2}(x,y)$ to be positive on $\mathcal{A}^{\ast }.$ When $\frac{%
h(y,W)}{||y||_{h}}<-\frac{1}{\eta },$ both solutions $F_{1,2}(x,y)$ are
negative.

Now, using the notations \eqref{NOT}, the expressions of the functions $%
F_{1,2}(x,y)$ are of Matsumoto type 
\begin{equation}
F_{1}(x,y)=\frac{\alpha ^{2}}{\alpha -\eta \beta }\text{ \ \ and \ \ }%
F_{2}(x,y)=-\frac{\alpha ^{2}}{\alpha +\eta \beta }  \label{Fstrong}
\end{equation}%
on the conic domain $\mathcal{A}^{\ast }$, rewritten as $\mathcal{A}^{\ast
}=\{(x,y)\in TM$ $|$ $\alpha +\eta \beta <0\}$. Actually $\mathcal{A}^{\ast
} $ is the interior of the set which includes all the rays starting at $0$
and crossing the indicatrices $I_{F_{1}}$ and $I_{F_{2}}$. By applying \cite[%
Corollary 4.15]{JS}, it turns out that both $F_{1,2}$ are strongly convex on 
$\mathcal{A}^{\ast }$ and thus, each of them is a conic Finsler metric on $%
\mathcal{A}^{\ast },$ for any $C^{\infty }$-function $\eta :M\rightarrow
(0,1].$ Indeed, for $F_{1,2}$ the strong convexity conditions $(\alpha \mp
2\eta \beta )(\alpha \mp \eta \beta )>0$ are satisfied for any $(x,y)\in 
\mathcal{A}^{\ast }$ and $C^{\infty }$-function $\eta :M\rightarrow (0,1].$
However, the closures 
of the indicatrix $I_{F_{1}}$ and $I_{F_{2}}\cup \{{0\}}$ do not glue at
their intersection with the boundary of $\mathcal{A}^{\ast }.$

One more case should be mentioned here, namely, when $\frac{|h(y,W)|}{||y||_{h}}=\frac{1}{\eta }$ for nonzero $%
y.$ It may occur when $\eta ||\mathcal{W}_{MAT}||_{h}=1$ and consequently, by
Okubo's method \cite{matsumoto}, we get the function $F(x,y)=\frac{1}{2}%
||y||_{h}$ with $y\in \mathcal{A}_{x}=\mathcal{A\ \cap \ }T_{x}M,$ for any $%
x\in M$, where by the set $\mathcal{A}$ we mean $\mathcal{A}=\{(x,y)~\in ~TM_{0}\ |\ ||y||_{h}-\eta
|h(y,W)|=0\}$. One sees immediately that  this case does not provide any Finsler
metric.\noindent

Therefore, we summarize the results obtained in this step, emphasizing that
the direction-dependent deformation of the background Riemannian metric $h$
by the vector field $\eta \mathcal{W}_{MAT}$, with $\eta ||\mathcal{W}%
_{MAT}||_{h}<1$ performed for any direction and for any $C^{\infty }$%
-function $\eta :M\rightarrow (0,1],$ supplies the Finsler metric $F(x,y)=%
\frac{\alpha ^{2}}{\alpha -\eta \beta }$ if and only if $||W||_{h}<\frac{1}{%
2\eta }$.

\paragraph{Step II: rigid translation}

By this step, we complete the proof of \cref{Thm1}. It remains to take into
consideration the effect produced by the addition of the scaled wind $%
(1-\eta )W$ which actually, generates a rigid translation to the strongly
convex indicatrix \eqref{indicatrix} provided by the equation of motion $%
v=u+\eta \mathcal{W}_{MAT}$ in the first step when $\eta ||\mathcal{W}%
_{MAT}||_{h}<1$. Essential to our arguments is \cref{Prop3}. More precisely,
we have to explore the Zermelo navigation on the Finsler manifold $(M,F)$ with
the navigation data $(F,(1-\eta )W)$, for any $C^{\infty }$-function $\eta
:M\rightarrow \lbrack 0,1]$, assuming the condition 
\begin{equation}
F(x,-(1-\eta )W)<1,  \label{CC}
\end{equation}%
where $F$ is either the Finsler metric \eqref{eta_mat}, having $||W||_{h}<%
\frac{1}{2\eta }\leq 1$, if $\eta (x)\in (0,1]$ for any $x\in M$, or the
background Riemannian metric $h,$ when $\eta (x)=0$ for some $x\in M.$ The
outcome will be the superwind metric, which arises as the unique
positive solution of the equation%
\begin{equation}
F(x,y-(1-\eta )\tilde{F}(x,y)W)=\tilde{F}(x,y),\text{ }  \label{II}
\end{equation}%
for any $(x,y)\in TM_{0}.$ In particular, if $\eta (x)=1$, for some $x\in M,$
we have the solution $\tilde{F}=\frac{\alpha ^{2}}{\alpha -\beta }$, which is actually
the Finsler metric $F$ from \eqref{eta_mat} with $\eta (x)=1.$ It is
worthwhile to mention that the requirement \eqref{CC} assures the following 
issues: the uniqueness of the superwind metric, the fact that its indicatrix
is strongly convex as well as the detail that $y=0$ belongs to the region
bounded by this indicatrix, for any $x\in M$ (see \cite[p. 10 and
Proposition 2.14]{CJS}).

From now on, by the Finsler metric $F$ which is involved in \eqref{II} and %
\eqref{CC}, we mean the function $F(x,y)=\frac{\alpha ^{2}}{\alpha -\eta \beta },$ for
any $C^{\infty }$-function $\eta :M\rightarrow \lbrack 0,1],$ including also
the above-mentioned particular possibilities for $\eta (x).$ In order to reach
the superwind metric as a Finsler metric, we first expand the left hand-side
of \eqref{II}. By substituting $y$ with $y-(1-\eta )\tilde{F}(x,y)W$ in %
\eqref{NOT} and making use of some algebraic manipulations, it follows that 
\begin{equation*}
\alpha ^{2}\left( x,y-(1-\eta )\tilde{F}(x,y)W\right) =\alpha
^{2}(x,y)+2(1-\eta )\beta (x,y)\tilde{F}(x,y)+(1-\eta )^{2}||W||_{h}^{2}%
\tilde{F}^{2}(x,y)
\end{equation*}%
and%
\begin{equation*}
\beta \left( x,y-(1-\eta )\tilde{F}(x,y)W\right) =\beta (x,y)+(1-\eta
)||W||_{h}^{2}\tilde{F}(x,y),
\end{equation*}%
where $\beta (x,W)=-||W||_{h}^{2}.$ Together with \eqref{II}, we get the
following irrational equation%
\begin{equation}
\tilde{F}\sqrt{\alpha ^{2}+2(1-\eta )\beta \tilde{F}+(1-\eta
)^{2}||W||_{h}^{2}\tilde{F}^{2}}=\alpha ^{2}+(2-\eta )\beta \tilde{F}%
+(1-\eta )||W||_{h}^{2}\tilde{F}^{2},  \label{MAMA_general}
\end{equation}%
for any $C^{\infty }$-function $\eta :M\rightarrow \lbrack 0,1],$ which can
be easily written as a polynomial equation 
\begin{equation}
\begin{array}{c}
\left( 1-\eta \right) ^{2}||W||_{h}^{2}(1-||W||_{h}^{2})\tilde{F}%
^{4}+2\left( 1-\eta \right) [1-\left( 2-\eta \right) ||W||_{h}^{2}]\beta 
\tilde{F}^{3} \\ 
~ \\ 
+\{[1-2\left( 1-\eta \right) ||W||_{h}^{2}]\alpha ^{2}-\left( 2-\eta \right)
^{2}\beta ^{2}\}\tilde{F}^{2}-2\left( 2-\eta \right) \alpha ^{2}\beta \tilde{%
F}-\alpha ^{4}=0.%
\end{array}
\label{MAMA_4}
\end{equation}%
Note that $\alpha $, $\beta $ and $\tilde{F}$ in \eqref{MAMA_general} and %
\eqref{MAMA_4} are evaluated at $(x,y)$. Moreover, if $\left( 1-\eta \right)
^{2}(1-||W||_{h}^{2})\neq 0$ then the equation \eqref{MAMA_4} admits four
roots. However, due to \eqref{CC}, among all roots of \eqref{MAMA_4}, there
is only a sole positive root, which actually must be the superwind metric.
From now on, for each $C^{\infty }$-function $\eta :M\rightarrow \lbrack
0,1],$ we denote by $\tilde{F}_{\eta }$ $\ $the superwind metric which
obviously satisfies \eqref{MAMA_general}.

We still face one serious obstacle for successfully carrying out the
argument, namely to control the condition \eqref{CC}, which assures that the
indicatrix of $\tilde{F}_{\eta }$ is strongly convex. Thus, $\tilde{F}_{\eta
}$ will collect all the requirements for a Finsler metric as well as
the fact that the $\tilde{F}_{\eta }$-geodesics locally minimize time.
Exploring this issue, we prove the result.

\begin{lemma}
\label{Lema3} The following statements are equivalent:
\begin{itemize}
\item [i)] for any $C^{\infty }$-function $\eta :M\rightarrow \lbrack 0,1]$%
, the indicatrix $I_{\tilde{F}_{\eta }}$ of the superwind metric $\tilde{F}%
_{\eta }$ is strongly convex;

\item [ii)] the wind $W$ is weak with either ~$||W||_{h}<1$ and $\eta
(x)\in \lbrack 0,\frac{1}{2}]$, or $||W||_{h}<\frac{1}{2\eta }$ and $\eta
(x)\in (\frac{1}{2},1],$ for any $x\in M;$

\item [iii)]  the superwind $\mathcal{W}_{\eta }$ given by \eqref{eq_superwind2x} is
restricted to either $||\mathcal{W}_{\eta }||_{h}<1$ and $\eta (x)\in
\lbrack 0,\frac{1}{2}]$, or $||\mathcal{W}_{\eta }||_{h}<\frac{1}{2\eta }$
and $\eta (x)\in (\frac{1}{2},1],$ for any $x\in M.$
\end{itemize}
\end{lemma}

\begin{proof}
In order to prove the equivalence i) $\Leftrightarrow $ ii), we have to handle the requirement \eqref{CC}, for any $C^{\infty }$-function $\eta :M\rightarrow \lbrack 0,1]$%
. If $\eta (x)=1$, $x\in M,$  the inequality \eqref{CC} is obviously checked.
Now, some basic computations show that the condition \eqref{CC} with $F(x,y)=%
\frac{\alpha ^{2}}{\alpha -\eta \beta }$ leads to the inequality $\frac{%
(1-\eta )||W||_{h}}{1-\eta ||W||_{h}}<1,$ which is equivalent to $%
||W||_{h}<1,$ for any $C^{\infty }$-function $\eta :M\rightarrow \lbrack 0,1)
$.

Combining the last condition and the inequality $||W||_{h}<\frac{1}{%
2\eta }$ which refers to the strong convexity restriction for the indicatrix 
$I_{F}$, for any $C^{\infty }$-function $\eta :M\rightarrow (0,1]$, it turns
out that the indicatrix $I_{\tilde{F}_{\eta }}$ is strongly convex if and
only if either $||W||_{h}<1$ and $\eta (x)\in \lbrack 0,\frac{1}{2}]$, or $%
||W||_{h}<\frac{1}{2\eta }$ and $\eta (x)\in (\frac{1}{2},1],$ for any $x\in
M.$ Since $\frac{1}{2\eta }<1$, for any $\eta (x)\in (\frac{1}{2},1],$ we
outline that the wind $W$ is weak for any $C^{\infty }$-function $\eta
:M\rightarrow \lbrack 0,1]$.

\medskip The proof of \noindent ii) $\Leftrightarrow $ iii) is based on the
the remark that $||\mathcal{W}_{\eta }||_{h}\leq ||W||_{h},$ for any $%
C^{\infty }$-function $\eta :M\rightarrow \lbrack 0,1]$ and moreover, the
maximum of $||\mathcal{W}_{\eta }||_{h}$ coincides with $||W||_{h}.$ 
\end{proof}

Accordingly to \cref{Lema3}, we conclude that the force of the superwind $%
\mathcal{W}_{\eta }$ can be described in terms of the common wind $W,$ by
the inequality $||W||_{h}<\tilde{b}_{0}$ in each navigation problem $%
\mathcal{P}_{\eta }$, where we regard $\tilde{b}_{0}$ as being defined by
\begin{equation}
\ \tilde{b}_{0}=\left\{ 
\begin{array}{cc}
1, & \text{if \ }\eta (x)\in \lbrack 0,\frac{1}{2}] \\ 
\frac{1}{2\eta }, & \text{if \ }\eta (x)\in (\frac{1}{2},1]%
\end{array}%
,\right.  \label{Strong_C}
\end{equation}%
for any $x\in M.$ Clearly, $\tilde{b}_{0} \leq 1$ for any $\eta (x)\in \lbrack 0,1].$

The results obtained in \textit{steps I} and \textit{II} carry out the proof of \cref{Thm1}.

We end this subsection underlying that, to solve the navigation problem $%
\mathcal{P}_{\eta }$, for any $C^{\infty }$-function $\eta :M\rightarrow
\lbrack 0,1]$, the implicit form of the superwind metric, provided by the equation \eqref{MAMA_4} is sufficient, as will be shown in the next section.
Obviously, \eqref{MAMA_general} includes two classical Finsler metrics. First,
if $\eta (x)=1$ for any $x\in M,$ then \eqref{MAMA_general} yields the
Matsumoto metric $\tilde{F}(x,y)=\frac{\alpha ^{2}}{\alpha -\beta },$ with $%
||W||_{h}<\frac{1}{2}.$ Second, if $\eta (x)=0$ for any $x\in M$, the
equation \eqref{MAMA_general} leads to 
\begin{equation}
\tilde{F}\sqrt{\alpha ^{2}+2\beta \tilde{F}+||W||_{h}^{2}\tilde{F}^{2}}%
=\alpha ^{2}+2\beta \tilde{F}+||W||_{h}^{2}\tilde{F}^{2}.
\label{RANDERS_mama}
\end{equation}%
Since $\alpha ^{2}+2\beta \tilde{F}+||W||_{h}^{2}\tilde{F}^{2}>0$, %
the equation \eqref{RANDERS_mama} is reduced to%
\begin{equation*}
(1-||W||_{h}^{2})\tilde{F}^{2}-2\beta \tilde{F}-\alpha ^{2}=0,
\end{equation*}%
which admits only the positive root $\tilde{F}(x,y)=\frac{\sqrt{\alpha
^{2}(1-||W||_{h}^{2})+\beta ^{2}}+\beta }{1-||W||_{h}^{2}},$ under a weak
wind $W,$  i.e., $||W||_{h}<1.$ Using the notations 
\begin{equation*}
\tilde{\alpha}^{2}=\frac{\alpha ^{2}}{1-||W||_{h}^{2}}+\tilde{\beta}%
^{2},\quad \text{ where }\quad \tilde{\beta}=\frac{\beta }{1-||W||_{h}^{2}},
\end{equation*}%
we get the Randers metric $\tilde{F}(x,y)=\tilde{\alpha}+\tilde{\beta}$
which solves Zermelo's navigation problem under the weak wind $W.$

\subsection{Time geodesics: proof of \cref{Thm2}}

\label{Subsec_3.2}We start by mentioning a few ideas involved in the proof of \cref{Thm2}.
Essential for our argument are two facts. The first is coming from the
feature of the superwind metric. It can be referred to as a Finsler metric which
belongs to the extension of the general $(\alpha ,\beta )$-metrics,
introduced in \cref{Subsec_2.1}. Thus, by applying \cref{Prop2} together with
some technical computation, we achieve the spray coefficients related to the
superwind metric $\tilde{F}_{\eta }.$ Further on, in view of \eqref{geo1},
it is immediately possible  to supply the equations of time geodesics of $\tilde{F}%
_{\eta }$. The second fact refers to the property that any such time
geodesic $\gamma $ has a unit length with respect to $\tilde{F}_{\eta }$, i.e., $\tilde{F}_{\eta }(\gamma (t),\dot{\gamma}(t))=1$. This is because,
above all, it is a trajectory in Zermelo's navigation developed in  \cref{Subsec_3.1} (\textit{step II}). Therefore, the time-minimal paths on $(M,h)$ under
the action of the superwind $\mathcal{W}_{\eta },$ with $\eta=\eta (x)\in [0,1],$ for any $x\in
M,$ will be derived by the
equations of time geodesics of $\tilde{F}_{\eta },$ using that along
them $\tilde{F}_{\eta }$ $\ $is equal to $1$. Furthermore, the implicit
formula \eqref{MAMA_4} for $\tilde{F}_{\eta }$ will be sufficient to proceed with the
proof. In view of all this, \cref{Thm2} follows.

We first fix an auxiliary result for the superwind metric $\tilde{F}_{\eta }$%
, which will be important in the sequel.

\begin{lemma}
\label{PropXX}For any $C^{\infty }$-function $\eta :M\rightarrow \lbrack
0,1],$ the superwind metric $\tilde{F}_{\eta }$ is an extension of the
general $(\alpha ,\beta )$-metrics.
\end{lemma}

\begin{proof}
Let us consider the notations $\tilde{\phi}=\frac{\tilde{F}}{\alpha }$ and $%
s=\frac{\beta }{\alpha }.$ If we divide \eqref{MAMA_4} by $\alpha ^{4},$
this leads to the following equation 
\begin{equation}
\begin{array}{c}
(1-\eta )^{2}||W||_{h}^{2}(1-||W||_{h}^{2})\tilde{\phi}^{4}+2(1-\eta
)[1-(2-\eta )||W||_{h}^{2}]s\tilde{\phi}^{3} \\ 
~ \\ 
+[1-2(1-\eta )||W||_{h}^{2}-(2-\eta )^{2}s^{2}]\tilde{\phi}^{2}-2(2-\eta )s%
\tilde{\phi}-1=0,%
\end{array}
\label{PHI}
\end{equation}%
which is equivalent to \eqref{MAMA_4}. Since $\tilde{F}_{\eta }$ is the
unique positive root of \eqref{MAMA_4}, it follows that the equation \eqref{PHI} admits a
sole positive root, denoted by $\tilde{\phi}_{\eta },$ for each $C^{\infty }$%
-function $\eta :M\rightarrow \lbrack 0,1]$ and $\tilde{F}_{\eta }(x,y)=\alpha \tilde{\phi}%
_{\eta }(||W||_{h}^{2},s,\eta ).$  We conclude
that $\tilde{\phi}_{\eta }$ is a $C^{\infty }$-function which depends on the variables: $||W||_{h}^{2}$, $s=%
\frac{\beta }{\alpha }$ and $\eta =\eta (x),$ where $\alpha $ and $\beta $
are given by \eqref{NOT}. Therefore, the claim is true. 
\end{proof}

In view of \cref{PropXX} and \eqref{PHI}, the function $\tilde{\phi}_{\eta }=%
\tilde{\phi}_{\eta }(||W||_{h}^{2},s,\eta )$ checks the following identity 
\begin{equation}
\begin{array}{c}
(1-\eta )^{2}||W||_{h}^{2}(1-||W||_{h}^{2})\tilde{\phi}_{\eta }^{4}+2(1-\eta
)[1-(2-\eta )||W||_{h}^{2}]s\tilde{\phi}_{\eta }^{3} \\ 
~ \\ 
+[1-2(1-\eta )||W||_{h}^{2}-(2-\eta )^{2}s^{2}]\tilde{\phi}_{\eta
}^{2}-2(2-\eta )s\tilde{\phi}_{\eta }-1=0,%
\end{array}
\label{SS3.3}
\end{equation}%
which plays an important role in the proof of \cref{Thm2}. More precisely,
it is the key tool to establish some relations between the function $\tilde{%
\phi}_{\eta }$ and its derivatives $\tilde{\phi}_{\eta 1},$ $\tilde{\phi}%
_{\eta 2},$ $\tilde{\phi}_{\eta 3},$ $\tilde{\phi}_{\eta 12},$ $\tilde{\phi}%
_{\eta 22}$ and $\tilde{\phi}_{\eta 32}$.

\begin{lemma}
\label{Lema4} Let $M$ be an $n$-dimensional $C^{\infty }$-manifold, $n>1,$
with the superwind metric $\tilde{F}_{\eta }(x,y)=\alpha \tilde{\phi}_{\eta
}(||W||_{h}^{2},s,\eta ).$ For any $C^{\infty }$-function $\eta
:M\rightarrow \lbrack 0,1],$ the function $\tilde{\phi}_{\eta }$ and its
derivative with respect to $s,$ i.e., $\tilde{\phi}_{\eta 2}$ hold the
following relations:%
\begin{equation}
\begin{array}{c}
C\tilde{\phi}_{\eta 2}=A\tilde{\phi}_{\eta },\qquad C(\tilde{\phi}_{\eta }-s%
\tilde{\phi}_{\eta 2})=B, \\~\\
C\tilde{\phi}_{\eta }=B+sA\tilde{\phi}_{\eta },\qquad (2-\eta )B-2A=-\eta 
\tilde{\phi}_{\eta }^{2}%
\end{array}
\label{SS3.1}
\end{equation}%
where%
\begin{equation}
\begin{array}{l}
A=-(1-\eta )[1-(2-\eta )||W||_{h}^{2}]\tilde{\phi}_{\eta }^{2}+(2-\eta )^{2}s%
\tilde{\phi}_{\eta }+2-\eta , \\ 
~ \\ 
B=-[1-2(1-\eta )||W||_{h}^{2}]\tilde{\phi}_{\eta }^{2}+2(2-\eta )s\tilde{\phi%
}_{\eta }+2, \\ 
~ \\ 
\begin{split}
C=& \ 2(1-\eta )^{2}||W||_{h}^{2}(1-||W||_{h}^{2})\tilde{\phi}_{\eta
}^{3}+3(1-\eta )[1-(2-\eta )||W||_{h}^{2}]s\tilde{\phi}_{\eta }^{2} \\
& +[1-2(1-\eta )||W||_{h}^{2}-(2-\eta )^{2}s^{2}]\tilde{\phi}_{\eta
}-(2-\eta )s,
\end{split}%
\end{array}
\label{SS3.2}
\end{equation}%
and $A$, $B$, $C$ are evaluated at $(||W||_{h}^{2},s,\eta ).$
\end{lemma}

\begin{proof}
By differentiating \eqref{SS3.3} with respect to $s,$ we get the first
relation from \eqref{SS3.1}. Then, in view of this,  the second
identity of \eqref{SS3.1} follows. The last two identities result from the notations %
\eqref{SS3.2} and \eqref{SS3.3}. 
\end{proof}

Moreover, since we work under the condition $||W||_{h}<\tilde{b}_{0},$ with $%
\tilde{b}_{0}$ defined in \eqref{Strong_C}, which according to \cref{Thm1}%
, assures that the indicatrix $I_{\tilde{F}_{\eta }}$ is strongly convex,
we can apply the direct implication of \cref{Prop1}. Thus, for any $%
C^{\infty }$-function $\eta :M\rightarrow \lbrack 0,1]$ and $s$ satisfying $%
|s|\leq ||W||_{h}<\tilde{b}_{0}$, the validity of the following inequalities
is guaranteed 
\begin{equation*}
\tilde{\phi}_{\eta }-s\tilde{\phi}_{\eta 2}>0,\qquad \tilde{\phi}_{\eta }-s%
\tilde{\phi}_{\eta 2}+(||W||_{h}^{2}-s^{2})\tilde{\phi}_{\eta 22}>0,
\end{equation*}%
when $n\geq 3,$ or only the right-hand side inequality, when $n=2$. Our
purpose now is to evaluate the expression $\tilde{\phi}_{\eta }-s\tilde{\phi}%
_{\eta 2}$ which corresponds to the superwind metric (i.e., the second
formula in \eqref{SS3.1}) for any dimension $n>1.$

\begin{lemma}
\label{Lema44} Let $M$ be an $n$-dimensional $C^{\infty }$-manifold, $n>1,$
with the superwind metric $\tilde{F}_{\eta }(x,y)=\alpha \tilde{\phi}_{\eta
}(||W||_{h}^{2},s,\eta ).$ For any $C^{\infty }$-function $\eta
:M\rightarrow \lbrack 0,1]$, the following statements hold:
\begin{itemize}
\item[i)] $C(||W||_{h}^{2},s,\eta )\neq 0$;

\item[ii)]  $B(||W||_{h}^{2},s,\eta )\neq 0$.
\end{itemize}
\end{lemma}

\begin{proof} In order to prove i) we assume by contradiction that there exists $%
s_{0}\in \lbrack -b,b]$  such that $C(||W||_{h}^{2},s_{0},\eta )=0,$ where $b=||W||_{h}<\tilde{b}_{0}$ and $\tilde{b}_{0}$
is given by \eqref{Strong_C}. In
view of our assumption and \eqref{SS3.1}, it follows that $%
A(||W||_{h}^{2},s_{0},\eta )=B(||W||_{h}^{2},s_{0},\eta )=0.$ After plugging these
into \eqref{SS3.3}, we arrive at%
\begin{equation}
(1-\eta )^{2}||W||_{h}^{2}(1-|||W||_{h}^{2})\tilde{\phi}_{\eta
}^{4}(||W||_{h}^{2},s_{0},\eta )+[(2-\eta )s_{0}\tilde{\phi}_{\eta
}(||W||_{h}^{2},s_{0},\eta )+1]^{2}=0,  \label{SS3.4}
\end{equation}%
which provides a contradiction. Indeed, since $||W||_{h}<1$ and $\tilde{\phi}%
_{\eta }(||W||_{h}^{2},s_0,\eta )>0,$ it turns out that $(1-\eta
)^{2}||W||_{h}^{2}(1-|||W||_{h}^{2})\tilde{\phi}_{\eta
}^{4}(||W||_{h}^{2},s_{0},\eta )\neq 0$ if $\eta (x)\neq 1$, for any $x\in M$. So, \eqref{SS3.4} is contradicted. If $\eta (x)=1$ for some $x\in M$, then
the left hand-side of \eqref{SS3.4} is reduced to $(s_{0}\tilde{\phi}_{\eta
}(||W||_{h}^{2},s_{0},\eta )+1)^{2}=\frac{1}{(1-s_{0})^{2}}\neq 0$ which
contradicts \eqref{SS3.4}. Thus, it is shown that $C\neq 0$ everywhere.

To prove the statement ii), we again assume towards a
contradiction that there is $\tilde{s}\in \lbrack -b,b],$ such that $%
B(||W||_{h}^{2},\tilde{s},\eta )=0,$ where $b=||W||_{h}<%
\tilde{b}_{0}$ and $\tilde{b}_{0}$ is given by \eqref{Strong_C}.  Therefore, we are searching for such an $\tilde{s}\in \lbrack -b,b].$ Let us take $s=\tilde{s}$ in the third formula
in \eqref{SS3.1}. Since $\tilde{\phi}_{\eta }(||W||_{h}^{2},\tilde{s},\eta
)>0,$ $C(||W||_{h}^{2},\tilde{s},\eta )\neq 0$ and $B(||W||_{h}^{2},\tilde{s}%
,\eta )=0,$ one has that $\tilde{s}\neq 0.$ Moreover, by the second formula
in \eqref{SS3.2}, $\tilde{\phi}_{\eta }(||W||_{h}^{2},\tilde{s},\eta )$
satisfies the polynomial equation%
\begin{equation}
\lbrack 1-2(1-\eta )||W||_{h}^{2}]\tilde{\phi}_{\eta }^{2}-2(2-\eta )\tilde{s%
}\tilde{\phi}_{\eta }-2=0  \label{SI}
\end{equation}%
and thus, for $s=\tilde{s}$ and for any $C^{\infty }$-function $\eta
:M\rightarrow \lbrack 0,1]$, the equation \eqref{SS3.3} is reduced to 
\begin{equation}
\begin{array}{c}
2\left( 1-\eta \right) ^{2}||W||_{h}^{2}(1-||W||_{h}^{2})\tilde{\phi}_{\eta
}^{2}+[2-3\eta -2\left( 2-\eta \right) \left( 1-\eta \right) ||W||_{h}^{2}]%
\tilde{s}\tilde{\phi}_{\eta }+1-2\left( 1-\eta \right) ||W||_{h}^{2}=0.%
\end{array}
\label{SII}
\end{equation}%
Since $||W||_{h}^{2}<\tilde{b}_{0}$, with $\tilde{b}_{0}$ given by %
\eqref{Strong_C}, one sees that $1-||W||_{h}^{2}\neq 0$ for any $%
C^{\infty }$-function $\eta :M\rightarrow \lbrack 0,1].$ However, there may exist $\eta (x)\in \lbrack 0,\frac{1}{2})$ for some $x\in M$ such that $1-2(1-\eta
)||W||_{h}^{2}=0.$ Thus, two cases must be analyzed separately:

\noindent a) if $1-2(1-\eta )||W||_{h}^{2}\neq 0,$ for any $C^{\infty }$%
-function $\eta :M\rightarrow \lbrack 0,1],$ then by ~\eqref{SI} and ~%
\eqref{SII}, we get $[1-4\eta \left( 1-\eta \right) ||W||_{h}^{2}]\tilde{\phi}%
_{\eta }-4\eta \tilde{s}=0,$ which provides a contradiction if $\eta (x)=0.$
Thus, $\eta (x)\neq 0$ for any $x\in M$ and making use of $\tilde{s}\neq 0$
and $\tilde{\phi}_{\eta }(||W||_{h}^{2},\tilde{s},\eta )>0,$ it turns out
that $1-4\eta \left( 1-\eta \right) ||W||_{h}^{2}\neq 0$ and   
\begin{equation}
\tilde{\phi}_{\eta }(||W||_{h}^{2},\tilde{s},\eta )=\frac{4\eta \tilde{s}}{%
1-4\eta \left( 1-\eta \right) ||W||_{h}^{2}}.  \label{SIII}
\end{equation}%
Plugging \eqref{SIII} into \eqref{SI}, one has that $\tilde{s}^{2}=\frac{%
[1-4\eta \left( 1-\eta \right) ||W||_{h}^{2}]^{2}}{4\eta \left[ 3\eta
-2+4\eta \left( 1-\eta \right) ||W||_{h}^{2}\right] }$ which contradicts $%
\tilde{s}^{2}\in (0,b^{2}]$ due to the condition $||W||_{h}<\tilde{b}_{0}$,
where $\tilde{b}_{0}$ is given by \eqref{Strong_C}.

\noindent b) if $1-2(1-\eta )||W||_{h}^{2}=0$, with $\eta (x)\in \lbrack 0,%
\frac{1}{2})$ for some $x\in M$, then \eqref{SI} leads to 
\begin{equation}
\tilde{\phi}_{\eta }(||W||_{h}^{2},\tilde{s},\eta )=-\frac{1}{(2-\eta )%
\tilde{s}},  \label{SIV}
\end{equation}%
which together with \eqref{SII} yields $-\eta \tilde{s}^{2}=\frac{1-2\eta }{%
4(2-\eta )}.$ Obviously, the last relation provides a contradiction (i.e.,  $%
\tilde{s}^{2}<0$) for $\eta (x)\in \lbrack 0,\frac{1}{2}).$

We conclude from above that $B(||W||_{h}^{2},\tilde{s},\eta )\neq 0,$ \ for
any $s\in \lbrack -b,b],$ $b=||W||_{h}<\tilde{b}_{0}$, where $\tilde{b}_{0}$
is given by \eqref{Strong_C}. 
\end{proof}

Therefore, by \cref{Lema44}, one has established that $\tilde{\phi}_{\eta }-s%
\tilde{\phi}_{\eta 2}\neq 0$ is also valid  when $n=2,$ for any $\ C^{\infty }$%
-function $\eta :M\rightarrow \lbrack 0,1]$ and $|s|\leq ||W||_{h}<\tilde{b}%
_{0},$ with $\tilde{b}_{0}$ given by \eqref{Strong_C}. Consequently, for any 
$n>1,$ $\tilde{\phi}_{\eta }-s\tilde{\phi}_{\eta 2}=\frac{B}{C}\neq 0$ and $%
\tilde{\phi}_{\eta 2}=\frac{A}{C}\tilde{\phi}_{\eta }$.

\begin{lemma}
\label{Lema5} Let $M$ be an $n$-dimensional $C^{\infty }$-manifold, $n>1,$
with the superwind metric $\tilde{F}_{\eta }(x,y)=\alpha \tilde{\phi}_{\eta
}(||W||_{h}^{2},s,\eta ).$ For any $C^{\infty }$-function $\eta
:M\rightarrow \lbrack 0,1],$ the first order derivatives of the function $%
\tilde{\phi}_{\eta }$ with respect to $||W||_{h}^{2}$ and $\eta $ (i.e., $%
\tilde{\phi}_{\eta 1}$ and $\tilde{\phi}_{\eta 3}$ respectively) and the
second order derivatives $\tilde{\phi}_{\eta 12}$, $\tilde{\phi}_{\eta 22}$
and $\tilde{\phi}_{\eta 32}$ hold the following relations:%
\begin{equation}
\begin{array}{l}
\tilde{\phi}_{\eta 1}=\frac{1-\eta }{2C}(B+\eta \tilde{\phi}_{\eta }^{2})%
\tilde{\phi}_{\eta }^{2}, \\ 
~ \\ 
\tilde{\phi}_{\eta 3}=\frac{1}{2C}\{s(\tilde{\phi}_{\eta }^{2}-B)+[(1-2\eta )%
\tilde{\phi}_{\eta }^{2}-B]||W||_{h}^{2}\tilde{\phi}_{\eta }\}\tilde{\phi}%
_{\eta }, \\ 
\\ 
\tilde{\phi}_{\eta 12}=\frac{1-\eta }{2C^{3}}\{A(B+C\tilde{\phi}_{\eta
})(B+\eta \tilde{\phi}_{\eta }^{2})+\eta ^{2}[2+(1-\eta )s\tilde{\phi}_{\eta
}]\tilde{\phi}_{\eta }^{4}\}\tilde{\phi}_{\eta }, \\ 
~ \\ 
\tilde{\phi}_{\eta 22}=\frac{1}{C^{3}}(A^{2}B+\eta ^{2}\tilde{\phi}_{\eta
}^{4}), \\ 
\\ 
\tilde{\phi}_{\eta 32}=\frac{1}{2C^{3}}\{s(\tilde{\phi}_{\eta
}^{2}-B)+[(1-2\eta )\tilde{\phi}_{\eta }^{2}-B]||W||_{h}^{2}\tilde{\phi}%
_{\eta }\}\{A(B+C\tilde{\phi}_{\eta })+\eta \lbrack 2+(2-\eta )s\tilde{\phi}%
_{\eta }]\tilde{\phi}_{\eta }^{2}\} \\~\\ 
\text{ \ \ \ \ \ \ \ \ }+\frac{1}{2C^{2}}[B(\tilde{\phi}_{\eta
}^{2}-B)+2\eta (s+||W||_{h}^{2}\tilde{\phi}_{\eta })\tilde{\phi}_{\eta
}^{3}].%
\end{array}
\label{RRR}
\end{equation}
\end{lemma}

\begin{proof}
By differentiating \eqref{SS3.3}  with respect to $||W||_{h}^{2}$ and $\eta ,
$ the first two expressions in \eqref{RRR} follow.
A direct computation leads to the derivatives of \eqref{SS3.2} with respect
to $s$, which read 
\begin{eqnarray*}
A_{2} &=&\frac{1}{C}[2A^{2}-\eta (2-\eta )\tilde{\phi}_{\eta }^{2}],\qquad
B_{2}=\frac{2}{C}(AB-\eta \tilde{\phi}_{\eta }^{2}), \\
C_{2} &=&-\frac{1}{C\tilde{\phi}_{\eta }}\{AB+\eta \lbrack 2+(2-\eta )s%
\tilde{\phi}_{\eta }]\tilde{\phi}_{\eta }^{2}\}+3A,
\end{eqnarray*}%
where $A_{2}=\frac{\partial A}{\partial s},$ $B_{2}=\frac{\partial B}{%
\partial s},$ $C_{2}=\frac{\partial B}{\partial s}.$ These, along with%
\begin{eqnarray*}
\tilde{\phi}_{\eta 12} &=&\frac{1-\eta }{2C^{2}}\left[ B_{2}C+2A(B+2\eta 
\tilde{\phi}_{\eta }^{2})-(B+\eta \tilde{\phi}_{\eta }^{2})C_{2}\right] 
\tilde{\phi}_{\eta }^{2}, \\
\tilde{\phi}_{\eta 22} &=&\frac{1}{C^{2}}(A_{2}C+A^{2}-AC_{2})\tilde{\phi}%
_{\eta }, \\
\tilde{\phi}_{\eta 32} &=&\frac{1}{2C^{2}}\{(\tilde{\phi}_{\eta
}^{2}-B)C+s[2A\tilde{\phi}_{\eta }^{2}-B_{2}C]+[3(1-2\eta )A\tilde{\phi}%
_{\eta }^{2}-AB-B_{2}C]||W||_{h}^{2}\tilde{\phi}_{\eta }\}\tilde{\phi}_{\eta
} \\
&&+\frac{1}{2C^{2}}\{s(\tilde{\phi}_{\eta }^{2}-B)+[(1-2\eta )\tilde{\phi}%
_{\eta }^{2}-B]||W||_{h}^{2}\tilde{\phi}_{\eta }\}(A-C_{2})\tilde{\phi}%
_{\eta }
\end{eqnarray*}%
imply the last three formulas in \eqref{RRR}. 
\end{proof}

The above preparatory part \noindent together with \cref{Prop2} helps us 
provide the spray coefficients corresponding to the superwind metric $\tilde{%
F}_{\eta }.$

\begin{lemma}
\label{Prop5} Let $M$ be an $n$-dimensional manifold, $n>1,$ with the
superwind metric $\tilde{F}_{\eta },$ for any $C^{\infty }$-function $\eta
:M\rightarrow \lbrack 0,1]$. Then the relationship between the spray
coefficients $\tilde{\mathcal{G}}_{\eta }^{i}$ of $\tilde{F}_{\eta }$ and
the spray coefficients $\mathcal{G}_{\alpha }^{i}=\frac{1}{4}h^{im}\left( 2%
\frac{\partial h_{jm}}{\partial x^{k}}-\frac{\partial h_{jk}}{\partial x^{m}}%
\right) y^{j}y^{k}$ of $\alpha $ is given by%
\begin{eqnarray}
\tilde{\mathcal{G}}_{\eta }^{i}(x,y) &=&\mathcal{G}_{\alpha
}^{i}(x,y)+\alpha Qs_{0}^{i}-\alpha ^{2}[R(r^{i}+s^{i})+\frac{1}{2}P\eta
^{i}]  \label{SPRAY_s} \\
&&+\{\mathit{\Theta }[-2\alpha Qs_{0}+r_{00}+\alpha ^{2}(2Rr+Pq)]+\alpha
\lbrack \mathit{\Omega }(r_{0}+s_{0})+\frac{1}{2}\mathit{\Lambda }\eta
_{0}]\}\frac{y^{i}}{\alpha }  \notag \\
&&-\{\mathit{\Psi }[-2\alpha Qs_{0}+r_{00}+\alpha ^{2}(2Rr+Pq)]+\alpha
\lbrack \mathit{\Pi }(r_{0}+s_{0})+\frac{1}{2}\mathit{\Upsilon }\eta
_{0}\}W^{i},  \notag
\end{eqnarray}%
$i=1,...,n$, where%
\begin{equation}
\begin{array}{l}
\eta _{k}=\frac{\partial \eta }{\partial x^{k}},\qquad \eta _{0}=\frac{%
\partial \eta }{\partial x^{k}}y^{k},\qquad \eta ^{i}=h^{ij}\eta _{j},\qquad
q=-W_{k}\eta ^{k},\text{\ } \\ 
\\ 
r_{00}=-W_{j|k}y^{j}y^{k},\qquad r_{0}=\frac{1}{2}%
(W_{j|k}+W_{k|j})y^{j}W^{k},\qquad r^{i}+s^{i}=h^{ij}W_{k|j}W^{k}, \\ 
\\ 
r=-W_{j|k}W^{j}W^{k},\qquad s_{0}=-\frac{1}{2}(W_{j|k}-W_{k|j})y^{j}W^{k},%
\qquad s_{0}^{i}=-\frac{1}{2}h^{ij}(W_{j|k}-W_{k|j})y^{k}, \\ 
\\ 
R=\frac{1-\eta }{2\alpha ^{4}B}[\alpha ^{2}B+\eta \tilde{F}_{\eta }^{2}]%
\tilde{F}_{\eta }^{2},\text{ \ }\text{ \ }P=\frac{1}{2\alpha ^{4}B}\{\beta (%
\tilde{F}_{\eta }^{2}-\alpha ^{2}B)+[(1-2\eta )\tilde{F}_{\eta }^{2}-\alpha
^{2}B]||W||_{h}^{2}\tilde{F}_{\eta }\}\tilde{F}_{\eta }, \\ 
\\ 
Q=\frac{1}{\alpha B}A\tilde{F}_{\eta },\text{ \ }\text{ \ }\mathit{\Theta }=%
\frac{\alpha }{2E\tilde{F}_{\eta }}(\alpha ^{6}AB^{2}-\eta ^{2}\beta \tilde{F%
}_{\eta }^{5}),\text{ \ \ }\mathit{\Psi }=\frac{\alpha ^{2}}{2E}(\alpha
^{4}A^{2}B+\eta ^{2}\tilde{F}_{\eta }^{4}), \\ 
~ \\ 
\mathit{\Omega }=\frac{1-\eta }{\alpha ^{2}BE}[(\alpha ^{2}B+\eta \tilde{F}%
_{\eta }^{2})(\alpha ^{6}B^{3}+\eta ^{2}||W||_{h}^{2}\tilde{F}_{\eta
}^{6})-\eta ^{2}\alpha ^{2}\tilde{F}_{\eta }^{5}(\beta B+||W||_{h}^{2}A%
\tilde{F}_{\eta })] \\ 
\text{ \ \ \ \ }~ \\ 
\mathit{\Pi }=\frac{1-\eta }{2\alpha ^{3}BE}\{(\alpha ^{2}B+\eta \tilde{F}%
_{\eta }^{2})(2\alpha ^{6}AB^{2}-\eta ^{2}\beta \tilde{F}_{\eta %
}^{5})+\eta ^{2}\alpha ^{2}B\tilde{F}_{\eta }^{4}[2\alpha ^{2}+\left( 1-\eta
\right) \beta \tilde{F}_{\eta }]\tilde{F}_{\eta }, \\ 
\\ 
\mathit{\Lambda }=\frac{1}{\alpha ^{2}BE\tilde{F}_{\eta }}\{\beta (\tilde{F}%
_{\eta }^{2}-\alpha ^{2}B)+[(1-2\eta )\tilde{F}_{\eta }^{2}-\alpha
^{2}B]||W||_{h}^{2}\tilde{F}_{\eta }\}(\alpha ^{6}B^{3}+\eta ^{2}\tilde{F}%
_{\eta }^{6}||W||_{h}^{2}) \\~\\
\qquad -\frac{1}{BE\tilde{F}_{\eta }}\{\frac{1}{2}\alpha ^{4}B^{2}(\tilde{F}%
_{\eta }^{2}-\alpha ^{2}B)+\eta \tilde{F}_{\eta }^{5}[\beta +(1-2\eta
)||W||_{h}^{2}\tilde{F}_{\eta }]\}(\beta B+||W||_{h}^{2}A\tilde{F}_{\eta }),
\\ 
\\ 
\mathit{\Upsilon }=\frac{1}{\alpha ^{3}BE}\{\beta (\tilde{F}_{\eta
}^{2}-\alpha ^{2}B)+[(1-2\eta )\tilde{F}_{\eta }^{2}-\alpha
^{2}B]||W||_{h}^{2}\tilde{F}_{\eta }\}[\alpha ^{6}AB^{2}+\eta \tilde{F}%
_{\eta }^{3}(\alpha ^{3}C-\eta \beta \tilde{F}_{\eta }^{2})] \\~\\
\qquad +\frac{\alpha ^{2}}{E}C[\frac{1}{2}\alpha ^{2}B(\tilde{F}_{\eta
}^{2}-\alpha ^{2}B)+\eta \tilde{F}_{\eta }^{3}(\beta +||W||_{h}^{2}\tilde{F}%
_{\eta })],%
\end{array}
\label{terms}
\end{equation}%
with 
\begin{equation}
\begin{array}{l}
A=-\frac{1}{\alpha ^{2}}\{\left( 1-\eta \right) \left[ 1-\left( 2-\eta
\right) ||W||_{h}^{2}\right] \tilde{F}_{\eta }^{2}-(2-\eta )^{2}\beta \tilde{%
F}_{\eta }-(2-\eta )\alpha ^{2}\}, \\ 
~ \\ 
B=-\frac{1}{\alpha ^{2}}\{[1-2(1-\eta )||W||_{h}^{2}]\tilde{F}_{\eta
}^{2}-2(2-\eta )\beta \tilde{F}_{\eta }-2\alpha ^{2}\}, \\ 
~ \\ 
C=\frac{1}{\alpha \tilde{F}_{\eta }}\left( \alpha ^{2}B+\beta A\tilde{F}%
_{\eta }\right) ,\text{ \ }E=\alpha ^{6}BC^{2}+(||W||_{h}^{2}\alpha
^{2}-\beta ^{2})(\alpha ^{4}A^{2}B+\eta ^{2}\tilde{F}_{\eta }^{4}).%
\end{array}
\label{ABC}
\end{equation}
\end{lemma}

\begin{proof}
By using \cref{Prop2} and further simple computations, the following
expressions are obtained 
\begin{equation}
\begin{array}{l}
s\tilde{\phi}_{\eta }+(b^{2}-s^{2})\tilde{\phi}_{\eta 2}=\frac{1}{C}%
(sB+||W||_{h}^{2}A\tilde{\phi}_{\eta }), \\ 
~ \\ 
(\tilde{\phi}_{\eta }-s\tilde{\phi}_{\eta 2})\tilde{\phi}_{\eta 2}-s\tilde{%
\phi}_{\eta }\tilde{\phi}_{\eta 22}=\frac{1}{C^{3}}(AB^{2}-\eta ^{2}s\tilde{%
\phi}_{\eta }^{5}), \\ 
~ \\ 
\tilde{\phi}_{\eta }-s\tilde{\phi}_{\eta 2}+(b^{2}-s^{2})\tilde{\phi}_{\eta
22}=\frac{1}{C^{3}}[BC^{2}+(||W||_{h}^{2}-s^{2})(A^{2}B+\eta ^{2}\tilde{\phi}%
_{\eta }^{4})], \\ 
~ \\ 
(\tilde{\phi}_{\eta }-s\tilde{\phi}_{\eta 2})\tilde{\phi}_{\eta 12}-s\tilde{%
\phi}_{\eta 1}\tilde{\phi}_{\eta 22}=\frac{1-\eta }{2C^{4}}\{(B+\eta \tilde{%
\phi}_{\eta }^{2})(2AB^{2}-\eta ^{2}s\tilde{\phi}_{\eta }^{5})+\eta
^{2}[2+(1-\eta )s\tilde{\phi}_{\eta }]B\tilde{\phi}_{\eta }^{4}\}\tilde{\phi}%
_{\eta }. \\ 
\\ 
(\tilde{\phi}_{\eta }-s\tilde{\phi}_{\eta 2})\tilde{\phi}_{\eta 32}-s\tilde{%
\phi}_{\eta 3}\tilde{\phi}_{\eta 22}=\text{\ }\frac{B}{2C^{3}}[B(\tilde{\phi}%
_{\eta }^{2}-B)+2\eta (s+||W||_{h}^{2}\tilde{\phi}_{\eta })\tilde{\phi}%
_{\eta }^{3}] \\~\\ 
\qquad \qquad \qquad \qquad \qquad \qquad+\frac{1}{C^{4}}\{s(\tilde{\phi}_{\eta }^{2}-B)+[(1-2\eta )\tilde{\phi}%
_{\eta }^{2}-B]||W||_{h}^{2}\tilde{\phi}_{\eta }\}[AB^{2}+\eta (C-\eta s 
\tilde{\phi}_{\eta }^{2})\tilde{\phi}_{\eta }^{3}]\text{\ }%
\end{array}
\label{RRRR}
\end{equation}

In view of \eqref{NOT} and \eqref{rs}, we get that 
\begin{equation}
\begin{array}{l}
a_{ij}=h_{ij},\;\;b_{i}=-W_{i},\;\;b^{i}=h^{ji}b_{j}=-W^{i}, \\~\\ 
r_{ij}=-\frac{1}{2}(W_{i|j}+W_{j|i}),\quad s_{ij}=-\frac{1}{2}%
(W_{i|j}-W_{j|i}),\quad \\~\\ 
\eta _{0}=\eta _{i}y^{i},\text{ \ \ }\eta ^{i}=h^{ij}\eta _{j},\text{ \ \ \ }%
q=-W^{i}\eta _{i},%
\end{array}
\label{RR}
\end{equation}%
where $\eta _{i}=\frac{\partial \eta }{\partial x^{i}},$ $W_{i|j}=\frac{%
\partial W_{i}}{\partial x^{j}}-\Gamma _{ij}^{k}W_{k},$ and $\Gamma
_{ij}^{k}=\frac{1}{2}h^{km}\left( \frac{\partial h_{jm}}{\partial x^{i}}+%
\frac{\partial h_{im}}{\partial x^{j}}-\frac{\partial h_{ij}}{\partial x^{m}}%
\right) $. Collecting the findings \eqref{RRRR} and \eqref{RR} together %
\eqref{rs}, one can apply \cref{Prop2} and thus, our claim follows at once. 
\end{proof}
Finally, on account of the system \eqref{geo1} and \cref{Prop5} with $%
\tilde{F}_{\eta }(\gamma (t),\dot{\gamma}(t))=1$, we can write the ODE
system \eqref{GGG} which yields the time-minimal paths $\gamma (t)=(\gamma
^{i}(t)),$ $i=1,...,n$ on $(M,h)$ under the action of the superwind $%
\mathcal{W}_{\eta }$, for any $C^{\infty }$-function $\eta :M\rightarrow
\lbrack 0,1].$ This concludes the proof of \cref{Thm2}.

\newpage

\section{Examples including juxtapositions with the preceding \linebreak 
 studies
}

\label{Sec_4}

By several comprehensive examples in dimension 2 and discussion we present the comparisons of the contributions coming from the above developed theory with the preceding findings, which were obtained in the presence of wind in the Zermelo sense and gravity, i.e., the slippery slope model,  including in particular the original Matsumoto scenario; cf. \cite[Sec. 5]{slippery}, \cite[Sec. 6.1]{general}. 
In the below subsections, we analyze the individual impacts of perturbations in the sense of the newly introduced notions of superwind and varying (space-dependent) traction, as well as their combined effects on behaviour of the time-minimizing geodesics and evolution of the related time fronts.

We first present a brief overview of the problem $\mathcal{P}_{\eta }$ on a
generalized \textquotedblleft slope\textquotedblright\ \
referred as a surface $(M,h)$ embedded in $\mathbb{R}^{3}$ under the
action of a superwind $\mathcal{W}_{{\eta }}=\eta \mathcal{W}_{MAT}+(1-\eta
)W$ with a $C^{\infty }$-function $\eta :M\rightarrow \lbrack 0,1].$ Since $%
(M,h)$ is a $2$-dimensional Riemannian manifold, throughout this section $%
(x^{1},x^{2})$ denotes a local coordinate system on a local chart at an
arbitrary point $O\in M$.
If we parametrize $M$ by $(x^{1},x^{2})\in M\longmapsto (x,y,z)\in \mathbb{R}%
^{3},$ where $x=x^{1},$ $y=x^{2},$ $z=f(x^{1},x^{2})$ and $f$ is a $%
C^{\infty }$-function on $M$, then the Riemannian metric $h$ induced on $M$
is defined by $\left( h_{ij}(x^{1},x^{2})\right) =\left( 
\begin{array}{cc}
1+f_{x^{1}}^{2} & f_{x^{1}}f_{x^{2}} \\ 
f_{x^{1}}f_{x^{2}} & 1+f_{x^{2}}^{2}%
\end{array}%
\right) $, $i,j=1,2,$ where the notation $f_{x^{i}}$ means the partial
derivative of $f$ with respect to $x^{i},$ $i=1,2.$

Let us consider the plane $\pi _{O}$ tangent to $M$ at the point $%
O=(x^{1},x^{2})\in M.$ Since it is spanned by the vectors $\frac{\partial }{%
\partial x^{1}}=(1,0,f_{x^{1}})$ and $\frac{\partial }{\partial x^{2}}%
=(0,1,f_{x^{2}})$, the common wind $W$ can be locally written as 
\begin{equation}
W=W^{1}\frac{\partial }{\partial x^{1}}+W^{2}\frac{\partial }{\partial x^{2}}%
,  \label{wind}
\end{equation}%
with $||W||_{h}=h_{ij}(x^{1},x^{2})W^{i}W^{j}.$ Aiming to describe $\mathcal{P}_{\eta }$ in the current background, we construct an
orthonormal basis $\{e_{1},e_{2}\}$ in the tangent plane $\pi _{O}$, where $%
e_{1}$ has the same direction as $W,$ explicitly one has that $e_{1}=\frac{1}{||W||_{h}}%
W.$ Since $e_{2}$ is orthogonal to $e_{1},$ it can be determined requiring
that $h(e_{1},e_{2})=0$ and $||e_{2}||_{h}=1,$ with $e_{2}=W_{\bot }^{1}%
\frac{\partial }{\partial x^{1}}+W_{\bot }^{2}\frac{\partial }{\partial x^{2}%
}.$ By some algebraic manipulation, we get that%
\begin{equation*}
W_{\bot }^{1}=\mp \frac{W^{k}h_{k2}}{||W||_{h}\det h}\text{ and }W_{\bot
}^{2}=\pm \frac{W^{k}h_{k1}}{||W||_{h}\det h},
\end{equation*}%
where $\det h$ denotes $\det \left( h_{ij}(x^{1},x^{2})\right) $. So, we can 
choose $e_{2}=\frac{1}{||W||_{h}\det h}(-W^{k}h_{k2}\frac{\partial }{%
\partial x^{1}}+W^{k}h_{k1}\frac{\partial }{\partial x^{2}}),$ (with the
Einstein convention that $W^{k}h_{k2}=W^{1}h_{12}+W^{2}h_{22}$ and $%
W^{k}h_{k1}=W^{1}h_{11}+W^{2}h_{21}$) and then, by using the notation $%
\mathbf{y}$ for an arbitrary tangent vector in $\pi _{O},$ it can be written
as%
\begin{equation*}
\mathbf{y}=y^{1}\frac{\partial }{\partial x^{1}}+y^{2}\frac{\partial }{%
\partial x^{2}}=\tilde{y}^{1}e_{1}+\tilde{y}^{2}e_{2},
\end{equation*}%
with respect to the bases $\{\frac{\partial }{\partial x^{1}},\frac{\partial 
}{\partial x^{2}}\}$ and $\{e_{1},e_{2}\}$ in $\pi _{O}$, respectively.
Next, by combining these with the explicit expressions for $e_{1}$ and $e_{2}$%
, we get the following connection between the coordinates $(\tilde{y}^{1},%
\tilde{y}^{2})$ and $(y^{1},y^{2})$
\begin{equation}
\tilde{y}^{1}=\frac{1}{||W||_{h}}(W^{k}h_{k1}y^{1}+W^{k}h_{k2}y^{2}),\text{
\ \ \ \ }\tilde{y}^{2}=\frac{\det h}{||W||_{h}}(-W^{2}y^{1}+W^{1}y^{2}).
\label{EE1}
\end{equation}%
Moreover, for any $C^{\infty }$-function $\eta :M\rightarrow \lbrack 0,1],$
the equations of motion under the action of a superwind $\mathcal{W}_{{%
\eta }}$ ($\mathbf{y}=u+\mathcal{W}_{{\eta }}$) in the coordinates $(\tilde{y%
}^{1},\tilde{y}^{2})$, which correspond to $\{e_{1},e_{2}\}$, read as
follows 
\begin{equation}
\left\{ 
\begin{array}{rll}
\tilde{y}^{1} & = & (1+\eta ||W||_{h}\cos \omega )\cos \omega +(1-\eta
)||W||_{h} \\ 
~ &  &  \\ 
\tilde{y}^{2} & = & (1+\eta ||W||_{h}\cos \omega )\sin \omega%
\end{array}%
,\right.  \label{2_indicatrix}
\end{equation}%
for any $\omega \in \lbrack 0,2\pi )$ which is referred as the 
angle between the self-velocity $u$ and the wind $W$ (measured clockwise from the latter), with $||u||_{h}=1.$ By eliminating the parameter $\omega $ from \eqref{2_indicatrix}, we obtain an
implicit equation of motion on 
  $(M,h)$, namely 
\begin{equation}
\sqrt{\lbrack \tilde{y}^{1}-(1-\eta )||W||_{h}]^{2}+(\tilde{y}^{2})^{2}}=(%
\tilde{y}^{1})^{2}+(\tilde{y}^{2})^{2}-(2-\eta )\tilde{y}^{1}||W||_{h}+(1-%
\eta )||W||_{h}^{2}.  \label{EE}
\end{equation}%
On account of \eqref{NOT} and \eqref{EE1}, let us observe that 
\begin{equation}
(\tilde{y}^{1})^{2}+(\tilde{y}^{2})^{2}=\alpha ^{2}\text{ \ \ and }\tilde{y}%
^{1}=-\frac{1}{||W||_{h}}\beta ,  \label{AB}
\end{equation}%
and thus, \eqref{EE} can be equivalently written in the form 
\begin{equation}
\sqrt{\alpha ^{2}+2(1-\eta )\beta +(1-\eta )^{2}||W||_{h}^{2}}=\alpha
^{2}+(2-\eta )\beta +(1-\eta )||W||_{h}^{2}.  \label{III}
\end{equation}%
Finally, by applying Okubo's method, \eqref{III} leads to %
\eqref{MAMA_general} which implicitly provides the superwind metric $\tilde{F%
}_{\eta }$. We note that the equation \eqref{III} or the system %
\eqref{2_indicatrix} give the indicatrix of $\tilde{F}_{\eta }$.

Now, without loss of generality we can assume that the slope is modeled by the
inclined plane $z=\frac{1}{2}x$ of the slope
angle $26.6^{\circ }$, i.e., $\ f(x^{1},x^{2})=\frac{1}{2}x$,
where $x=x^{1},$ $y=x^{2}$, which will be denoted by $\mathfrak{S}$. It follows that $%
\left( h_{ij}(x^{1},x^{2})\right) =\left( 
\begin{array}{cc}
5/4 & 0 \\ 
0 & 1%
\end{array}%
\right) ,$ $i,j=1,2$, and thus,%
\begin{equation}
\alpha ^{2}=h_{ij}y^{i}y^{j}=\frac{5}{4}(y^{1})^{2}+(y^{2})^{2}\text{, \ \ }%
\mathcal{G}_{\alpha }^{i}=0\text{ \ \ and \ \ \ }\beta =-\frac{5}{4}%
W^{1}y^{1}-W^{2}y^{2}.  \label{ALL}
\end{equation}%
Moreover, in this case $||W||_{h}^{2}=\frac{5}{4}(W^{1})^{2}+(W^{2})^{2}$
and by applying the general theory presented in the previous sections, the
wind force shoud be restricted to $||W||_{h}<\tilde{b}_{0},$ with $%
\tilde{b}_{0}$ defined in \eqref{Strong_C}. This refers to the strong
convexity condition which corresponds to the superwind metric $\tilde{F}%
_{\eta }$ with \eqref{ALL}, for any $C^{\infty }$-function $\eta
:M\rightarrow \lbrack 0,1].$ In view of \eqref{EE1}, it follows that $\tilde{%
y}^{1}=\frac{1}{||W||_{h}} \left( \frac{5}{4} W^{1}y^{1}+W^{2}y^{2}\right) $ and $\tilde{y}^{2}=\frac{5}{4||W||_{h}}%
\left( -W^{2}y^{1}+W^{1}y^{2}\right) $ and by using \eqref{2_indicatrix}, we
deduce the following parametric equations for the indicatrix of $\tilde{F}%
_{\eta }$ 
\begin{equation}
\left\{ 
\begin{array}{rll}
y^{1} & = & -\frac{2}{\sqrt{5}}[1-\eta (\frac{\sqrt{5}}{2}W^{1}\cos \theta
+W^{2}\sin \theta )]\cos \theta +(1-\eta )W^{1} \\ 
~ &  &  \\ 
y^{2} & = & -[1-\eta (\frac{\sqrt{5}}{2}W^{1}\cos \theta +W^{2}\sin \theta
)]\sin \theta +(1-\eta )W^{2}%
\end{array}%
,\right.  \label{INCL_IND}
\end{equation}%
for any clockwise direction $\theta \in \lbrack 0,2\pi )$ of the velocity $u$
with respect to $\frac{\partial }{\partial x^{1}}.$ In particular, if the
wind $W$ is the gravitational wind on $\mathfrak{S}$, i.e., 
$W=\mathbf{G}^{T}=-\frac{2\bar{g}}{5}\frac{\partial }{\partial x^{1}}$, with $%
||\mathbf{G}^{T}||_{h}=\frac{\bar{g}}{\sqrt{5}},$ \eqref{2_indicatrix} are
reduced to \cite[(72), with $\eta=0$]{general}, where $\bar{g}$ denotes the rescaled magnitude of the acceleration of gravity. 

Now, we are in position to write the $\tilde{F}_{\eta}$-geodesic equations which correspond to $\mathfrak{S}$. 
In view of \cref{Thm2} the time-minimal paths  $\gamma (t)=(x^{1}(t),x^{2}(t))$ on $\mathfrak{S}$ are provided by the
solutions of the ODE system%
\begin{eqnarray}
0 &=&\ddot{x}^{i}+2\alpha \tilde{Q}s_{0}^{i}-2\alpha ^{2}[\tilde{R}%
(r^{i}+s^{i})+\frac{1}{2}\tilde{P}\eta ^{i}]  \label{GGG_inclined} \\
&&+2\{\mathit{\tilde{\Theta}}[-2\alpha \tilde{Q}s_{0}+r_{00}+\alpha ^{2}(2%
\tilde{R}r+\tilde{P}q)]+\alpha \lbrack \mathit{\tilde{\Omega}}(r_{0}+s_{0})+%
\frac{1}{2}\mathit{\tilde{\Lambda}}\eta _{0}]\}\frac{\dot{x}^{i}}{\alpha } 
\notag \\
&&-2\{\mathit{\tilde{\Psi}}[-2\alpha \tilde{Q}s_{0}+r_{00}+\alpha ^{2}(2%
\tilde{R}r+\tilde{P}q)]+\alpha \lbrack \mathit{\tilde{\Pi}}%
r_{0}(r_{0}+s_{0})+\frac{1}{2}\mathit{\tilde{\Upsilon}}\eta _{0}]\}W^{i}, 
\notag
\end{eqnarray}%
$i=1,2,$ where all terms $\tilde{R}$, $\tilde{Q}$, $\tilde{P},$ $\mathit{%
\tilde{\Theta}}$, $\mathit{\tilde{\Omega}}$, $\mathit{\tilde{\Pi}}$, $%
\mathit{\tilde{\Psi}}$, $\mathit{\tilde{\Lambda}}$, $\mathit{\tilde{\Upsilon}%
}$, etc. are given by \eqref{geo_tilde} and everywhere in %
\eqref{GGG_inclined}, $x^{1}=x^{1}(t),$ $x^{2}=x^{2}(t)$, $y^{1}=\dot{x}%
^{1}(t)$ and $y^{2}=\dot{x}^{2}(t).$

To ease notation, we henceforth write $x$ for $x^1$ and $y$ for $x^2$ on further reading. This is consistent with the symbols applied in the attached figures.

\subsection{Action of a superwind on a uniform slippery slope
} 
 

\label{sec_ex1}

First, we show the influence of a superwind of varying direction and norm on the Finslerian geodesics on a slippery slope $\mathfrak{S}$, where the traction coefficient is constant, $\eta\in [0, 1]$; cf. \cite{slippery}. 
In order to cover 
 all types of navigation in the full range of $\eta$ the strong convexity condition \eqref{Strong_C} 
 requires $||W||_h<\frac12$, which refers to the most restrictive case, i.e., MAT ($\eta=1$). 
 Therefore, let us consider a superwind $\mathcal{W}_{\eta }$ with the related ordinary wind $W$ 
 given by
\begin{equation}
\label{exeq_superwind}
\mathcal{W}_{\eta=0}(x,y)=\frac{2}{5}\cos x\frac{\partial }{\partial x}+\frac{1}{5}\cos y
\frac{\partial }{\partial y}
\end{equation} 
 and visualized in \cref{fig_ex1a}, where $||W||_h=\frac{%
1}{5}\sqrt{5\cos ^{2}x+\cos ^{2}y}
$. Indeed, we have $||W||_h\leqslant\frac{\sqrt{6}}{5}\approx0.4899$, for any $(x, y)\in \mathfrak{S}$ and the condition is checked for the whole slope $\mathfrak{S}$. Consequently, by using \eqref{INCL_IND}, the transformation of the Riemannian indicatrix represented by the elliptical $h$-circle and centered at $(0, 0)$  with respect to $\eta$ are shown in \cref{fig_ex1b}. In general, they include the anisotropic deformation and rigid translation, which yield the Pascal's Snails. Note both the analogy with and the difference from the impact 
 of the active gravitational wind $\mathbf{G}_\eta=\eta \mathbf{G}_{MAT}+(1-\eta
)\mathbf{G}^{T}$ as in \cite{slippery}, which becomes a particular superwind in the current study, where $\mathbf{G}_{\eta=0}=\mathbf{G}^{T}$ blows, in contrast to $W,$ in a fixed direction on the entire  $\mathfrak{S}$. 
  It is obvious that $\mathcal{W}_{\eta=0}(0,0)=\frac{2}{5}\frac{\partial }{\partial x}+\frac{1}{5}
\frac{\partial }{\partial y}$, while $\mathbf{G}_{\eta=0}(0,0)=-\frac{2\bar{g}}{5}
\frac{\partial }{\partial x}=-0.446\frac{\partial }{\partial x}$, $\bar{g}=1.115<\frac{\sqrt{5}}{2}\approx 1.118$, which checks $||\mathbf{G}^{T}||_h<\frac12$; for clarity, see \cref{fig_ex1b} (right). 
For comparison, one can glance at the comprehensive analysis of time-minimizing navigation on the slippery $\eta$-slope under gravity which is presented in \cite{slippery, general}. Note that because of the arbitrary direction and norm (mild) of an ordinary wind $W$ which are now admitted, the notion of superwind allows us now to generalize the standard Zermelo navigation considerably.


\begin{figure}[h!]
\centering
~\includegraphics[width=0.487\textwidth]{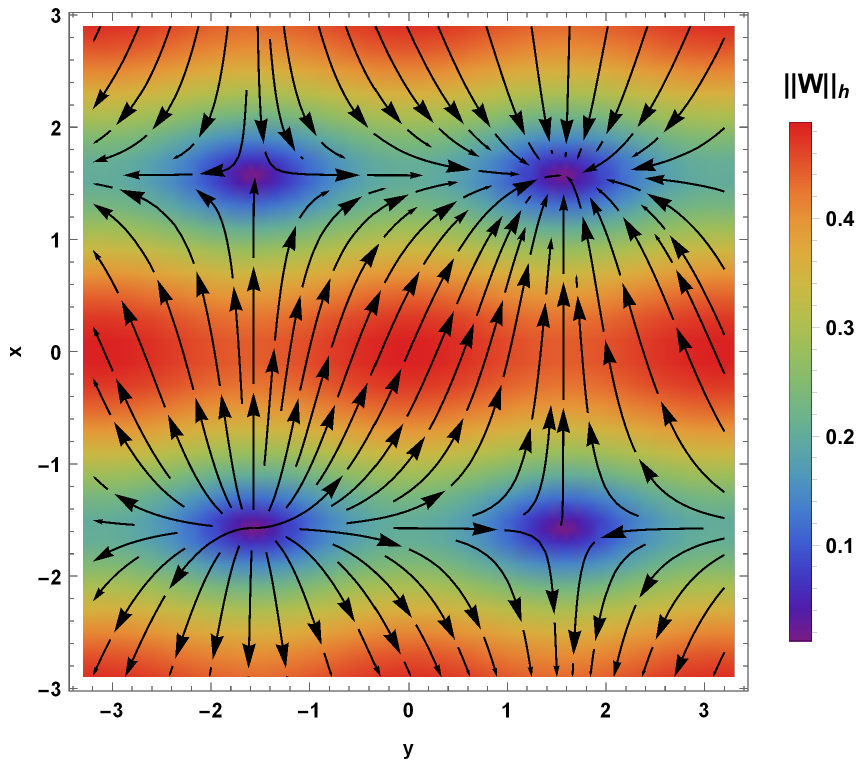}%
~\includegraphics[width=0.487\textwidth]{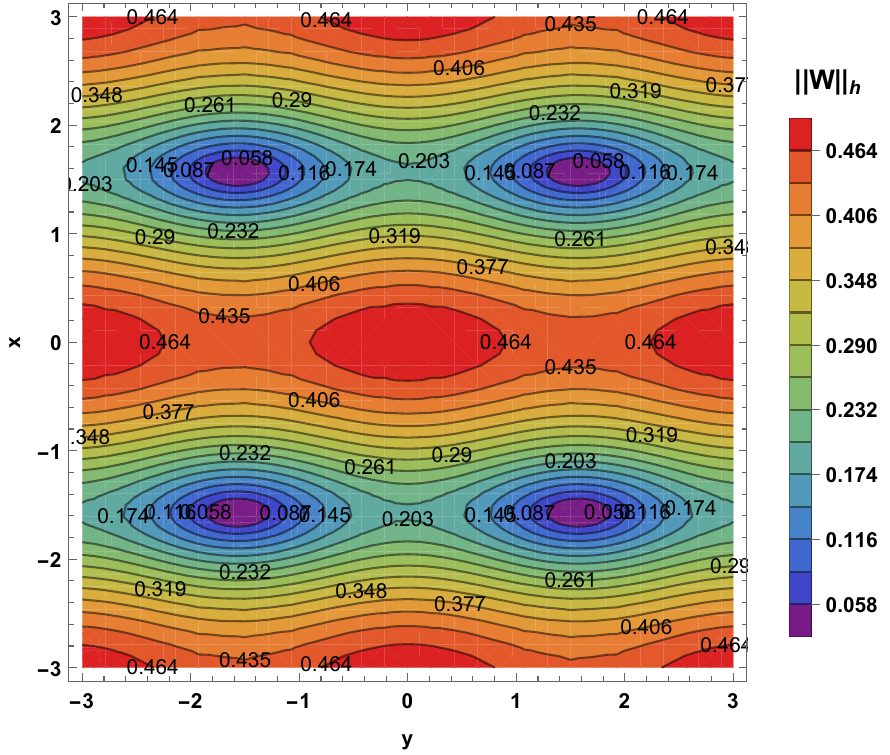}
\caption{Left: The stream density plot of the ordinary wind $W(x,y)$ given by \eqref{exeq_superwind}, i.e., the superwind in the Zermelo case that acts on the inclined plane $\mathfrak{S}$. The color-coded background represents the wind ``force'', satisfying the strong convexity condition for all $(x, y)\in\mathfrak{S}$, i.e., $||W||_h\leqslant\frac{\sqrt{6}}{5}\approx0.4899<\frac12$. Right: The contour plot of $\mathcal{W}_{\eta=0}=W$, as on the left.}
\label{fig_ex1a}
\end{figure}

\begin{figure}[h!]
\centering
~\includegraphics[width=0.489\textwidth]{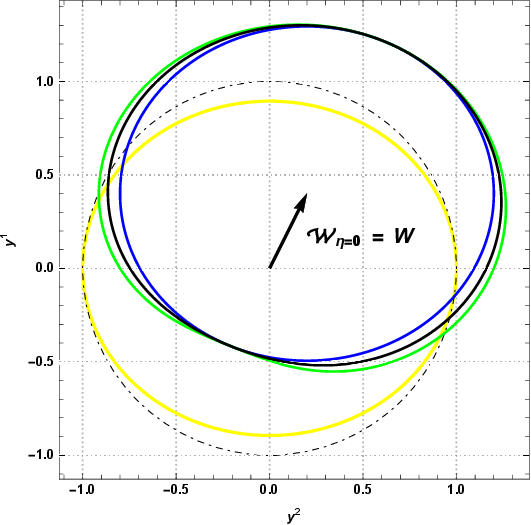} 
~\includegraphics[width=0.489\textwidth]{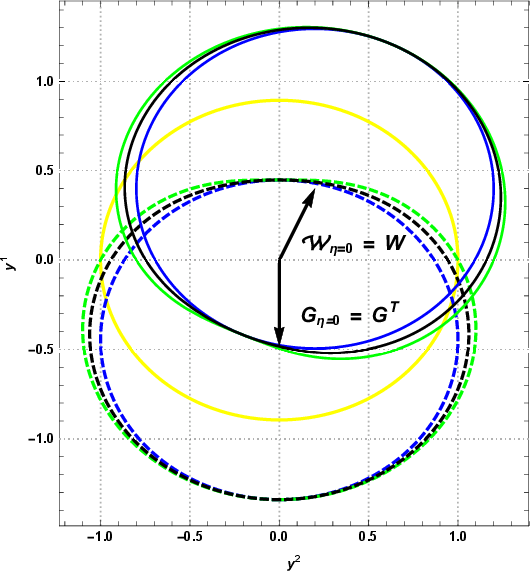} 
\caption{Left: The Finslerian indicatrices centered at $O=(0, 0)$ in the tangent plane  $y^1Oy^2$ under the action of the superwind $\mathcal{W}_{{\eta}}$ (solid), where the space-dependent wind $\mathcal{W}_{\eta=0}=W(x,y)$ is given by \eqref{exeq_superwind}, in the presence of constant tractions, i.e.,  $\eta\in\{0 \ \textnormal{(blue, ZNP)}, \frac34 \ \textnormal{(black}), 1 \ \textnormal{(green, MAT)}\}$.  For comparison, the initial Riemannian indicatrix being the elliptical $h$-circle (yellow) in the absence of superwind and the common Euclidean circle (dashed grey). Right: As on the left and compared in addition with the Finslerian indicatrices deformed under a different superwind being the gravitational wind $\mathbf{G}_\eta$ (dashed), where  $\bar{g}=1.115<\frac{\sqrt{5}}{2}\approx 1.118$, satisfying the condition for strong convexity $||\mathbf{G}^{T}||_h<\frac12$. In contrast to $W$,  $\mathbf{G}_{\eta=0}=\mathbf{G}^{T}$ blows in a fixed  direction (steepest descent) and has a constant norm on $\mathfrak{S}$, i.e.,  $||\mathbf{G}^{T}||_h=0.49$.}
\label{fig_ex1b}
\end{figure}

Taking into account the system \eqref{GGG_inclined}, the scenario in a three-dimensional view of $\mathfrak{S}$  is presented in \cref{fig_ex1c} (left) including the time-minimizing geodesics for $\eta=\frac34$ and the corresponding isochrone under the superwind $\mathcal{W}_{\eta}$ which lies entirely between the classical Matsumoto and Zermelo cases, where $t=5$. The behaviour of the current time front in the presence of a different perturbation remains in line with the outcomes obtained in the previous study in the presence of gravity (\cite{slippery}). Moreover, the evolution of the isochrones in time ($t\in\{3, 5, 7\}$) is shown in \cref{fig_ex1c} (right), where, for the same time instants, the Matsumoto front is outermost and the Zermelo front is innermost. For the sake of clarity, the behaviours of the Finslerian geodesics ($t=3$) are also compared in \cref{fig_ex51}, emphasizing the action of the superwind in the backgrounds.

\begin{figure}[H]
\centering
~\includegraphics[width=0.53\textwidth]{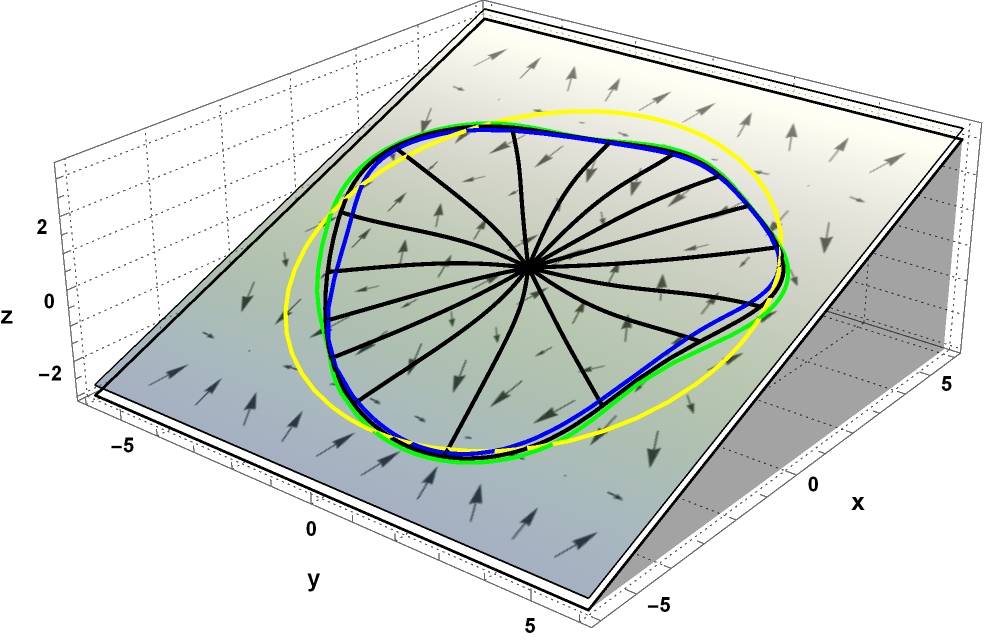}
~\includegraphics[width=0.44\textwidth]{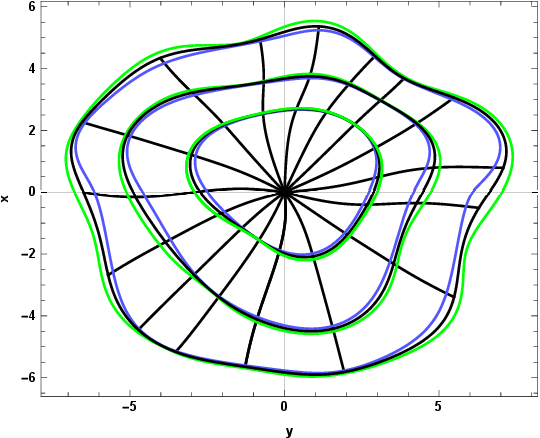} 
\caption{Left: The time-minimizing geodesics and time front centered at $(0, 0)$ under the action of the superwind $\mathcal{W}_{{\eta}}$ and constant traction $\eta=\frac34$ (black) on the inclined plane $\mathfrak{S}$, compared with the edge cases: $\eta=0$ (blue, ZNP) and $\eta=1$ (green, MAT) as well as the Riemannian isochrone (yellow), where the related ordianry wind (in the Zermelo sense) reads  \eqref{exeq_superwind} (its action is indicated by black arrows); $t=5$. Right: The behaviour of the time-minimizing geodesics ($\eta=\frac34$) and evolution of the corresponding time fronts for $t\in\{3\ \textnormal{(innermost)}, 5\ \textnormal{(middle)}, 7\ \textnormal{(outermost)}\}$ (color-coded as on the left). The geodesics are drawn with a step of $\Delta\protect\theta=\protect\frac{\pi}{8}$, i.e., 16 time-minimizing paths in both subfigures.}
\label{fig_ex1c}
\end{figure}
\begin{figure}[H]
\centering
~\includegraphics[width=0.46\textwidth]{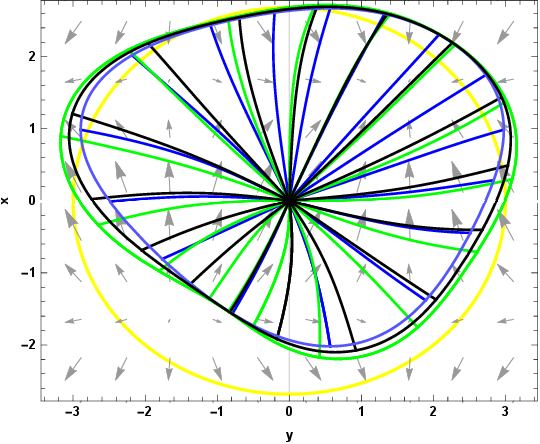} 
~\includegraphics[width=0.48\textwidth]{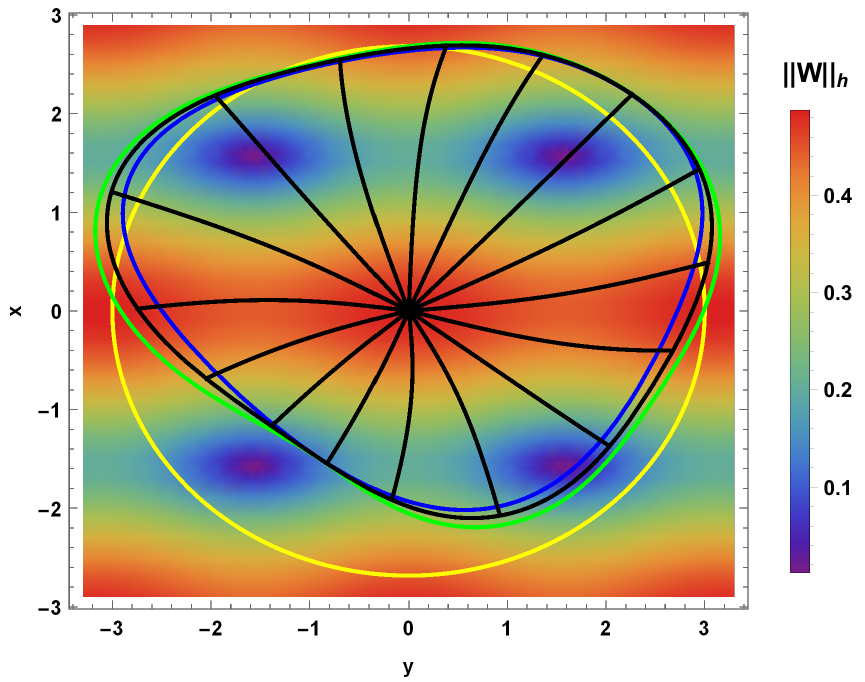} \\~\\
\caption{Left: The time-minimizing geodesics and time fronts centered at $(0, 0)$ under the action of the superwind related to \eqref{exeq_superwind} (grey arrows), for different types of $\eta$-navigation on $\mathfrak{S}$, i.e., $\eta\in\{0 \ \textnormal{(ZNP, blue)}, \frac34 \ \textnormal{(black}), 1 \ \textnormal{(MAT, green)}\}$; $t=3$. The geodesics are presented with a step of $\Delta\protect\theta=\protect\frac{\pi}{8}$ and the unperturbed Riemannian isochrone is marked in yellow. 
Right: As on the left, with the color-coded background that represents the space-dependent wind norm $||W||_h$; $t=3$.}
\label{fig_ex51}
\end{figure}

\begin{figure}[h!]
\centering
~\includegraphics[width=0.482\textwidth]{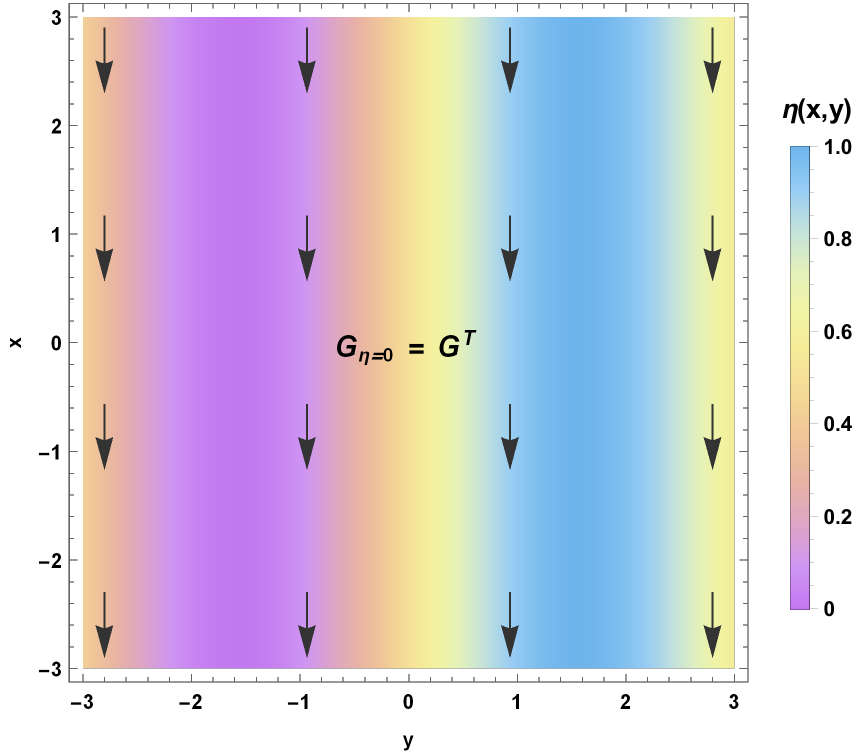}
~\includegraphics[width=0.49\textwidth]{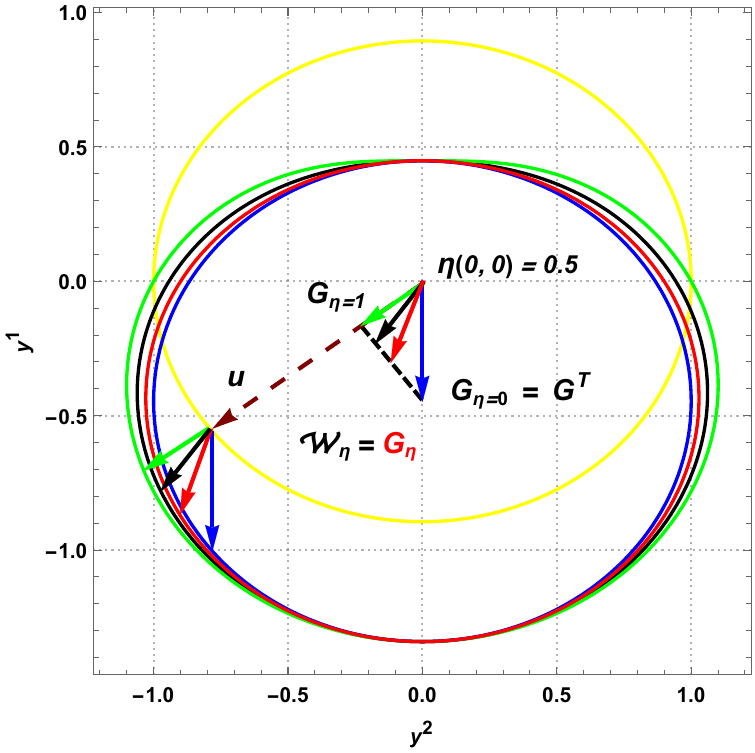} 
\caption{Left: The deforming factors on the inclined plane $\mathfrak{S}$: the gravitational wind $\mathbf{G}_{{\eta}}$ considered as a particular superwind, where $W=\mathbf{G}_{{\eta=0}}=\mathbf{G}^{T}$ blows in the steepest downhill direction (black arrows) of the slope and the space-dependent traction coefficient is given by \eqref{eta_1} (the color-coded  background); $\bar{g}=1.115<\frac{\sqrt{5}}{2}\approx1.118$. 
Right: A comparison of the Finslerian indicatrices (lima\c{c}ons) in the tangent plane $y^1Oy^2$, centered at $O=(0, 0)$, under $\mathbf{G}_{{\eta}}$ and varying $\eta(x,y)$ (red), $\eta=1$ (MAT, green),  $\eta=0$ (ZNP, blue), $\eta=\frac34$ (black) and the Riemannian case (yellow); $t=1$. Note the $y^1$-axial symmetry of the indicatrix referring to the varying traction (red) 
 which coincides herein with the indicatrix in the presence of constant traction on the entire slope, i.e., $\eta=\frac12$, since $\eta(0, 0)=
\frac12$. The craft's own velocity is denoted by $u$ (dashed brown), with 
 $||u||_h=1$.}
\label{fig_ex2}
\end{figure}

\subsection{Influence of a non-uniform slippery slope under a gravitational wind 
} 


\label{sec_ex2}

In contrast to the example from the preceding subsection and investigations in \cite{slippery, cross, slipperyx, general}, the traction coefficient is now space-dependent, i.e., 
\begin{equation} 
\eta(x,y)=\frac12(1+\sin y) \in [0, 1],   
\label{eta_1}
\end{equation}
 for any $(x, y)\in \mathfrak{S}$.
Moreover, the superwind $\mathcal{W}_{\eta }$ is represented by the active gravitational wind $\mathbf{G}_{\eta}$, where  the related ordinary wind (in the Zermelo sense) reads  $W=\mathbf{G}_{{\eta=0}}=\mathbf{G}^{T}=-\frac{2\bar{g}}{5}\frac{\partial }{\partial x}$  and blows in the steepest downhill direction on the inclined plane $\mathfrak{S}$; see \cref{fig_ex2} (left). 
The relations between the Finslerian indicatrices (lima\c{c}ons) centered at the origin, where $\eta(0,0)=\frac12$ are shown in \cref{fig_ex2} (right). The axial symmetry of the indicatrix referring to the varying traction should be noted, unlike the corresponding time front shown in \cref{fig_ex2a} (left). 
 It is worth pointing out that the time-minimizing geodesics with the varying traction that come from \eqref{GGG_inclined} (red in \cref{fig_ex2a} (bottom left)) are not in fact the Euclidean straight lines, although they may look so 
 in the attached graph. In general, this is in contrast to the geodesics coming from the constant traction scenarios, e.g., the Randers (ZNP) and Matsumoto (MAT) paths which are straight. For clarity, the changes of the ratio $\frac{x'(t)}{y'(t)}$ in time for two exemplary individuals, i.e., with the inital direction $\theta_0\in\{106^\circ, 130^\circ\}$ measured clockwise from the axis $-x$ show clearly that the directions of the tangent vectors to the geodesics vary (\cref{fig_ex2a} (bottom right)). Moreover, the geodesics and time fronts for $\eta\in\{\frac12, \frac34\}$ are compared in addition in \cref{fig_ex2a} (bottom left). 

Finally, all the analyzed cases in the presence of the superwind $\mathbf{G}_{\eta}$ and varying traction \eqref{eta_1} are juxtaposed clearly on the planar map in \cref{fig_ex2b}. The impact of the varying $\eta$ is visible, though not substantial, since the related (red) isochrone lies entirely between the boundary  fronts (ZNP and MAT) which are not far away from each other in this example, concerning the weak wind. The fact that the isochrone lies between the Randers and Matsumoto fronts was actually expected on the basis of the preceding analysis on the slippery slope models under a gravitational wind in \cite{slippery, general}. Furthermore, observe the lack of axial symmetry of the isochrone (red) w.r.t. to $x$ should be observed, unlike all the other cases that refer to the constant values of $\eta$.

\begin{figure}[H]
\centering
~\includegraphics[width=0.435\textwidth]{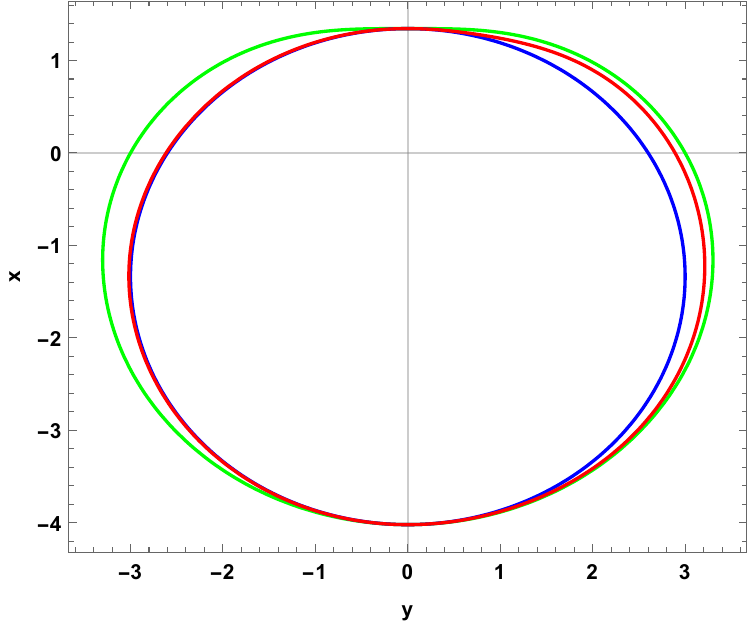}
~\includegraphics[width=0.54\textwidth]{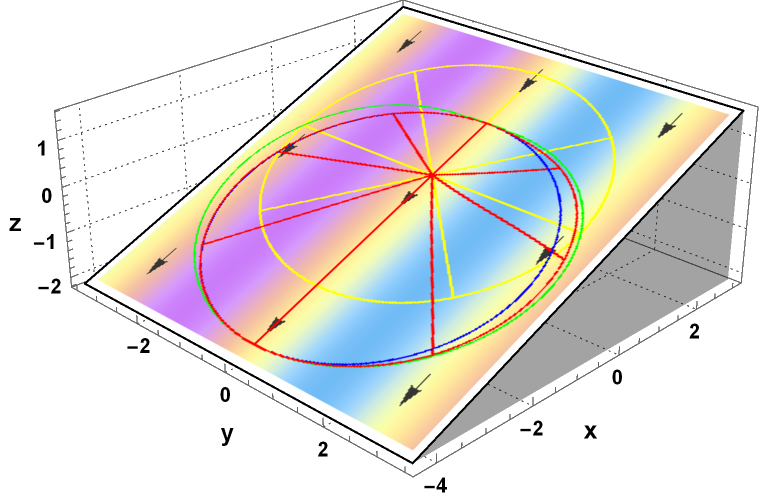} \\~\\
~\includegraphics[width=0.489\textwidth]{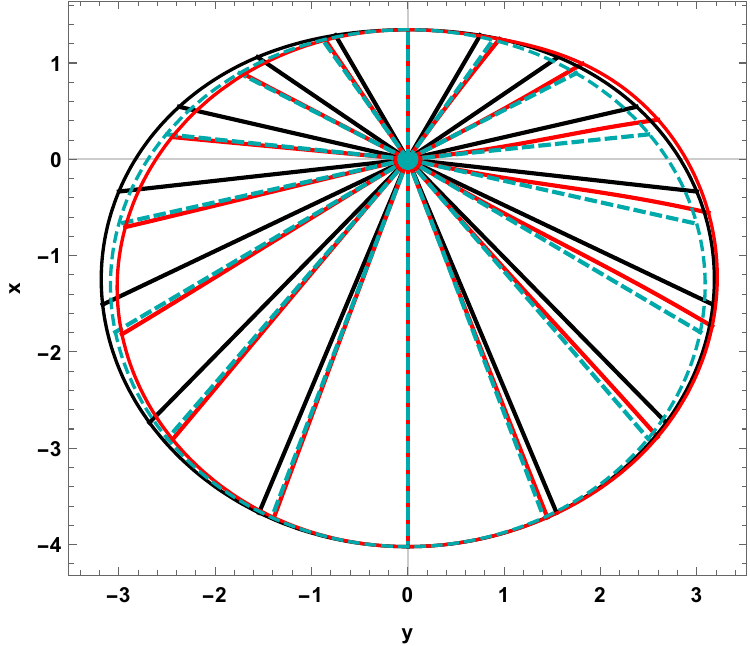}
~\includegraphics[width=0.489\textwidth]{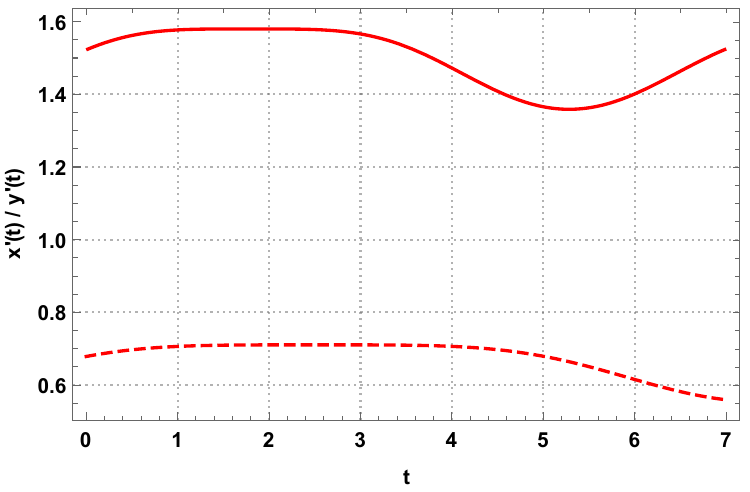}
\caption{Top left: The time front centered at $(0, 0)$ under the action of the gravitational wind $\mathbf{G}_{{\eta}}$ considered as a particular superwind and the space-dependent traction coefficient \eqref{eta_1} (red; note its lack of $x$-symmetry), compared to the edge cases, i.e.,  $\eta=0$ (blue, ZNP) and $\eta=1$ (green, MAT); $t=3$. 
Bottom left: As above (red) and compared with the cases $\eta=\frac34$ (black) and $\eta=\frac12$ (dashed), together with the respective geodesics drawn with the step of $\Delta\protect\theta=\protect\frac{\pi}{8}$, i.e., 16 time-minimizing paths. Note that the time-minimizing geodesics for the varying $\eta$ (red) are not the Euclidean straight lines, although they look 
 to be 
 so due to the relatively slight changes. In contrast, the Finslerian geodesics for both $\eta=\frac34$ (black) and $\eta=\frac12$ (dashed lines) are the Euclidean straight lines; $t=3$. 
Bottom right: The change of the ratio $\frac{x'(t)}{y'(t)}$ in time, showing the directions of the tangent vectors to the geodesics related to the space-dependent traction vary, where the exemplary initial directions $\theta(0)$ equal to $106^\circ$ (solid red) and $130^\circ$ (dashed red) are measured clockwise from $- x$; $t=7$. 
 Top right: The scenario presented on the inclined plane $\mathfrak{S}$ under the influence of the  superwind $\mathbf{G}_{{\eta}}$ and the varying traction coefficient \eqref{eta_1}, including the time fronts and time-minimizing geodesics (colour-coded as above) as well as the Riemannian geodesics (yellow; drawn with a step of $\Delta\protect\theta=\protect\frac{\pi}{4}$) and,  additionally, the time front; $t=3$.}
\label{fig_ex2a}
\end{figure}

\begin{figure}[H]
\centering
~\includegraphics[width=0.6\textwidth]{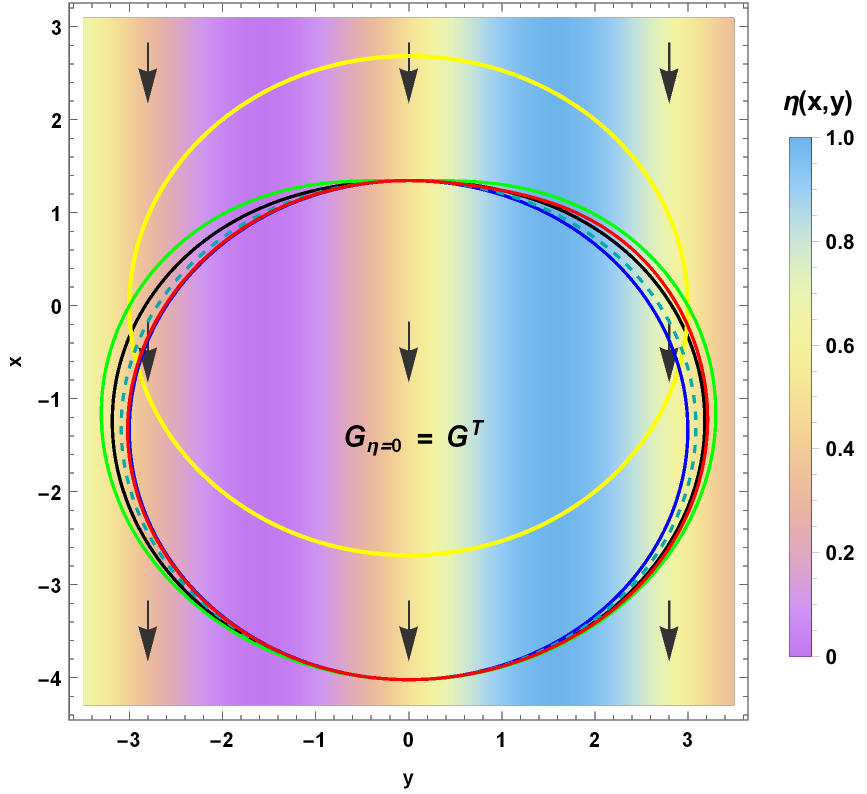}
\caption{A comparison of time fronts for $t=3$: with the space-dependent traction coefficient  \eqref{eta_1} (red), Riemannian (yellow), $\eta=0$ (ZNP, blue), $\eta=\frac12$ (dashed cyan), $\eta=\frac34$ (black), $\eta=1$ (MAT, green), under the influence of the gravitational wind $\mathbf{G}_{{\eta}}$ (black arrows) acting in the steepest downhill direction, which is considered as a particular superwind in the current study; $\bar{g}=1.115<\frac{\sqrt{5}}{2}\approx1.118$. The isochrone corresponding the varying $\eta$ (red) is not symmetric w.r.t. the axis $x$, unlike all other time fronts for $\eta=const.$}
\label{fig_ex2b}
\end{figure}

\subsection{
Action of a superwind on a non-uniform slippery slope
}

\label{sec_ex3}

%
%

The last example presents the combined effect of a superwind and space-dependent traction coefficient on the ramp $\mathfrak{S}$. The former perturbation is given by \eqref{exeq_superwind} like in \cref{sec_ex1}, while the latter now reads
\begin{equation}
\eta(x,y)=e^{-(x^2+y^2)} \in (0, 1],  
\label{eta_2}
\end{equation}
for any $(x, y)\in \mathfrak{S}$. The corresponding stream, density and contour plots are presented graphically in \cref{fig_ex3}; see also \cref{fig_ex1a}. Now one can observe the influence of varying $\eta$ on the behaviour of time-optimal paths and evolution of time fronts under the same superwind in comparison with the example, including the uniform slippery slope ($\eta=const.$) analyzed in \cref{sec_ex1}. Making use of \eqref{GGG_inclined}, the time-minimizing geodesics starting at $(0, 0)$ and related isochrones are presented for $t\in\{2, 5\}$ as well as compared with the boundary cases (Zermelo-Randers and Matsumoto) and the  undeformed Riemannian case in \cref{fig_ex3a} and \cref{fig_ex3b}. An analogous scenario, however with a different initial point at $(\frac34, \frac34)$, is shown in \cref{fig_ex3c}.  Note that the time front for the varying $\eta$ (red) comes closer and closer to the Zermelo one (blue) in the example under consideration, if the distance from the origin increases. Clearly, this is a direct  consequence of the formula for  the applied traction coefficient \eqref{eta_2}, which includes the exponential function; see also \cref{fig_ex3} (right) in this regard. Namely, the impact of the varying traction becomes minor, where the geodesics move away from the circular neighbourhood of the origin and almost coincide with the ZNP case, i.e., $\eta(x, y)\rightarrow0$.  Then, the deformation of the geodesics is caused 
mainly by the ordinary wind $W=\mathcal{W}_0$ defined in \eqref{exeq_superwind}. 

For completeness, the detailed juxtapositions of time-minimizing geodesics and respective time fronts, where $t\in\{3, 5, 7\}$ under the action of the superwind $\mathcal{W}_\eta$ with varying and fixed traction coefficients, as well as at two different initial points, are analyzed in \cref{fig_ex3d}; cf. \cite{slippery}. 
We remark that the difference between time fronts for the individual types of $\eta$-navigation will be more visible if one considers a stronger superwind, i.e., $||W||_h>\frac12$. In our exposition, however, we aimed to consider and compare $\eta$-navigation in the full range, and therefore, the strict condition for strong convexity yields $||W||_h<\frac12$. Furthermore, the obtained isochrones lie entirely between the Zermelo-Randers and Matsumoto cases, as expected; for comparison, see also the outcomes (\cref{fig_ex2b}) analyzed for different space-dependent traction \eqref{eta_1}.       

\begin{figure}[H]
\centering
~\includegraphics[width=0.489\textwidth]{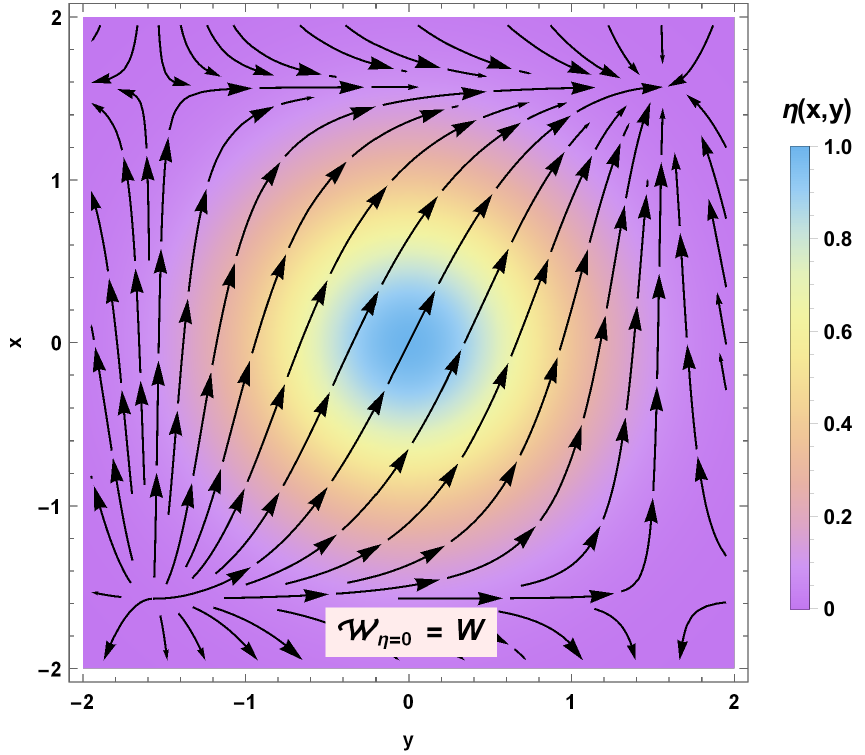}
~\includegraphics[width=0.489\textwidth]{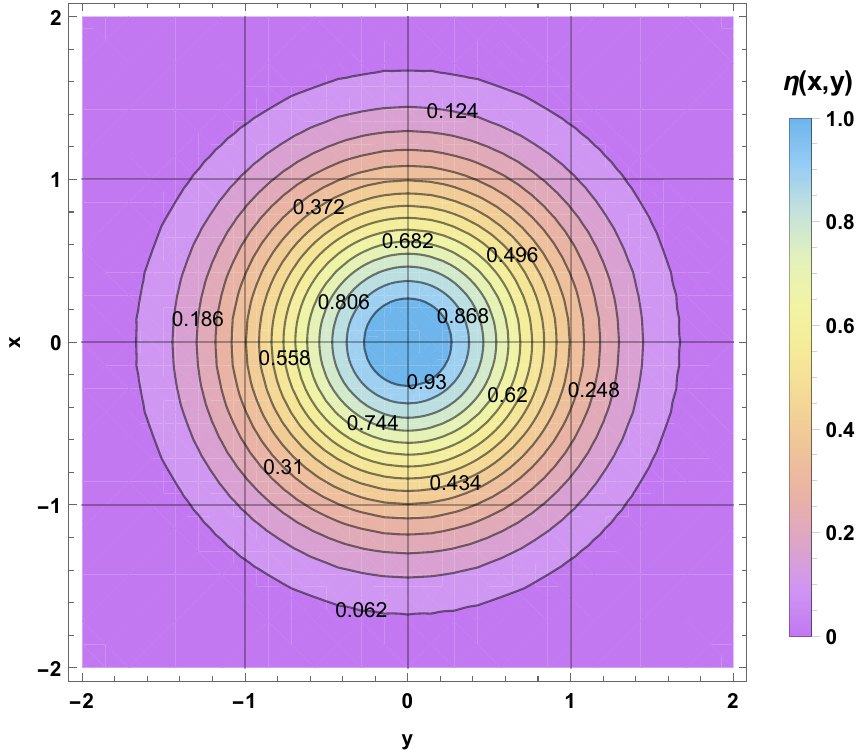}
\caption{Left: A stream plot of the ordinary wind in the Zermelo sense $W(x,y)=\mathcal{W}_{\eta=0}$ (black arrows) as in \eqref{exeq_superwind} and density plot of the space-dependent traction coefficient  \eqref{eta_2} (the color-coded background) on  $\mathfrak{S}$. Right: A contour plot of the varying traction coefficient $\eta(x,y)\in(0, 1]$ given by \eqref{eta_2}, as on the left.}
\label{fig_ex3}
\end{figure}

\begin{figure}[H]
\centering
~\includegraphics[width=0.54\textwidth]{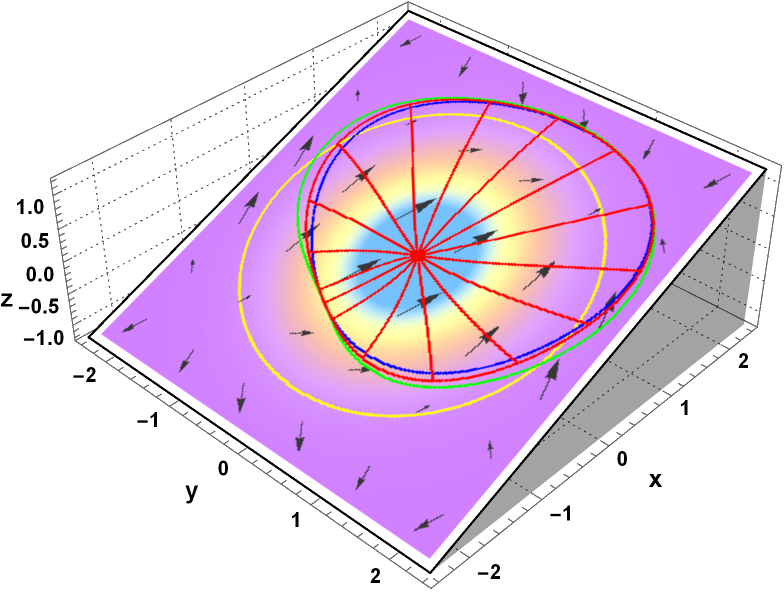}
~\includegraphics[width=0.45\textwidth]{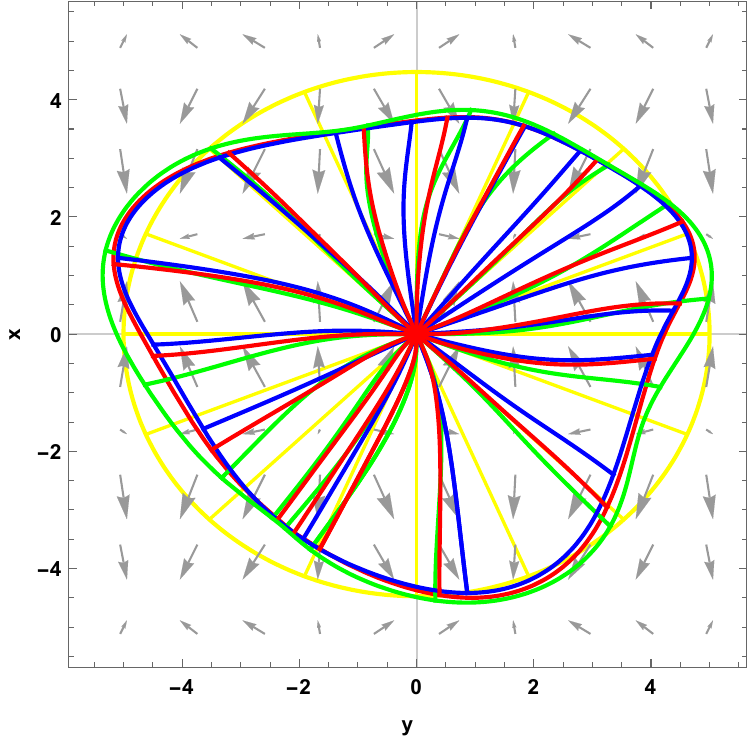}
\caption{Left: The time-minimizing geodesics and time fronts for the varying  traction \eqref{eta_2} (red) on $\mathfrak{S}$, together with the time fronts for $\eta=0$ (blue, ZNP), $\eta=1$ (green, MAT) and the unperturbed Riemannian case (yellow);  $t=2$. The initial points is positioned at $(0, 0)$. In the background, the action of the vector field \eqref{exeq_superwind} (black arrows) and space-dependent traction coefficient $\eta(x,y)$ (pastel colors). Right: The same scenario as on the left, presented on the planar map and with the respective geodesics in addition, drawn with a step of $\Delta\protect\theta=\protect\frac{\pi}{8}$; $t=5$.}
\label{fig_ex3a}
\end{figure}

\begin{figure}[H]
\centering
~\includegraphics[width=0.489\textwidth]{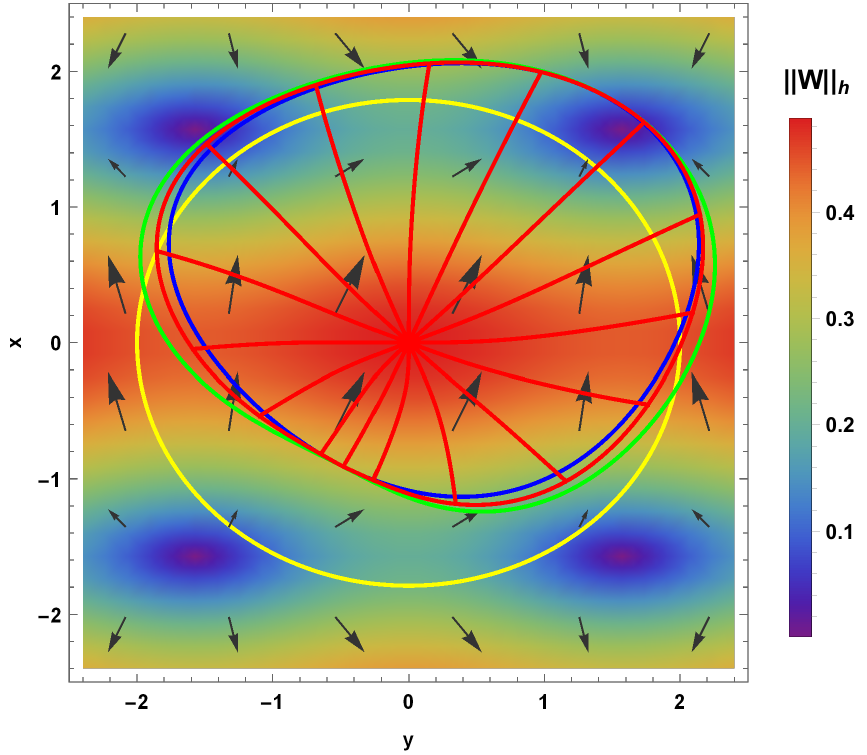}
~\includegraphics[width=0.489\textwidth]{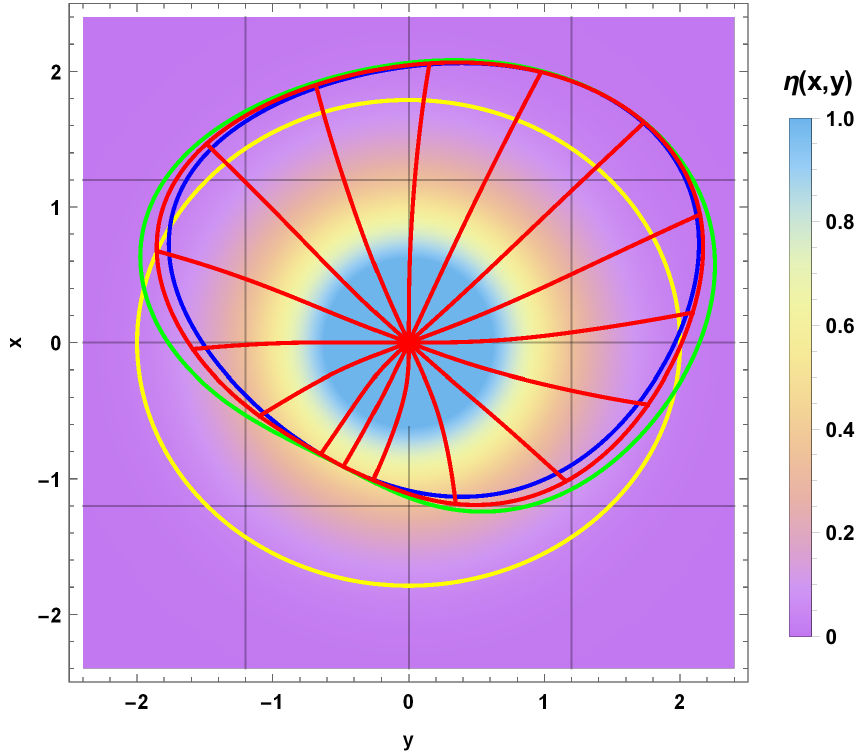}
\caption{The time fronts for: the varying traction coefficient \eqref{eta_2} 
  (red), $\eta=0$ (blue, ZNP), $\eta=1$ (green, MAT) and the unperturbed Riemannian case (yellow); $t=2$. The time-minimizing geodesics for $\eta(x,y)$ are drawn in red with a step of $\Delta\protect\theta=\protect\frac{\pi}{8}$. In the background, the color-coded wind ``force'' $||W||_h$ (left), where $W$ is given by  \eqref{exeq_superwind} (black arrows), and density plot of the space-dependent traction coefficient \eqref{eta_2} (right).}
\label{fig_ex3b}
\end{figure}

\begin{figure}[H]
\centering
~\includegraphics[width=0.53\textwidth]{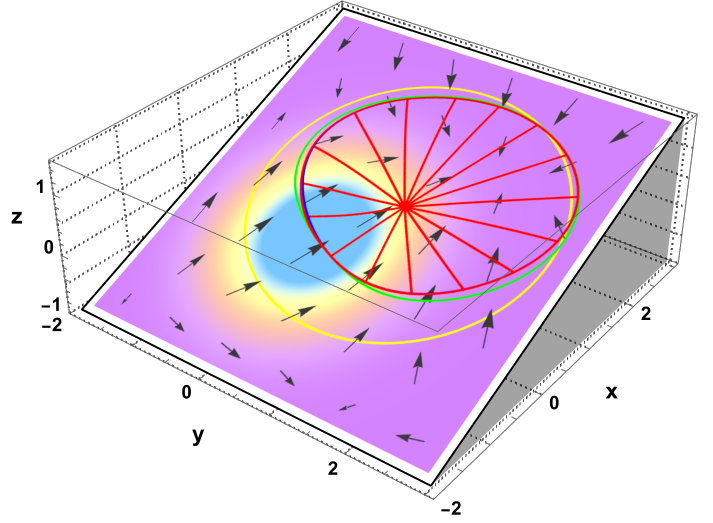}
~\includegraphics[width=0.449\textwidth]{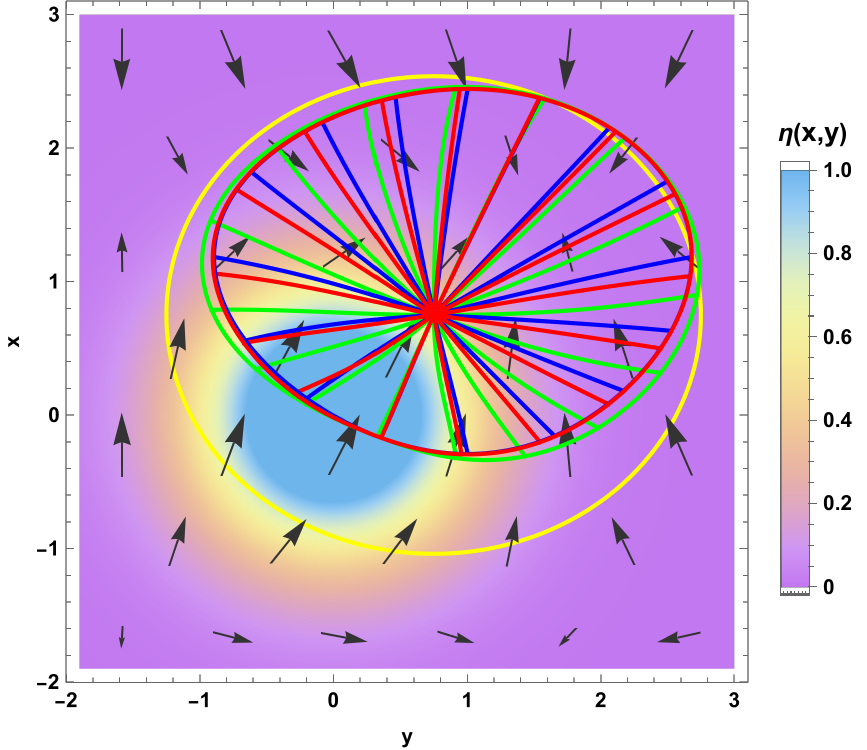}
\caption{The analogous scenario presented respectively as in \cref{fig_ex3a} (left) and \cref{fig_ex3b} (right), where the initial point is located at $(\frac34, \frac34)$, for comparison.}
\label{fig_ex3c}
\end{figure}

We remark that in the presented approach, the action of an arbitrary (weak) wind like in ZNP enables one to solve the optimal navigation problems in particular on the common (horizontal) Euclidean plane, while all scenarios under a gravitational wind simplify to the Riemannian case because the surface gradient then vanishes and hence, $\mathbf{G}^{T}=\overrightarrow{0}$ (see  \cite{slippery,cross,slipperyx,general}).  

\begin{figure}[H]
\centering
~\includegraphics[width=0.489\textwidth]{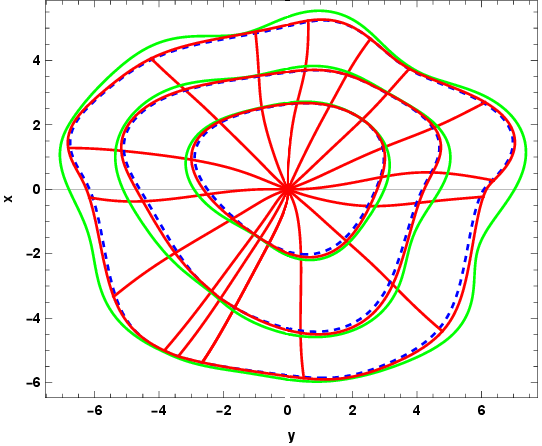}
~\includegraphics[width=0.489\textwidth]{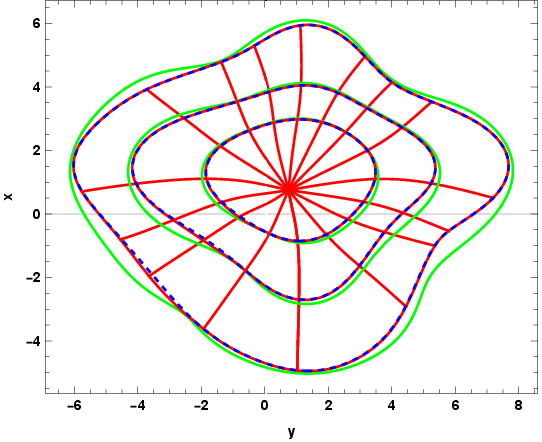}\\~\\
~\includegraphics[width=0.489\textwidth]{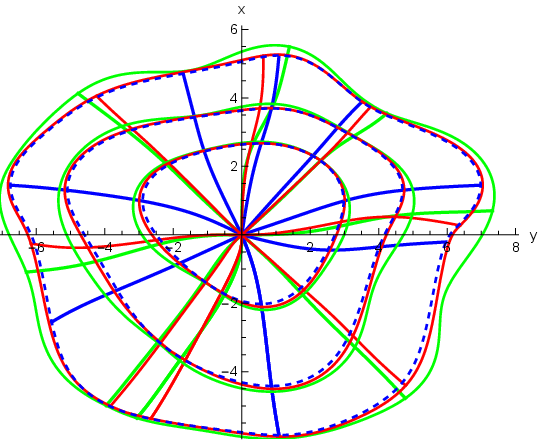}
~\includegraphics[width=0.489\textwidth]{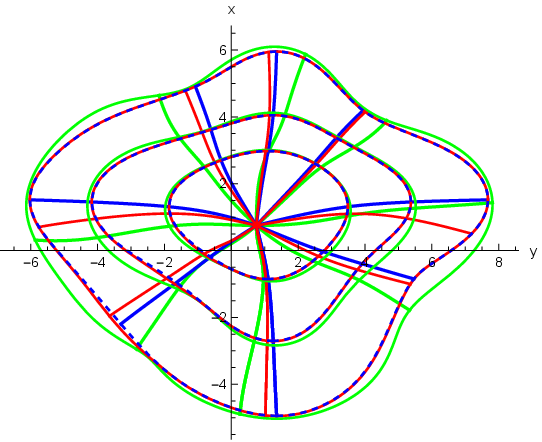}\\~\\
~\includegraphics[width=0.489\textwidth]{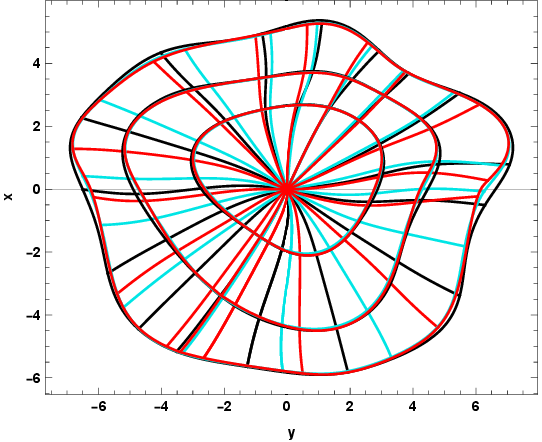}
~\includegraphics[width=0.489\textwidth]{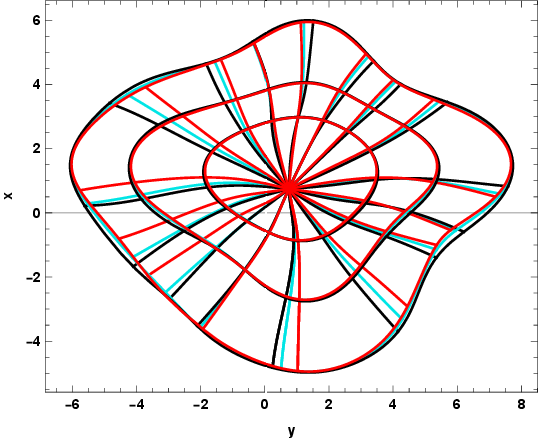}
\caption{The comparisons of the time-minimizing geodesics ($t=7$) and related time fronts ($t\in\{3 \  \textnormal{(innermost)}, 5 \ \textnormal{(middle)} ,7 \ \textnormal{(outermost)}\}$) under the influence of the superwind $\mathcal{W}_\eta$, where $\mathcal{W}_{\eta=0}=W$ is given by \eqref{exeq_superwind} and various tractions, i.e.,  $\eta(x,y)$ as in \eqref{eta_2} (red), $\eta\in\{0 \ \textnormal{(blue, ZNP)},  \frac12 \ \textnormal{(cyan)}, \frac34 \ \textnormal{(black)}, 1 \ \textnormal{(green, MAT)}\}$, where the initial point is located at $(0, 0)$ (left) and $(\frac34, \frac34)$ (right). The geodesics are drawn with a step of $\Delta\protect\theta=\protect\frac{\pi}{8}$ (top, bottom) and  $\Delta\protect\theta=\protect\frac{\pi}{4}$ (middle). For visual clarity, the time fronts in the Zermelo case ($\eta=0$) are shown as the blue dashed lines.} 
\label{fig_ex3d}
\end{figure}

%


%


\addcontentsline{toc}{section}{References} 
\bibliographystyle{plain}

\end{document}